\documentclass[a4paper,11pt]{article} 
 
\usepackage{multirow}     
\usepackage{mathtools}
\usepackage{subcaption}
\usepackage{amsfonts,amssymb,amsmath,amsthm,latexsym,bbm} 
\usepackage{enumerate}
\usepackage{hyperref} 
\usepackage{epsf,epsfig}
\usepackage{xcolor,colortbl,color}
\usepackage{graphicx,graphics} 
\usepackage{caption}  
\usepackage[utf8]{inputenc}
\usepackage{tikz}
\usepackage[english]{babel}
\usepackage{soul}%strikeout
\makeatletter \@addtoreset{equation}{section} \makeatother
\makeatletter \@addtoreset{enunciato}{section} \makeatother
\newcounter{enunciato}[section]

\newtheorem{ittheorem}{Theorem}
\newtheorem{itlemma}{Lemma}
\newtheorem{itproposition}{Proposition}
\newtheorem{itdefinition}{Definition}
\newtheorem{itremark}{Remark}
\newtheorem{itclaim}{Claim}
\newtheorem{itfact}{Fact}

\newtheorem{itconjecture}{Conjecture}
\newtheorem{itcorollary}{Corollary}

\newenvironment{theorem}{\addtocounter{enunciato}{1}
\begin{ittheorem}}{\end{ittheorem}}
\newenvironment{lemma}{\addtocounter{enunciato}{1}
\begin{itlemma}}{\end{itlemma}}
\newenvironment{proposition}{\addtocounter{enunciato}{1}
\begin{itproposition}}{\end{itproposition}}
\newenvironment{definition}{\addtocounter{enunciato}{1}
\begin{itdefinition}}{\end{itdefinition}}
\newenvironment{remark}{\addtocounter{enunciato}{1}
\begin{itremark}}{\end{itremark}}

\newenvironment{conjecture}{\addtocounter{enunciato}{1}
\begin{itconjecture}}{\end{itconjecture}}
\newenvironment{corollary}{\addtocounter{enunciato}{1}
\begin{itcorollary}}{\end{itcorollary}}

\newcommand{\tod}{\rightarrow^{(d)}}
\newcommand{\todp}{\rightarrow^{L^p}}
\newcommand{\todtwo}{\rightarrow^{L^2}}
\newcommand{\todd}[1]{\xrightarrow[#1 \to \infty]{(d)}}
\newcommand{\be}[1]{\begin{equation}\label{#1}}
\newcommand{\ee}{\end{equation}}
\newcommand{\bl}[1]{\begin{lemma}\label{#1}}
\newcommand{\el}{\end{lemma}}
\newcommand{\br}[1]{\begin{remark}\label{#1}}
\newcommand{\er}{\end{remark}}
\newcommand{\bt}[1]{\begin{theorem}\label{#1}}
\newcommand{\et}{\end{theorem}}
\newcommand{\bd}[1]{\begin{definition}\label{#1}}
\newcommand{\ed}{\end{definition}}
\newcommand{\bp}[1]{\begin{proposition}\label{#1}}
\newcommand{\ep}{\end{proposition}}
\newcommand{\bc}[1]{\begin{corollary}\label{#1}}

\newcommand{\bcj}[1]{\begin{conjecture}\label{#1}}
\newcommand{\ecj}{\end{conjecture}}

\def\<{\langle}
\def\>{\rangle}
\def\chv#1{\{\,#1\,\}}

\def\abs#1{\left\vert #1 \right\vert} %% absolute value %%
\def\norm#1{\left\Vert #1 \right\Vert} %% norm %%
\def\prt#1{\left( #1 \right)} %% round bracket %%
\def\crt#1{\left[ #1 \right]} %% square bracket %%
\def\Ind#1{\mathbbm{1}_{#1}}  %% indicator function %%
\def\eatspace#1{\relax}
\def\unskipit{\expandafter\eatspace}

\newcommand{\gep}{\varepsilon}

\newcommand{\mc}[1]{{\mathcal #1}}

\newcommand{\bb}[1]{{\mathbb #1}}

\newcommand{\R}{\ensuremath{\mathbb{R}}}
\newcommand{\var}{\ensuremath{\mathbb{V}\mathrm{ar}}}
\newcommand{\Z}{\ensuremath{\mathbb{Z}}}
\newcommand{\N}{\ensuremath{\mathbb{N}}}

\newcommand{\dd}{\ensuremath{\mathrm{d}}}

\def\crt#1{\left[ #1\right]} %square brackets
\def\prt#1{\left( #1\right)} %round brackets
\def\chv#1{\{\,#1\,\}} %Curly brackets curly braces
\def\Ind#1{ \mathbbm{1}_{#1}} %Indicator function
\def\abs#1{\left\vert #1\right\vert}
\def\norm#1{\left\Vert #1\right\Vert}

\newcommand{\Proof}[1]{\paragraph{\sc #1}}
\newcommand{\QED}{\hspace*{\fill}{\large$\Box$}\smallskip}

\begin{document}   

%%%%%%%%%% Title, authors %%%%%%%%%%%%%%%
\title{Random walk in cooling random environment:\\
recurrence versus transience and mixed fluctuations}

\author{\renewcommand{\thefootnote}{\arabic{footnote}}
Luca Avena,\,\,
Yuki Chino,\,\,
Conrado da Costa,\,\,
Frank den Hollander\,\footnotemark[1]}

\date{\today}

\footnotetext[1]{Mathematical Institute, Leiden University,
P.O.\ Box 9512, 2300 RA Leiden, The Netherlands}

%%%%%%%%%%%%%%%%%%%%%%%%%%%%%%%%%%%%%%%%

\maketitle

%%%%%%%%%% Abstract %%%%%%%%%%%%%%%%%%%%

\begin{abstract}
  \,  This is the third in a series of papers in which we consider one-dimensional Random Walk in Cooling Random Environment (RWCRE). The latter is obtained by starting from one-dimensional Random Walk in Random Environment (RWRE) and resampling the environment along a sequence of deterministic times, called \emph{refreshing times}.
In the present paper we explore two questions for general refreshing times. First, we investigate how the recurrence versus transience criterion known for RWRE changes for RWCRE. Second, we explore the fluctuations for RWCRE when RWRE is either recurrent or satisfies a classical central limit theorem. We show that the answer depends in a delicate way on the choice of the refreshing times.
An overarching  goal of our paper is to investigate how the behaviour of a random process with a rich correlation structure can be affected by resettings. 

\\
{\it Acknowledgment:} The research in this paper was supported through NWO Gravitation Grant NETWORKS-024.002.003. We are grateful to Zhan Shi for suggesting the line of proof of a fluctuation bound for recurrent RWRE that is needed in our analysis of RWCRE. 
\end{abstract}

\newpage

\tableofcontents

%%%%%%%%%% SECTION 1 %%%%%%%%%%%%%%%%%%%%%%%%%%%%%%%%%%%%%%%
\section{Introduction, main results and discussion} \label{Intro} 

%%%%%%%%%%%%%%%%%%%%%%%%%%%%%%%%%%%%%%%%%%%%%%%%%%%%%%%%
\subsection{Background and outline} \label{Back}

Random Walk in Random Environment (RWRE) is a classical model for a particle moving in a non-homogeneous medium, consisting of a random walk with random transition probabilities, sampled at time zero from a given law. \emph{Random Walk in Cooling Random Environment} (RWCRE) is a dynamic version of RWRE in which the environment is \emph{fully resampled} along a sequence of deterministic times, called \emph{refreshing times}.
% If the increments between consecutive refreshing times stay bounded, then correlations over time are weak and we expect to see a behaviour that is close to that of a homogeneous random walk. Conversely, if these increments diverge, then we expect to see a behaviour that is closer to that of RWRE. In particular, the faster the divergence, the more the behaviour should be like RWRE. RWCRE allows for different scenarios depending on the incremental structure of the refreshing times. Since RWRE may exhibit \emph{anomalous behaviour} due to the occurrence of \emph{trapping} (i.e., the random walk spends a long time in local niches of the environment), the same is true for RWCRE. We will refer to the resampling as \emph{cooling} because the medium in RWRE can be thought of as \emph{frozen}. We will see that cooling gives rise to interesting \emph{new phenomena}.

In order to understand RWCRE, we need certain \emph{concentration properties} of RWRE. Some of these are available from the literature, but others are not.  %and need to be developed along the way. 
A few preliminary results were obtained in Avena and den Hollander~\cite{AdH17} under the \emph{annealed} measure and subject to certain regularity conditions on the refreshing times.
% In particular, a weak law of large numbers and a centred central limit theorem were derived when the RWRE is recurrent. In the regime where the increments of the refreshing times diverge we found that, interestingly, while the limiting speed is the \emph{same} as for RWRE, the fluctuations are \emph{different} from those for RWRE, both in scale and in law. These results were further developed in Avena, Chino, da Costa and den Hollander~\cite{ACdCdH18} under the \emph{quenched} measure when the increments of the refreshing times diverge. Namely, a strong law of large numbers and a large deviation principle were derived. The rate function turns out to be the \emph{same} as for RWRE, and it was argued that this is not in contradiction with the fact that the fluctuations may be different, because the rate function is \emph{non-analytic} in a neighbourhood of the limiting speed. 

In the present paper we find conditions for recurrence versus transience and we identify fluctuations for \emph{general} cooling schemes with non-standard limit laws. In Section~\ref{RWRE} we define one-dimensional RWRE and recall some basic facts that are needed throughout the paper. In Section~\ref{model} we define RWCRE. Both these sections are largely copied from~\cite{ACdCdH18}, but are needed to set the stage and fix the notation. In Section~\ref{results} we state our main theorems. In Section~\ref{discussion} we place these theorems in their proper context and state a number of open problems. Proofs are provided in Sections~\ref{s:rxt}--\ref{s:Gauss}. Along the way we need a few refined properties of RWRE  that are of independent interest. These properties are stated in Section~\ref{RWREres}  and are proved in Appendices~\ref{appA}--\ref{appC}.

 %%%%%%%%%%%%%%%%%%%%%%%%%%%%%%%%%%%%%%%%%%%%%%%%%%%%%%%%%%%%
\subsection{RWRE: Basic facts} \label{RWRE}

Throughout the paper we use the notation $\N_0 = \N \cup \{0\}$ with $\N = \{1,2,\dots\}$. %, and the convention that $\tfrac00=1$. 
The classical one-dimensional static model is defined as follows. Let $\omega=\{\omega(x) \colon\,x\in\Z\}$ be an i.i.d.\ sequence with probability distribution
\begin{equation} \label{alpha}
\mu := \alpha^{\Z}
\end{equation} 
for some probability distribution $\alpha$ on $(0,1)$.  We assume that $\alpha$ is \emph{non-degenerate} and  write $\langle\cdot\rangle$ the corresponding expectation. We also assume that $\alpha$ is \emph{uniformly elliptic}, i.e., 
\begin{equation} \label{uellcond}
\exists\,\,\mathfrak{c}>0\colon \qquad \alpha(\mathfrak{c} \leq \omega(0) \leq 1-\mathfrak{c})=1.
\end{equation} 
%In Section~\ref{discussion} we will comment on how to relax this .  

%%%%%%%%%%%%%%%%%%%%%%%%%%%%%%%%%
\begin{definition}[{\bf RWRE}] \label{RWREdef} 
 Let  $\omega$ be an environment sampled from $\mu$. We call \emph{Random Walk in Random Environment} the Markov chain $Z = (Z_n)_{n\in\N_0}$ with state space $\Z$ and transition probabilities 
\begin{equation} \label{Ker_om}
P^{\omega}(Z_{n+1} = x + e \mid Z_n = x) = \left\{
\begin{array}{ll}
\omega(x), &\mbox{ if } e = 1,\\  
1 - \omega(x), &\mbox{ if } e = - 1,
\end{array}
\right. 
%\qquad 
\end{equation}
for $x \in \Z,\,n \in \N_0.$  We denote by $P_x^{\omega}(\cdot)$ the \emph{quenched} measure of $Z$ starting from $Z_0=x\in\Z$, and by 
%%
%\begin{equation} 
$P_{x}^{\mu}(\cdot) := \int_{(0,1)^{\Z}} P_x^{\omega}(\cdot)\,\mu(\dd\omega),$
%\end{equation}
%%
the \emph{annealed} measure. The corresponding expectations are denoted by $E_x^{\omega}$ and $E_{x}^{\mu}$.
\end{definition}
%%%%%%%%%%%%%%%%%%%%%%%%%%%%%%%%%%%%%%%%

\noindent
The understanding of one-dimensional RWRE is well developed, both under the quenched and the annealed measure. For a general overview, we refer the reader to the lecture notes by Zeitouni~\cite{ZZ04}.  Below we collect some basic facts and definitions. 

%The asymptotic properties of RWRE are controlled by the distribution of the \emph{ratio} of the transition probabilities to the left and to the right. 
%Define
%%
%\begin{equation} \label{rhodef}
%\rho := \frac{1 - \omega(0)}{\omega(0)}.
%\end{equation}
%%
The average displacement is $E_0^\mu[Z_1] = \langle \tfrac{1-\rho}{1+\rho}\rangle$, where $\rho := \tfrac{1 - \omega(0)}{\omega(0)}$. The following proposition due to Solomon~\cite{Sol75} characterises recurrence versus transience and limiting speed. Without loss of generality we may assume that
\begin{equation} \label{+0R}
\langle \log \rho \rangle \leq 0,
\end{equation}
because the reverse can be included via a reflection argument. Indeed, if $\widetilde{\omega}$ is defined by
%\begin{equation}\label{tildeomega}
$\widetilde{\omega}(x) = 1-\omega(-x)$,  for $x \in \Z$,
%\end{equation}
then $P_0^\omega(-Z \in \cdot\,) = P_0^{\widetilde{\omega}}(Z \in \cdot\,)$.
  
%%%%%%%%%%%%%%%%%%%%%%%%%%%%%%%%%%%%%%%%
\begin{proposition}[{\bf Recurrence, transience, speed of RWRE}~\cite{Sol75}]
\text{}\\
 \label{prop:LLN}
Suppose that~\eqref{+0R} holds. Then:
\begin{itemize} 
\item $Z$ is recurrent when $\langle \log \rho \rangle = 0$.
\item $Z$ is transient to the right when $\langle \log \rho \rangle <0$, and for $\mu \text{-a.e.} \,\omega$, %$Z$ has a speed $P_0^\omega $\text{-a.s.}: 
\begin{equation} \label{speed}
\lim_{n\to\infty} \frac{Z_n}{n} =: v_\mu  =  \left\{
\begin{array}{ll}
0, &\mbox{ if }  \langle\rho\rangle \geq 1,\\
\frac{1-\langle\rho\rangle}{1+\langle\rho\rangle} > 0 , &\mbox{ if } \langle\rho\rangle < 1.
\end{array}
\right.
\end{equation}
\end{itemize}
\end{proposition}
%%%%%%%%%%%%%%%%%%%%%%%%%%%%%%%%%%%%%%%%

\noindent 
The above proposition shows that the speed of RWRE is a deterministic function of $\mu$ (or $\alpha$; recall~\eqref{alpha}). %Note that for $\alpha$ such that $\langle \log \rho \rangle < 0$ and $\langle\rho\rangle \geq 1$, the random walk is transient to the right with zero speed. %In this regime, $Z_n \to\infty$ as $n\to\infty$, but only sublinearly in $n$ due to the presence of \emph{traps}, i.e., local niches of the environment slowing the walk down.

In the recurrent case the scaling was identified by Sinai~\cite{S82} and the limit law by Kesten~\cite{K86}. The next proposition summarises their results. We write $\tod$ to denote convergence in distribution and $\todp$ to denote convergence in $L^p$.

%%%%%%%%%%%%%%%%%%%%%%%%%%%%%%%%%%%%%%%%
\begin{proposition}[{\bf Scaling limit: recurrent RWRE~\cite{S82},~\cite{K86}}] \label{prop:Rec}
\text{}\\
Let $\alpha$ be such that $\langle\log\rho\rangle = 0$ and $\sigma_0^2 = \langle\log^2\rho\rangle \in (0,\infty)$. 
Then, under the annealed measure $P_0^\mu$, 
\begin{equation} \label{Sinai}
\begin{aligned}
\frac{Z_n}{\sigma_0^2 \log^2 n} \tod V
\end{aligned}
\end{equation} 
where the \emph{Sinai-Kesten} random variable $V$  is defined by $P(V \in A) := \int_A v(x) \, dx$ with 
\begin{equation} \label{densityofV}
v(x) := \frac{2}{\pi} \sum_{k\in\N_0} \frac{(-1)^k}{2k+1} 
\exp \left[- \frac{(2k+1)^2 \pi^2}{8} |x| \right], \qquad x \in \R.
\end{equation}
\end{proposition}
%%%%%%%%%%%%%%%%%%%%%%%%%%%%%%%%%%%%%%%

\noindent 
Note that the law of $V$ is symmetric with finite variance $\sigma^2_V \in (0,\infty)$. It was shown in~\cite{AdH17} that for $\alpha$ satisfying \eqref{uellcond}, under the annealed measure $P^{\mu}_0$,
\begin{equation}\label{sinaip}
\frac{Z_n}{\sigma_0^2\log^2 n} \todp V \qquad \forall \, p > 0,
\end{equation}
%%
%i.e., the convergence in distribution in holds in $L^p$. 
In the transient case the scaling and the limit law were identified by Kesten, Kozlov and Spitzer~\cite{KKS75}. The next proposition recalls their result only for the case where the scaling and the limit law are classical. We say that  $\alpha$ is \emph{$s$-transient} when $\langle\log \rho\rangle<0$, $\langle\rho^s\rangle=1$ and $\langle\rho(\log\rho)_+ \rangle<\infty$.

%%%%%%%%%%%%%%%%%%%%%%%%%%%%%%%%%%%%%%%%
\begin{proposition}[\bf Scaling limit: transient RWRE~\cite{KKS75}] \label{prop:SL_trans_RWRE}
\text{}\\
Let $\alpha$ be  $s$-transient with $s \in(2,\infty)$. Then there exists a $\sigma_s \in (0,\infty)$ such that, under the annealed measure $P^\mu_0$, 
\begin{equation}\label{normalstatic}
\frac{Z_n - v_\mu n}{\sigma_s \sqrt{n}} \tod  \Phi,
\end{equation}
where $\Phi$ stands for a standard normal random variable. 
\end{proposition}
%%%%%%%%%%%%%%%%%%%%%%%%%%%%%%%%%%%%%%%%%

\noindent 
%The scaling and the limit laws were identified also for the cases $s \in (0,1)$, $s = 1$, $s \in (1,2)$ and $s = 2$. For the first three cases the scales are $n^s$, $n/\log^2 n$ and $n^{1/s}$, respectively, and the limit laws are functions of stable laws. For the fourth case the scale is $\sqrt{n \log n}$ and the limit law is Gaussian. In the present paper we only consider the case $s \in (2,\infty)$.  

%%%%%%%%%%%%%%%%%%%%%%%%%%%%%%%%%%%%%%%%%%%%%%%%%%%%%%%%%%%%
\subsection{RWCRE: Cooling} \label{model}

The cooling random environment is the \emph{space-time} random environment built by partitioning $\N_0$, and assigning independently to each piece an environment sampled from $\mu$ in~\eqref{alpha} (see Fig.~\ref{fig:CRE}). Formally, let $\tau \colon\, \N_0 \to \N_0$ be a strictly increasing function with $\tau(0) = 0$, referred to as the \emph{cooling map}. The cooling map determines a sequence of \emph{refreshing times} $\prt{\tau(k)}_{k\in \N_0}$. %that we use to construct the dynamic random environment.

%%%%%%%%%%%%%%%%%%%%%%%%%%%%%%%%%%%%%%%%%%%%%
\begin{definition}[{\bf Cooling Random Environment}] 
\text{}\\
Given a cooling map $\tau$ and an i.i.d.\ sequence of random environments $\Omega=(\omega_k)_{k\in\N}$ with law $\mu^\N$, the \emph{cooling random environment} is built from the pair $(\Omega,\tau)$ by assigning, for each $k\in\N$, the environment $\omega_k$ to the $k$-th interval $I_k$ defined by
%%
%\begin{equation} \label{interval}
$I_k := [\tau(k-1), \tau(k))$, 
%\end{equation}
 which has size
%%
%\begin{equation} \label{coolingreg}
$T_k := \tau(k) - \tau(k - 1)$ for $k \in \N$.
%\end{equation}
%%
\end{definition}
%%%%%%%%%%%%%%%%%%%%%%%%%%%%%%%%

%%%%%%%%%%%%%%%%%%%%%%%%%%%%%%%%%%%%%%%%
\begin{definition}[{\bf RWCRE}] 
 Let $\tau$ be a cooling map and $\Omega$ an environment  sequence sampled from $\mu^\N$. We call \emph{Random Walk in Cooling Random Environment} (RWCRE) the Markov chain $X = (X_n)_{n\in\N_0}$ with state space $\Z$ and transition probabilities
\begin{equation} \label{Ker_Om}
P^{\Omega,\tau}(X_{n+1} = x + e \mid X_n = x) = \left\{
\begin{array}{ll}
\omega_{\ell(n)}(x), &e = 1,\\
1 - \omega_{\ell(n)}(x), &e = - 1,
\end{array}
\right. 
%\quad ,
\end{equation}
for $x \in \Z,\,n \in \N_0$, where
%%
%\begin{equation} \label{ell}
$\ell(n) := \inf\{k\in\N\colon\, \tau(k) > n\}$,
%\end{equation}
%%
 is the index of the interval that $n$ belongs to. Similarly to Definition~\ref{RWREdef}, we denote by
\begin{equation} \label{RWCREmeas}
P_x^{\Omega,\tau}(\cdot), \qquad  P_{x}^{\mu,\tau}(\cdot) := \int_{[(0,1)^\Z]^\N} 
P_x^{\Omega,\tau}(\cdot)\,\mu^\N(\dd\Omega),
\end{equation}
 the corresponding \emph{quenched} and \emph{annealed} measures, respectively.
\end{definition} 
%%%%%%%%%%%%%%%%%%%%%%%%%%%%%%%%%%%%%%%%%
%In words, RWCRE moves according to a given environment sampled from $\mu$ until the next refreshing time, when a new environment is sampled from $\mu$. The increments of the random walk trajectories are independent across the intervals. During the interval $I_k$, RWCRE behaves like RWRE in the environment $\omega_k$. Our goal is to understand in what way this makes RWCRE behave similarly as RWRE.       
%%%%%%%%%%%%%%%%%%%%%%%%%%%%%%%%%%%%%%%%%%%%%%%%
\begin{figure}[htbp]
\begin{center}
\includegraphics[clip,trim=3.6cm 24cm 2.8cm 2cm, width=.95\textwidth]{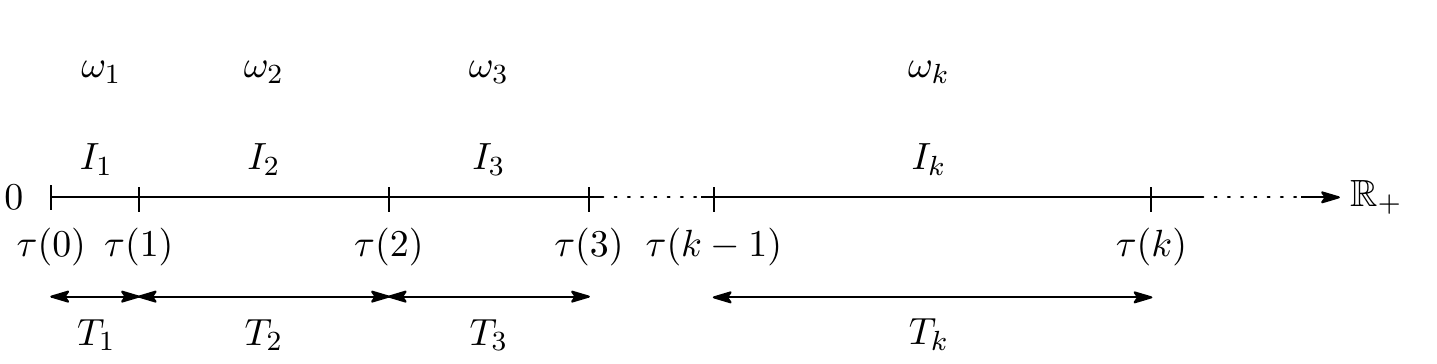}   
\caption{Structure of the cooling random environment $(\Omega,\tau)$.}
\label{fig:CRE}
\end{center}
\end{figure}
%%%%%%%%%%%%%%%%%%%%%%%%%%%%%%%%%%%%%%%%%%%%%%%%%

The position $X_n$ admits a decomposition into independent pieces.  For $k \in \N$, define the \emph{refreshed increments}  as
%%
%\begin{equation} \label{space_incr}
$Y_k := X_{\tau(k)} - X_{\tau(k-1)}$, and  for $n \in \N_0$, define the \emph{boundary increment} as $\bar{Y}^n := X_n - X_{\tau(\ell(n)-1)}$
%\end{equation}
%%
and the \emph{running time at the boundary} as 
%%
%\begin{equation} \label{time_incr}
$\bar{T}^n := n - \tau(\ell(n)-1)$.
%\end{equation}
%%
Note that 
\begin{equation} \label{time_incr*}
\sum_{k=1}^{\ell(n)-1} T_k + \bar{T}^n = n.
\end{equation}
By construction, we can write $X_n$ as the sum 
\begin{equation} \label{X_dec}
X_n = \sum_{k=1}^{\ell(n)-1} Y_k + \bar Y^n, \quad  n\in\N_0.
\end{equation}
This decomposition shows that, in order to analyse $X$, we must analyse the vector 
%%
%\begin{equation}
$(Y_1,\ldots,Y_{\ell(n)-1},\bar{Y}^n)$
%\end{equation}
%%
consisting of independent components, each distributed as an increment of $Z$ (defined in Section~\ref{RWRE}) in a given environment over a given length of time determined by $\Omega$, $\tau$ and $n$. Fig.~\ref{fig:RWCRE} illustrates the decomposition of $X_n$. 
%More precisely, for any measurable function $f\colon\,\Z\to\R$, any $\Omega$ sampled from $\mu^\N$ and any $\tau$,
%%
%\begin{equation} \label{equiv_distr}
%E^{\Omega,\tau}_0\crt{f(Y_k)} = E^{\omega_k}_0\crt{f(Z_{T_k})}, \quad k\in\N,
%\qquad E^{\Omega,\tau}_0\crt{f(\bar{Y}^n)} = E^{\omega_{\ell(n)}}_0\crt{f(Z_{\bar{T}^n})}.
%\end{equation}
%%
%%%%%%%%%%%%%%%%%%%%%%%%%%%%%%%%%%%%%%%%%%%%%%%%%
\begin{figure}[htbp]
\begin{center}
\includegraphics[clip,trim= 2cm 10cm 1cm 14cm, width=.95\textwidth]{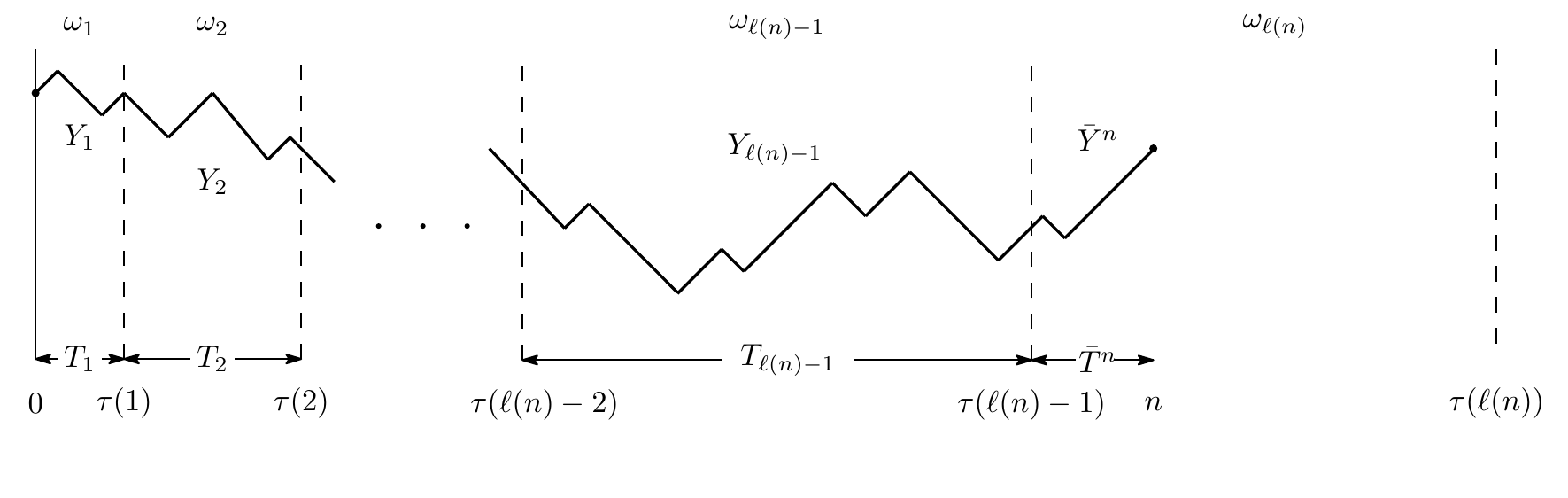}
\caption{The decomposition of RWCRE
into pieces of RWRE 
as in~\eqref{X_dec}.}
\label{fig:RWCRE}
\end{center}
\end{figure}
%%%%%%%%%%%%%%%%%%%%%%%%%%%%%%%%%%%%%%%%%%%%%%%%% 

To ease the notation, we will sometimes write 
%%
%\begin{equation}\label{Y0}
$T_0 := \bar{T}^n$, and $ Y_0 := \bar{Y}^n.$
%\end{equation}
%%
%In this way, the right-hand sides of~\eqref{time_incr*} and ~\eqref{X_dec} can be written as $\sum_{k=0}^{\ell(n)-1} T_k$ and $\sum_{k=0}^{\ell(n)-1} Y_k$.

%%%%%%%%%%%%%%%%%%%%%%%%%%%%%%%%%%%%%%%%%%%%%%%%%%%%%%%%%%%%
\subsection{Main results for RWCRE} \label{results}

In what follows, we write $\bb{P}$ for the annealed measure in~\eqref{RWCREmeas} when the random walk starts at the origin, suppressing $\mu,\tau, 0$ from the notation. We will denote by $\bb{E} $ and $\var$ the corresponding expectation and variance. We will further denote by $\mathfrak{X}_{n}$ the variance-scaled displacement at time $n\in \N_0$, 
\begin{equation} \label{e:sca}
\mathfrak{X}_{n} := \frac{X_{n}}{\sqrt{\var (X_n)}},
\end{equation}
with the convention that $\mathfrak{X}_{0} = 0$.

%%%%%%%%%%%%%%%%%%%%%%%%%%%%%%%%%%%%%%%%%%%%%%%%%%%%%%%%%%%%
\subsubsection{Recurrence versus transience} \label{ss:RvT}

We start by exploring how the cooling map affects the recurrence versus transience criterion for RWRE (see Proposition~\ref{prop:LLN}). A few remarks are in place. Since for any event $A$,
\begin{equation}
\bb{P}(A)=1 \quad \iff \quad P_0^{\Omega,\tau}(A)=1 \quad \mu^\N\text{-a.s}.,
\end{equation}
we do not distinguish between quenched and annealed statements when it comes to zero-one laws. Moreover, due to the resampling, RWCRE is \emph{tail-trivial}, i.e., all events in the tail sigma-algebra have probability zero or one.  We know from Proposition~\ref{prop:LLN} that RWRE is recurrent if and only if $\langle\log\rho\rangle = 0$.  We say that $\alpha$ is \emph{recurrent} or \emph{right-transient} when $\langle\log\rho\rangle = 0$, respectively, $\langle\log\rho\rangle< 0$. For RWCRE the classification of recurrence versus transience is more delicate, because it also depends on the cooling map $\tau$. In what follows we say that $(\alpha,\tau)$ is \emph{recurrent} or \emph{transient} when 
\begin{equation}
\bb{P}(X_n=0 \,\, \text{i.o.})=1 \quad \text{or} \quad \bb{P}(X_n=0 \,\, \text{i.o.})=0.
\end{equation}
We say that $(\alpha,\tau)$ is \emph{right transient} or \emph{left transient} when
\begin{equation}
\bb{P}\left(\lim_{n\to \infty} X_n=\infty \right)=1 \quad \text{or} \quad \bb{P}\left(\lim_{n\to \infty} X_n=-\infty \right)=1.
\end{equation}
By tail triviality, $\{0,1\}$ are the only possible values for the above events. 

Our first theorem gives two conditions on the cooling map under which recurrence and transience are not affected by the resampling. 

%%%%%%%%%%%%%%%%%%%%%%%%%%%%%%%%%%%%%%%%%%%%%%%%%%

\begin{theorem}[{\bf Stability of recurrence versus transience}] \label{thm:RvT}
\text{}\\ \vspace{-.5cm}
%%%%%
\begin{itemize}
\item[{\rm(a)}] If $\alpha$ is right-transient, then $(\alpha,\tau)$ is right-transient for all $\tau$ such that 
\begin{equation} \label{truecooling}
\lim_{k\to\infty} T_k=\infty.
\end{equation}
\item[{\rm(b)}] If $\alpha$ is recurrent, then $(\alpha,\tau)$ is recurrent when 
\begin{equation} \label{sufficient_rec}
\liminf_{n\to\infty} \abs{\bb{E} \left[ \mathfrak{X}_n \right]} = 0.
\end{equation}
The latter holds for all symmetric $\alpha$ and all $\tau$,
and also for all non-symmetric $\alpha$ when $\tau$ is such that 
\begin{equation}\label{sufficient_fast}
\liminf_{k\to\infty} \frac{1}{k^\gamma} \log T_k >0 \text{ for some } \gamma > \frac34.
\end{equation} 
\end{itemize}
\end{theorem}

\noindent
Non-symmetric $\alpha$ means that the laws of $\omega$ and $\widetilde{\omega}$ are different (see below \eqref{+0R}). Note that~\eqref{sufficient_fast} is much more stringent than~\eqref{truecooling}.
% and corresponds to a regime where the resampling is so slow that it cannot affect the recurrence. THIS IS WRONG???????

\medskip\noindent
{\bf Remark:} If the refreshing increments stay bounded (a regime that in \cite{AdH17} was referred to as `no cooling'), then RWCRE has little relation to RWRE and no resemblance is to expected. 
%An example is the cooling map $\tau(k)=k$, $k\in\N_0$, where the environment is resampled every unit of time. In this case RWCRE is, under the annealed law, equivalent to homogeneous random walk. 

\medskip
A recurrence criterion for general cooling maps is lacking and is presumably delicate, as shown by the following examples for which a weaker form of divergence of the increments is still in force. To weaken~\eqref{truecooling} we consider refreshing time increments that \emph{Cesaro diverge}, i.e., increments satisfying
\begin{equation} \label{Cesarocooling}
\lim_{\ell \to\infty} \frac{1}{\ell} \sum_{k=1}^\ell T_k= \infty.
\end{equation}

%%%%%%%%%%%%%%%%%%%%%%%%%%%%%%%%%%%%%%%%
\paragraph{\bf Counterexamples to stability}
\label{Rexamples}

\begin{itemize}
\item[{\bf(Ex.1)}] \emph{Right-transient can turn into left-transient or recurrent}: 
There exist a right-transient $\alpha$ and two cooling maps $\tau'=\tau'(\alpha)$ and $\tau''=\tau''(\alpha)$ satisfying~\eqref{Cesarocooling} such that $(\alpha,\tau')$ and $(\alpha,\tau'')$ are left-transient and recurrent, respectively.

\item[{\bf(Ex.2)}]\emph{Recurrent can turn into transient}:
There exist a recurrent $\alpha$ and a cooling map $\tau=\tau(\alpha)$ satisfying~\eqref{truecooling} such that $(\alpha,\tau)$ is transient.
\end{itemize}

\noindent
In Section~\ref{s:rxt} we prove Theorem~\ref{thm:RvT} and show {\bf (Ex.1)} and {\bf (Ex.2)}.
%are given in Section~\ref{s:rxt}.  

%%%%%%%%%%%%%%%%%%%%%%%%%%%%%%%%%%%%%%%%%%%%%%%%%%%%%%%%%%%%
\subsubsection{Fluctuations in the Sinai regime}
\label{subsec:recfl}

The following statements identify the scaling limits of RWCRE for recurrent $\alpha$. They show that the scaling depends in a delicate way on the cooling map. In particular, Theorem~\ref{thm:rlpts} below gives a characterisation of the possible limit points as mixtures of Sinai-Kesten and Gaussian random variables, while Corollary~\ref{cor:rRecThree} and {\bf (Ex.3)}--{\bf (Ex.6)} below give a further characterisation of the various possible regimes.  

To state our results we need the following definitions. Set
\begin{equation} \label{e:lfp} %lambda final parameter
\lambda_{\tau,n}(k) :=
\frac{\sqrt{\var (Y_k)}}{\sqrt{\var(X_n)}}\,\Ind{\{0\leq k<\ell(n)\}},  \quad n \in \N,\, k\in \N_0,
\end{equation}
and $\lambda_{\tau,0}(k) := \delta_{k0}$, $k\in\N_0$. Note that, by \eqref{X_dec}, $\lambda_{\tau,n}$ is a vector of real numbers with unit $\ell_2(\bb{N}_0)$-norm, i.e.,
%%
%\begin{equation}\label{e:l2norm}
$\norm{\lambda_{\tau,n}}_{2}^2 := \sum_{k\in\N_0} \lambda_{\tau,n}(k)^2 = 1$.
%\end{equation}
%%
With this notation, we can write
\begin{equation} \label{e:c3} %coupling representation
\mathfrak{X}_n - \bb{E}[\mathfrak{X}_n]
= \sum_{k=0}^{\ell(n)-1} \lambda_{\tau,n}(k)\,
\frac{Y_k - \bb{E}\crt{Y_k}}{\sqrt{\var(Y_k)}}.
\end{equation}
Let $(V_j)_{j \in \bb{N}_0}$ be  a family of i.i.d.\ Sinai-Kesten random variables (see~\eqref{densityofV}). For $\lambda = \prt{\lambda(j)}_{j \in \bb{N}_0}$, define the $\lambda$-mixture of normalised Sinai-Kesten random variables as
\begin{equation} \label{e:KSmix} % Kesten Sinai mixture
V^{\otimes \lambda} := \sum_{j\in\N_0} \lambda(j) (\sigma_V^{-1}V_j).
\end{equation}

For $\lambda \in \ell_2(\bb{N}_0)$, let $\lambda^\downarrow$ be the vector obtained from $\lambda$ by reordering the entries of $\lambda$ in decreasing order. Consider the equivalence relation
%%
%\begin{equation}\label{sim}
%\begin{aligned}
$\lambda\sim \lambda'$
when 
$\lambda^\downarrow
=\lambda'^\downarrow$
%\end{aligned}
%\end{equation}
%%
and put $[\lambda] := \chv{\lambda' \in \ell_2(\bb{N}_0)\colon \lambda' \sim \lambda }$. The following lemma, which is proven in Section~\ref{SKmix}, guarantees that up to reordering $V^{\otimes \lambda}$ corresponds to a unique vector $\lambda$. 

%%%%%%%%%%%%%%%
\begin{lemma}[{\bf Characterisation of Sinai-Kesten mixtures}]
\label{SKlemma}
\text{}\\
$V^{\otimes \lambda}$ and $V^{\otimes \lambda'}$ have different distributions if and only if $[\lambda] \neq [\lambda']$. 
\end{lemma}

Define by $\lambda^{0\downarrow}$ the vector obtained from $\lambda$ by putting $\lambda(0)$ as the first entry and reordering the other entries in decreasing order. This notation is needed in order to isolate the boundary increment. In what follows, $(n_i)_{i\in\N_0}$ denotes a strictly increasing sequence of integers with $n_0=0$.

%%%%%%%%%%%%%%%%%%%%%%%%%%%%%%%%%%%%%%%%
\begin{theorem}[{\bf Limit distributions in the Sinai regime}] \label{thm:rlpts} 
\text{}\\
%recurrent limit points
Let $\alpha$ be recurrent with $\sigma_0\in (0,\infty)$ and let $\tau$ be a cooling map. Under the annealed measure $\bb{P}$,  the sequence of centred random variables $\prt{\mathfrak{X}_n - \bb{E}[\mathfrak{X}_n]}_{n\in\N_0}$ is tight and its limit points are characterised as follows. If $(n_i)_{i\in\N_0}$ is such that 
\begin{equation}
\lim_{i\to\infty} \lambda^{0\downarrow}_{\tau,n_i}(k) =: \lambda_*(k) \qquad \forall\,k \in \N_0, 
\end{equation}
then
\begin{equation}\label{lpoints0}
\mathfrak{X}_{n_i} - \bb{E}[\mathfrak{X}_{n_i}]
\todp
V^{\otimes \lambda_*} 
+ a(\lambda_*)\,\Phi \qquad \forall\,p>0,
\end{equation}
where $a(\lambda_*) := (1 - \|\lambda_*\|_2^2)^{\tfrac12}$, $\Phi$ is a standard normal random variable, and $V^{\otimes \lambda_*} $ is as in \eqref{e:KSmix}.
\end{theorem}
%%%%%%%%%%%%%%%%%%%%%%%%%%%%%%

It is possible to distinguish between the different scaling limits by looking at the asymptotic behavior of $(\lambda_{\tau,\tau(k)}(k))_{k\in\N_0}$, the sequence of relative weights of the refreshed increments.

%%%%%%%%%%%%%%%%%%%%%%%%%%%%%%
\begin{corollary}[{\bf Limit distributions for regular cooling maps}]  \label{cor:rRecThree}
\text{}\\ %\vspace{-.2cm}
%%%%%%%%%%
For any $p>0$, under the annealed measure $\bb{P}$
\begin{itemize}
\item[{\rm (a)}] $\mathfrak{X}_n -\bb{E}\crt{\mathfrak{X}_n}$ converges in $L^p$  if and only if $\lambda_{\tau,\tau(k)}(k) \to 0$, in which case
\begin{equation}\label{conv0a}
\mathfrak{X}_n -\bb{E}\crt{\mathfrak{X}_n} \todp \Phi.
\end{equation}
\item[{\rm(b)}] If $\lambda_{\tau,\tau(k)}(k) \to q \in(0,1]$, then 
\begin{equation}\label{conv0b}
\mathfrak{X}_{\tau(k)} \todp V^{\otimes \lambda_q},
\end{equation}
where
%%
%\begin{equation} \label{e:qmc} %q mixture coefficients
$\lambda_q(0) := 0$, and for $j \in \N$, $ \lambda_q^2(j) := q^2 (1 - q^2)^{j-1}$. 
%\end{equation}
%%
Moreover, if for a subsequence $(n_i)_{i\in\N_0}$ the limit
%%
%\begin{equation} \label{e:bqmc}
$w := \lim_{i\to\infty} \lambda_{\tau,n_i}(0)$
%\end{equation}
%%
exists, then
\begin{equation}\label{bconver0b}
\mathfrak{X}_{n_i}\todp w\,\sigma_V^{-1}V_0
+ (1 - w^2)^{\frac12} V^{\otimes \lambda_q}.
\end{equation}
\end{itemize}
%%%%%%%%%%
\end{corollary}
%%%%%%%%%%%%%%%%%%%%%%%%%%%%%%

\noindent
The proofs of Theorem~\ref{thm:rlpts} and Corollary~\ref{cor:rRecThree} are given in Section~\ref{sec:Fluc}. 
%Corollary~\ref{cor:rRecThree}(a) characterises convergence of the full sequence. 
%Convergence holds when the boundary increment can be neglected, and always leads to a standard Gaussian limit law. 
%Corollary~\ref{cor:rRecThree}(b) says that if the boundary increment cannot be neglected, then subsequential limits can still be characterised.  

%%%%%%%%%%%%%%%%%%%%%%%%%%%%%%%%%%%%%%%%
\paragraph{\bf Examples of subsequential limits}
\label{4examples}

We illustrate Corollary~\ref{cor:rRecThree} by considering examples of cooling maps that diverge at different rates. In examples {\bf (Ex.3)-- (Ex.6)} below all convergence statements are under the annealed measure $\bb{P}$. 
%%%%%
\begin{itemize}
\item[{\bf(Ex.3)}] \emph{Polynomial cooling }: If $k^{-\beta} T_k \to B$ for some $B,\beta \in (0,\infty)$,  then
\begin{equation} \label{poly*}
\frac{X_n - \bb{E}[X_n]}{\sigma_0^2n^{\frac{1}{2(\beta +1)}}\log^{2}n}
\todp
\prt{\frac{\beta}{\beta + 1}}^2 B^{-\frac{1}{2(\beta + 1)}}\sigma_V \Phi.
\end{equation}
\item[{\bf(Ex.4)}] \emph{Exponential cooling}: If $ k^{-1} \log T_k \to c \in (0,\infty)$, then 
\begin{equation} \label{e:i*}%
\frac{X_n}{\sigma_0^2\log^{\frac{5}{2}}n}  \todp \frac{1}{\sqrt{5c^5}}\,\sigma_V \Phi.
\end{equation}
\item[{\bf(Ex.5)}] \emph{Double exponential cooling}: If $k^{-1} \log\log T_k \to c \in (0,\infty)$, then
\begin{equation} \label{e:ii*}%
\frac{X_{\tau(\ell)}}{ \sigma_0^2\log^{2}\tau(\ell)} \todp q_c^{-1}  \sigma_VV^{\otimes\lambda_{q_c}} 
 \text{ with } q_c^2 = \frac{e^{4c}-1}{e^{4c}} \in (0,1). 
\end{equation}
\item[{\bf(Ex.6)}] \emph{Faster than double exponential cooling}: If $ k^{-1} \log\log T_k \to \infty$, then 
\begin{equation} \label{e:iii*} %
\frac{X_{\tau( \ell)}}{\sigma_0^2 \log^{2}\tau( \ell)} \todp V.
\end{equation}
\end{itemize}
%%%%%

In {\bf(Ex.5)} and {\bf(Ex.6)} we can even characterise \emph{all} the limit points. Indeed, if a subsequence $\prt{n_i}_{i\in\N_0}$ is such that %(recall \eqref{time_incr})
\begin{equation}\label{e:beg}
\lim_{i\to\infty} \frac{\log \bar{T}^{n_i} }{ \log \tau(\ell(n_i)-1)} =: b \in [0,\infty], % \text{ exists}.
\end{equation}
then
\begin{equation}\label{e:bcases}
\frac{X_{n_i}}{\sigma_0^2 \log^{2}n_i} \todp \begin{cases}
q_c^{-1} \sigma_VV^{\otimes \lambda_{q_c}} + b^2V_0, &\text{ if }b \leq 1,\\
b^{-2}q_c^{-1} \sigma_VV^{\otimes \lambda_{q_c}} + V_0, &\text{ if } b > 1,
\end{cases}
\end{equation}
with $b^{-1} = 0$ when $b = \infty$. 
%Note that $b$ controls the relative sizes of the running time at the boundary and the last resampling time.

The claims in {\bf (Ex.3)}--{\bf (Ex.6)} are proven in Section~\ref{ss:threeexamples}. 
%We conclude our analysis in the Sinai regime by noting a property that was conjectured in~\cite{AdH17}: as we increase the rate at which the cooling increments diverge, some sort of \emph{crossover} emerges in the order of the scaling and in the limit laws: from pure Gaussian to pure Sinai-Kesten through mixtures.        

%%%%%%%%%%%%%%%%%%%%%%%%%%%%%%%%%%%%%%%%%%%%%%%%%%%%%%%%%%%%
\subsubsection{Fluctuations in the Gaussian regime}
 
We next examine the scaling limit when $\alpha$ is $s$-transient with $s \in (2,\infty)$, i.e., when RWRE satisfies a classical CLT (recall Proposition~\ref{prop:SL_trans_RWRE}).

%%%%%%%%%%%%%%%%%%%%%%%%%%%%%%%%%%%%%%%%
\begin{theorem}[{\bf Scaling limit in the Gaussian regime}] \label{thm:TranGauss}
\text{}\\
Let $\alpha$ be $s$-transient with $s \in (2,\infty)$ and let $\tau$ be any cooling map. Then, under the annealed measure $\bb{P}$,
\begin{equation} \label{e:GFcis} %gaussian fluctuation centered term implicit scaling
\mathfrak{X_n} - \bb{E}\crt{\mathfrak{X_n}} \todtwo  \Phi.
\end{equation}
\end{theorem}

\noindent
Theorem~\ref{thm:TranGauss} says that in the Gaussian regime also RWCRE converges to a Gaussian. 
%This is not surprising, because if RWRE is in a regime where homogenisation is dominant, then this dominance will be even further amplified by the presence of the refreshing times, which weakens the space-time dependence.
 However, the scaling of the variance as a function of the cooling map is subtle, as we show next.

\begin{corollary}[{\bf Gaussian limits and stability of the variance}] 
\label{Gaussregular}
\text{}\\
Fix $s \in (2,\infty)$. The sequence %$(n^{-1/2}(X_n - \bb{E}[X_n]))_{n\in\N_0}$
$\left((X_n - \bb{E}[X_n])/\sqrt{n}\,\right)_{n\in\N_0}$ is tight and its limit points correspond to the limit points of 
%%
%\begin{equation}\label{sfactor}
$\sigma^2_{s,\tau}(n) := \var\prt{X_n}/n$.
%\end{equation}
%%
Namely, given a subsequence $(n_i)_{i\in\N_0}$, if $\sigma_{s,\tau}(n_i) \to \sigma$, then, under the annealed \mbox{measure} $\bb{P}$,
\begin{equation}\label{gaussgeneral}
\frac{X_{n_i}-\bb{E}[X_{n_i}]}{\sigma \sqrt{n_i}} \todp \Phi.   
\end{equation}
Moreover, if $T_k \to \infty$, then 
\begin{equation}\label{gausscooling}
\sigma_{s,\tau}(n) \to \sigma_s,
\end{equation}
with $\sigma_s$ the standard deviation from Proposition~{ \ref{prop:SL_trans_RWRE}}.
\end{corollary}
%%%%%%%%%%%%%%%%%%%%%%%%%%%%%%%%%%%%%%%%

We conclude our analysis of the Gaussian regime by looking into the centering term in \eqref{e:GFcis}.
\paragraph{Centering and correction in the law of large numbers}
\label{2examples}

%Interestingly, 
In general the centering term $\bb{E}\crt{X_n}$ in Theorem~\ref{thm:TranGauss} (recall \eqref{e:sca}) cannot be replaced by %$v(\alpha,\tau)\, n$ with $v(\alpha,\tau)$ 
the limiting speed of $X$. In {\bf(Ex.7)} below we provide a class of rapidly diverging cooling maps for which such a replacement causes no harm. In {\bf(Ex.8)} below we indicate that there exist slowly diverging cooling maps for which it does. 
\begin{itemize}
\item[{\bf(Ex.7)}] \emph{Stable centering for rapidly diverging cooling maps:} 
For $s \in (2,\infty)$, if $T_k \to \infty$ and
\begin{equation}\label{scltgrowth} 
\sup_{n\in \N_0} \sum_{k = 0}^{\ell(n)-1} \lambda_{\tau,n}(k) <\infty,
\end{equation}
then
\begin{equation}\label{zero}
\frac{\bb{E}\crt{X_n} - v_\mu\, n}{\sqrt{n}} \to 0 ,
\end{equation} 
with $v_\mu$ the RWRE speed in~\eqref{speed}, from which it follows via \eqref{e:GFcis} and \eqref{gausscooling} that
\begin{equation}\label{sclt} 
\frac{X_n - v_\mu n}{\sigma_s \sqrt{n}}\todtwo   \Phi.
\end{equation}
Moreover, \eqref{scltgrowth} holds when $\liminf_{k\to\infty} k^{-1}\log T_k>0$.

\item[{\bf(Ex.8)}] \emph{Counterexample with slowly diverging cooling maps:}
For any $s \in (2,\infty)$ there exist an $s$-transient $\alpha$ and a cooling map $\tau$ with $T_k \to \infty$, such that~\eqref{sclt} fails. In particular, there exist ``extreme'' examples for which
\begin{equation}\label{infinity}
\limsup_{n\to\infty}
\frac{\bb{E}\crt{X_n} - v_\mu\, n}{\sqrt{n}} = \infty.
\end{equation} 
In such cases, the sequence %$(n^{-1/2}(X_n - v_\mu n))_{n\in\N_0}$
$((X_n - \bb{E}[X_n])/\sqrt{n}\,)_{n\in\N_0}$ is not even tight (compare with Corollary~\ref{Gaussregular}).
\end{itemize}

\noindent
Condition~\eqref{scltgrowth} in {\bf(Ex.7)} imposes a growth condition on $T_k$.  {\bf(Ex.8)} shows that the convergence in~\eqref{sclt} may fail even when $T_k \to \infty$.

%It is worth clarifying that the existence of $v(\alpha,\tau)$ depends on whether a law of large numbers is in force. The latter was proven in~\cite{AdH17} for the case of bounded regular $(T_k)_{k\in\N}$, and in~\cite{ACdCdH18} under \eqref{truecooling}. Yet, such a law of large numbers can be established for more general cooling maps under mild regularity assumptions. Here we only mention that the strong law of large numbers in~\cite{ACdCdH18} can be extended to the case where the refreshing time increments Cesaro diverge, as in ~\eqref{Cesarocooling}. In this case, $v(\alpha,\tau)$ exists a.s.\ and is equal to $v_\mu$.

The proofs of Theorem~\ref{thm:TranGauss} and Corollary~\ref{Gaussregular} are given in Section~\ref{s:Gauss}.

%%%
\subsection{Auxiliary properties of RWRE} \label{RWREres}

In our analysis of RWCRE we need a few results about RWRE. 
% that have not been derived in the literature, yet they are interesting in their own right.  
The first states that in Proposition~\ref{prop:SL_trans_RWRE} the convergence can be extended to $L^p$ for $p<s$.

%%%%%%%%%%%%%%%%%%%%
\begin{theorem}[{\bf $L^p$-convergence in the Gaussian regime}] \label{thm:RWREextra1}
\text{}\\
Suppose that the assumptions in Proposition~{ \ref{prop:SL_trans_RWRE}} are in force. Then
\begin{equation}\label{Lpsnormal}
\frac{Z_n-v_\mu n}{\sigma_s\sqrt{n}} \todp \Phi \qquad \forall\, p < s.
\end{equation}
\end{theorem}
%%%%%%
\noindent The second result concerns various forms of oscillation of the mean of RWRE. %The construction of the cooling maps used in the examples given in Section~\ref{results} relies on the nature of these oscillations. 

%%%%%%%%%%%%%%%%%%%%
\begin{theorem}[{\bf Oscillations of the mean}] \label{thm:RWREextra2}
\text{}\\ \vspace{-.4cm}
%%%%%
\begin{itemize}
\item[{\rm(I)}] There is a recurrent $\alpha$ such that $E^\mu_{0}[Z_n] \neq 0$ for infinitely many $n\in\N$.
\item[{\rm(II)}] For every $s \in (2,\infty)$, there is an $s$-transient $\alpha$ such that $ E^\mu_{0}[Z_n] \neq v_\mu n$ for infinitely many $n\in\N$.
\item[{\rm(III)}] If $\alpha$ is recurrent with $\sigma_0 \in (0,\infty)$, then for every $0<\gamma<\tfrac23$ there is a $C=C(\alpha,\gamma) \in (0,\infty)$ such that
%%% 
\begin{equation}\label{goal_expectation}
\left| E^\mu_{0}\left[\frac{Z_n}{\sigma_0^2\log^2n}\right]\right| \leq \frac{C}{\log^\gamma n}, 
\qquad n\in\N.
\end{equation}
\end{itemize}
\end{theorem}
%%%%%%%%%%%%%%%%%%%%
 
\noindent
The proofs of Theorems~\ref{thm:RWREextra1}--\ref{thm:RWREextra2} are given in Appendices~\ref{appA}--\ref{appC}. The line of proof of Theorem~\ref{thm:RWREextra2}(III) was suggested by Zhan Shi.

%%%%%%%%%%%%%%%%%%%%%%%%%%%%%%%%%%%%%%%%%%%%%%%%%%%%%%%%%%%%
\subsection{Discussion and open problems} \label{discussion}
\text{}

\noindent
{\bf Ellipticity.}
The uniform ellipticity assumption in \eqref{uellcond} is needed in the proof of Theorem~\ref{thm:RWREextra2}(III) only. Once this would be extended, all our results would carry over. In the proof of Theorem~\ref{thm:RvT}(a) we need a concentration property for which it suffices to have a very mild form of ellipticity. In the proof of Theorem~\ref{thm:rlpts} and Corollary~\ref{cor:rRecThree} we use \eqref{sinaip}, which was proved in \cite{AdH17} under \eqref{uellcond} only, but should be true more generally.     

\medskip\noindent
{\bf Stability of recurrence and transience.}
While RWRE asymptotics are  non-local due to  space-time correlations, for  RWCRE, resampling  \emph{adds extra noise} and weakens space-time dependencies. From this perspective, we can view RWCRE as a \emph{perturbation} of RWRE. Theorem~\ref{thm:RvT} describes how this perturbation affects the recurrence versus transience criterion known for RWRE. Theorem~\ref{thm:RvT}(a) shows that transience is preserved as soon as the increments of the refreshing times diverge, while Theorem~\ref{thm:RvT}(b) says that the situation is more delicate for recurrence, unless $\alpha$ is symmetric. In fact, as shown in {\bf(Ex.2)}, for non-symmetric $\alpha$, resampling is capable of destroying recurrence. We will see in Section~\ref{s:rxt} that this happens because there are increments of the refreshing times during which the average displacement of RWRE is strictly positive. By repeating such increments often enough, we are able to pull the random walk away from the origin. The increments of the refreshing times in such cooling maps are diverging, but \emph{slowly enough} so that RWCRE is qualitatively different from RWRE. As shown in {\bf (Ex.1)}, cooling can even turn right-transience into left-transience. 

\medskip\noindent
{\bf Mixed fluctuations in Sinai regime.}
It is well-known that trapping phenomena are predominant when RWRE is recurrent (see Proposition~\ref{prop:Rec}). The underlying correlation structure gives rise  to subdiffusive scaling with a non-Gaussian limit law. Theorem~\ref{thm:rlpts} and Corollary~\ref{cor:rRecThree} show that this scenario is affected by the extra noise introduced by the cooling. %in a sensitive way that depends on how frequent the refreshing times occur. 
Indeed %, for all cooling maps the fluctuations live on a scale larger than those of RWRE 
 RWCRE is less localised, although convergence in distribution of the full sequence is not guaranteed in general. Theorem~\ref{thm:rlpts} shows that regular subsequential limits are characterised by mixtures of Gaussian laws and properly weighted Sinai-Kesten laws. Corollary~\ref{cor:rRecThree}(a) provides a necessary and sufficient condition, in terms of  the boundary term in the sum in \eqref{X_dec}, under which all subsequential limits coincide, in which case a standard Gaussian law emerges after a  scaling that is gauged by the divergence in the cooling map. Corollary~\ref{cor:rRecThree}(b), instead, says that if the boundary increment is not negligible, then the full sequence oscillates, and properly chosen subsequences lead to different mixed limit laws, as illustrated by {\bf (Ex.3)}--{\bf (Ex.6)}. These subsequences can be further characterised depending on whether the boundary term dominates or competes with the other terms, as illustrated in~\eqref{e:bcases}. Such results yield the answer to a conjecture put forward in~\cite{AdH17}, where the analysis of the fluctuations in the Sinai regime was carried out for cooling maps for which Lindeberg-Feller type conditions are satisfied, essentially corresponding to the condition in Corollary~\ref{cor:rRecThree}(a). 

\medskip\noindent
{\bf CLT in the Gaussian regime and centering issues.}
RWCRE can be seen as an interpolation between RWRE and a homogeneous random walk. Thus, not surprisingly, Theorem~\ref{thm:TranGauss} shows that if RWRE satisfies a CLT (i.e., when $s\in(2,\infty)$), then the same is true for RWCRE. Yet, as is clear from Corollary~\ref{Gaussregular}, the cooling can make the variance oscillate on scale $n$, but not under \eqref{truecooling}. {\bf(Ex.7)} and {\bf(Ex.8)} shown that,  if the cooling map is ``sufficiently concentrated" as captured in condition~\eqref{scltgrowth}, it \emph{must} be centered with the average displacement.

\medskip\noindent
{\bf Refined properties of RWRE.}
Section~\ref{RWREres} collects a few refined properties of RWRE that are not available in the literature but are needed in our proofs. In particular, Theorem~\ref{thm:RWREextra1} extends the mode of convergence in Proposition~\ref{prop:SL_trans_RWRE} to $L^p$, and we use the latter in the proof of Theorem~\ref{thm:TranGauss}. Concerning Theorem~\ref{thm:RWREextra2}, items (I) and (II) are similar in spirit, and say that in the recurrent and transient regime, respectively, the limiting speeds are not achieved after a finite time. These statements may sound plausible, but the disorder does not allow for a simple proof, as can be appreciated from Appendix~\ref{appB}. We use items (I) and (II) to construct {\bf(Ex.2)} and {\bf(Ex.8)}, respectively. Item  (III) gives some control (possibly not optimal) on the rate of convergence in Proposition~\ref{prop:Rec}, which we use in the proof of~\eqref{sufficient_rec}.

\medskip\noindent
{\bf Extensions and open problems:}

\begin{itemize}

\item ({\bf Regime with limiting stable laws}).
The only regime for which we have not analysed RWCRE fluctuations is when $\alpha$ is $s$-transient with $s \in (0,2]$. In this regime, the RWRE fluctuations are more intricate. Under the annealed measure it is known that, after an appropriate scaling, RWRE converges to certain stable laws or inverse-stable laws. Under the quenched measure fluctuations are drastically different and actually have only been partially characterised. In particular, different subsequential limits are possible under the quenched measure. For precise statements we refer the reader to~\cite{ZZ04} and references therein. The analysis of RWCRE with $s \in (0,2]$ should lead to interesting cooling-dependent crossover phenomena.% and deserves a more careful analysis.
  
\item ({\bf Higher dimensions}).
The focus in the present paper and in~\cite{AdH17}, \cite{ACdCdH18} is on \emph{one-dimensional} RWCRE. It is natural to consider RWCRE also in higher dimensions. However, much less is known for RWRE in higher dimensions, and most of the relevant results require additional and often technical assumptions (see~\cite{ZZ04}). Nonetheless, some of our arguments and results may be adapted to higher dimensions, in particular, those concerning the stability of directional transience and directional speed.

\item ({\bf Recurrence criterion for arbitrary cooling}).
We partially solved the problem of recurrence versus transience in Theorem~\ref{thm:RvT}. The following problem is left open:  If $\alpha$ is recurrent and non-symmetric, then what is a \emph{necessary and sufficient} condition on $\tau$ such that RWCRE is recurrent?

\item ({\bf RWRE oscillations}).
Some of the statements in Theorem~\ref{thm:RWREextra2} are non-optimal. For example, in part (2) we should be able to show that $E_0^\mu[Z_n] \neq v_\mu n$ for infinitely many $n\in\N$ for every $s$-transient $\alpha$ with $s \in (2,\infty)$.  Such an improvement would allow us to strengthen the statement in {\bf(Ex.8)} by saying that for every $s$-transient $\alpha$ with $s \in (2,\infty)$ there exists a $\tau$ such that~\eqref{infinity} is satisfied.
\end{itemize}

%%%%%%%%%%%%%%% SECTION 2 %%%%%%%%%%%%%%%%%%%%%%%%%%%%%%%%%%%%
\section{Proofs: Recurrence versus transience} \label{s:rxt}

%%%%%%%%%%%%%%%%%%%%%%%%%%%%%%%%%%%%%%%%%%%%%%%%%%%%%%%%%%%%
\subsection{Transience is preserved for any cooling with diverging increments} \label{ss:transience}
%We prove Theorem~\ref{thm:RvT}(a).
%\begin{proof}
\Proof{Proof of Theorem~\ref{thm:RvT}{\rm (a)}}
%Without loss of generality w
We assume that  $\langle\log\rho\rangle < 0$.%, i.e., that RWRE is right-transient.

\paragraph{ Basic coupling}
Let us consider a probability space $(S,{\cal S},{\cal P})$ on which random variables $(X_n)_{n\in\N_0}$ and $(Z^{(k)}_n)_{k\in\N,n \in \N_0}$ are defined such that
\begin{equation}\label{coupling}
\begin{aligned}
& Y_k = Z^{(k)}_{T_k},\, k \in \N, \quad  \bar{Y}^n = Z^{(\ell(n))}_{\bar{T}^n},\, n \in \N_0,\\
&{\cal P}\left(\left(X_n\right)_{n\in\N_0}\in\cdot\,\right)= \bb{P}\left(\left(X_n\right)_{n\in\N_0}\in\cdot\,\right), \\
&(Z^{(k)}_n)_{k \in \N,n \in \N_0} \text{ are independent in $k\in\N$,}\\
&{\cal P}\left((Z_n^{(k)})_{n\in\N_0}\in \cdot\,\right) = P^\mu_0\left(\left(Z_n\right)_{n\in\N_0}\in\cdot\,\right),\, k\in\N.
\end{aligned}
\end{equation}
This constitutes a \emph{coupling} of RWRE and RWCRE. We write $\cal{E}$ to denote expectation with respect to $\cal{P}$.
\paragraph{Leftmost record} 
Set $W := \inf\{Z_n \colon n\in\N_0\}$ and $W^{(k)} := \inf\{Z^{(k)}_n \colon n\in\N_0\}$, $k \in \N$.
% and note that, for every $k \in \N$, $W^{(k)} \leq Z^{(k)}_{T_k}$ and $W^{(k)}\leq 0$. 
By~\eqref{X_dec}, for any $a >0$ and $\ell\in\N$,
\begin{equation}
\label{e:drt}
\begin{aligned}
X_{\tau(\ell)} = \sum_{k=1}^\ell Y_k \, &\, =\, \sum_{k=1}^\ell  Z^{(k)}_{T_k} \Ind{\{Z^{(k)}_{T_k} > a\}} 
+ \sum_{k=1}^\ell  Z^{(k)}_{T_k} \Ind{\{Z^{(k)}_{T_k} \leq a\}}\\
&\geq \sum_{k=1}^\ell  Z^{(k)}_{T_k}  \Ind{\{Z^{(k)}_{T_k} > a\}} + \sum_{k=1}^\ell W^{(k)}.
\end{aligned}
\end{equation}
The following lemma tells us that the expectation of $-W$ is finite.% when RWRE is right-transient.
  
%%%%%%%%%%%%%%%%%%%%%%%%%%%%%%%%%
\begin{lemma} \label{l:fmi}
Suppose that $\langle \log \rho \rangle < 0$. Then $E^{\mu}_0[-W] <  \infty$.
\end{lemma}
%%%%%%%%%%%%%%%%%%%%%%%%%%%%%%%%%%
 
%%%%%%%%%%%%%%%%%%%%%%%%%%%%%%%%%%%%%%%%
\begin{proof}
Write
\begin{equation} \label{e:exp}
E^\omega_0[- W] = \sum_{m\in\N} P^\omega_0(W \leq - m).
\end{equation}
%%
%For $m \in \N$, let $H_{- m}$ denote the first hitting time of $-m$. 
For $j\in\Z$, let $\rho_j := \frac{1-\omega(j)}{\omega(j)}$
% Since
%%%
%\begin{equation} \label{e:expalt}
%\prod_{j=-m+1}^{i-1} \rho_{j} = \exp \left\{ (m+i-1) \, \frac{\sum_{j=-m+1}^{i-1} \log \rho_j}{m+i-1} \right\},
%\qquad i \in \N,
%\end{equation}
% %%
%and $(\log\rho_j)_{j\in\Z}$ are i.i.d., the strong law of large numbers tells us that $\mu$-a.s.
%%%
%\begin{equation} \label{e:slln}
%\lim_{i\to\infty} \frac{\sum_{j=-m+1}^{i-1} \log \rho_j}{m+i-1} = \langle\log\rho\rangle.
%\end{equation}
%%
 and for  $m \in \N, \, \gep > 0$, define
\begin{equation}\label{Omegaset}
\Omega(m,\gep) := \left\{ \omega \colon \sup_{i \in \N} \abs{\frac{\sum_{j = -m+1}^{i-1} \log \rho_j}{m+i-1}
- \langle \log \rho \rangle} < \gep \right\}.
\end{equation}
For $0<\gep < -\tfrac12 \langle\log\rho\rangle$ and $\omega\in \Omega(m,\gep)$,
\begin{equation}\label{e:expdecayW}
\prod_{j=-m+1}^{i-1} \rho_{j} \leq e^{\frac12 \langle\log\rho\rangle (m+i-1)}.
\end{equation} 
Therefore there is a $c>0$ such that, for all $m \in \N$ and $\omega\in \Omega(m,\gep)$,
\begin{equation} \label{e:qedi} 
\begin{aligned}
P^\omega_0(W \leq - m)%&= P^\omega_0\left(\inf_{n\in\N} Z_n  \leq - m\right) %= P^\omega_0(\tau_{-m} < \infty)\\[0.3cm]
&\leq  \sum_{i=1}^\infty \prod_{j = -m+1}^{i-1} \rho_{j} \leq e^{- c m},
\end{aligned}
\end{equation}
where the first inequality follows from a standard computation for RWRE (see~\cite[p.196 (2.1.4)]{ZZ04}), and the inequality uses~\eqref{e:expdecayW}.

Next we note that %rely on the exponential decay in $m$ of the probability of $\Omega^c(m,\gep)$, the complement of $\Omega(m,\gep)$, i.e., 
there is a $c'> 0$ such that, for all $m\in\N$, using 
\begin{equation} \label{e:dOm}
\begin{aligned}
%\mu(\Omega^c(m,\gep))= 
\mu\left(\omega \colon \sup_{i \in \N} \abs{\frac{\sum_{j = -m+1}^{i-1} \log \rho_j}{m+i-1} 
- \langle \log \rho \rangle} \geq \gep \right) \leq e^{- c' m},
\end{aligned}
\end{equation}
where the  inequality follows from the union bound in combination with the large deviation principle for the i.i.d.\ random variables $(\log \rho_j)_{j\in\Z}$. (For the latter the uniform ellipticity assumption in \eqref{uellcond} amply suffices, but
can be substantially weakened.). Combining \eqref{e:qedi} and~\eqref{e:dOm} we see that, for all $m\in\N$,
\begin{equation} \label{e:est_W}
\begin{aligned}
P^\mu_0 (W \leq -m)  %\\
%&= \int_{\Omega(m,\gep)} P^\omega_0 (W \leq -m) \, d\mu(\omega)
%+\int_{\Omega^c(m,\gep)} P^\omega_0 (W \leq -m) \, d\mu(\omega)\\
%&\leq  e^{-c m} + e^{-c' m} 
\leq 2\,e^{-(c \wedge c') m}.
\end{aligned}
\end{equation}
The result follows from~\eqref{e:exp} and~\eqref{e:est_W}.%This, together with%~\eqref{e:exp} and
%~\eqref{e:est_W},we get the claim. 
\end{proof}

\paragraph{Transience along subsequences via the leftmost record}
We continue %with 
the proof of Theorem~\ref{thm:RvT}(a). Pick $a := 4\,E^{\mu}_0[- W]<\infty$. Since $\alpha$ is right-transient and $T_k \to \infty$, we have
%%
%\begin{equation}
$ \bb{P}(Y_k > a) =  P^{\mu}_0(Z_{T_k} > a) \to 1$.
%\end{equation}
%%
From stochastic domination together with the independence of $Z^{(k)}_{T_k}$, $k\in\N$, we get that
\begin{equation} \label{e:elz}
\liminf_{\ell\to\infty} \frac{1}{\ell} \sum_{k=1}^\ell Z^{(k)}_{T_k} \Ind{\{Z^{(k)}_{T_k} > a\}} 
\geq a \liminf_{\ell\to\infty}  \frac{1}{\ell} \sum_{k=1}^\ell  \Ind{\{Z^{(k)}_{T_k} > a\}}
\geq a  \qquad {\cal P}\text{-a.s.}
\end{equation}
Now, applying the law of large numbers  %for $W^{(k)}$, 
%%
%\begin{equation} \label{e:elw}
%\lim_{\ell\to\infty} \frac{1}{\ell} \sum_{k=1}^\ell W^{(k)} = E^\mu_0[W] = - \tfrac{1}{4} a \qquad {\cal P}\text{-a.s.}
%\end{equation}
%%
and~\eqref{e:elz} into~\eqref{e:drt} we get
\begin{equation}
\liminf_{\ell\to\infty} \frac{X_{\tau(\ell)}}{\ell} \geq  \frac{3}{4} a \qquad \bb{P}\text{-a.s.}
\end{equation}
%%
%Therefore there exists an $L \in \N$ such that, for $\ell > L$,
%%%
%\begin{equation}
%X_{\tau(\ell)} \geq \tfrac{1}{2} a \, \ell \qquad \bb{P}\text{-a.s.},
%\end{equation}
%%
which settles right-transience along the sequences of refreshing times.

\paragraph{Transience of the full sequence}
For  $\ell$  large and $n \in [\tau(\ell), \tau(\ell+1))$ we have $X_n \geq \tfrac12 a \, \ell + W^{(\ell)}$.
Hence
\begin{equation}
\bb{P}\left( \inf_{n \in [\tau(\ell), \tau(\ell+1))} X_n \leq 0 \right) \leq  {\cal P}\left(W^{(\ell)} \leq -\frac12 a \, \ell \right).
\end{equation}
By~\eqref{e:est_W}, %there is a $c''>0$ such that
%%
%\begin{equation}
$\sum_{\ell\in\N}{\cal P}\left(W^{(\ell)} \leq - \tfrac12 a \, \ell \right) <\infty$
%\leq e^{-c''\,\ell},
%\end{equation}
%%%
%which implies that
%%
%\begin{equation}
%\sum_{\ell\in\N}  \bb{P}\left(\inf_{n \in [\tau(\ell), \tau(\ell+1))} X_n \leq 0 \right) < \infty.
%\end{equation}
%%
and hence, by the Borel-Cantelli lemma, %$\bb{P}\left(\inf_{n \in [\tau(\ell), \tau(\ell+1))} X_n \leq 0 \;\;  \text {i.o.}\right) = 0$,
the right-transience of the full sequence follows.
%\end{proof}
\QED

%%%%%%%%%%%%%%%%%%%%%%%%%%%%%%%%%%%%%%%%%%%%%%%%%%%%%%%%
\subsection{Recurrence is preserved for fast enough cooling}

%We prove Theorem~\ref{thm:RvT}(b).

%\begin{proof}
\Proof{Proof of Theorem~\ref{thm:RvT}{\rm (b)}}
The sequence $(\lambda_{\tau,n})_{n \in \N_0}$ of $\ell_2(\N_0)$-unit vectors in~\eqref{e:lfp}%--\eqref{e:l2norm} 
~admits a subsequence $(n_i)_{i\in\N_0}$ for which there is a vector $\lambda_*$ with $\|\lambda_*\|_2 \leq 1$ such that, for every $k \in \N_0$,
%\begin{equation}
$\lambda_{\tau,n_i}(k) \to \lambda_*(k)$.
%\end{equation}
By Theorem~\ref{thm:rlpts} (to be proved in Section~\ref{sec:pTSKlp}), and condition~\eqref{sufficient_rec}
\begin{equation}\label{convnocenter}
\frac{X_{n_i}}{\sqrt{\var(X_{n_i})}}
\tod V^{\otimes \lambda_*} + a\,\Phi.
\end{equation}
Since $\var (X_n) \to \infty$, $(n_i)_{i\in\N_0}$ can be chosen such that
\begin{equation} \label{e:n_j}
\frac{n_{i-1}}{\sqrt{\var(X_{n_i})}} < \frac12, \qquad i \in \N. 
\end{equation}
%%
%With this choice,
%%%
%\begin{equation} \label{e:nii}
%\begin{aligned}
%\tfrac{X_{n_i}}{\sqrt{\var(X_{n_i})}} > 1 \quad
%&\Longrightarrow \quad \tfrac{X_{n_i} - X_{n_{i-1}}}{\sqrt{\var(X_{n_i})}} >  \tfrac12,\\
%\tfrac{X_{n_i}}{\sqrt{\var(X_{n_i})}} <- 1\quad
%&\Longrightarrow \quad  \tfrac{X_{n_i} - X_{n_{i-1}}}{\sqrt{\var(X_{n_i})}} < - \tfrac12.
%\end{aligned}
%\end{equation}
%%
Now, because $V^{\otimes \lambda_*} + a\, \Phi$ has full support on $\R$, %and~\eqref{convnocenter} holds, ~\eqref{e:n_j} 
there is an $\gep > 0$ for which
\begin{equation}
\begin{aligned}
\bb{P} \left( \frac{X_{n_i} - X_{n_{i-1}} }{\sqrt{\var(X_{n_i})}} > \frac12 \right)
&\geq \bb{P} \left( \frac{X_{n_i}} {\sqrt{\var(X_{n_i})}} > 1 \right) > \gep,\\
\bb{P} \left( \frac{X_{n_i}- X_{n_{i-1}}}{\sqrt{\var(X_{n_i}})} < - \frac12 \right) 
&\geq \bb{P} \left( \frac{X_{n_i}} {\sqrt{\var(X_{n_i})}} < -1 \right) > \gep.
\end{aligned}
\end{equation}
The independence of $(X_{n_i} - X_{n_{i-1}})_{i\in\N}$ and the Borel-Cantelli lemma imply%, we obtain that
\begin{equation}
\label{BC}
\bb{P} \left( \frac{X_{n_i} - X_{n_{i-1}}}{\sqrt{\var(X_{n_i})}}
> \frac12 \,\, \text{i.o.} \right) = 1,
\qquad
\bb{P} \left( \frac{X_{n_i} - X_{n_{i-1}}}{\sqrt{\var(X_{n_i})}}
< - \frac12 \,\, \text{i.o.} \right) = 1.
\end{equation}
Since $X$ makes steps of size 1 only, we get from~\eqref{e:n_j} %that
%\begin{equation}
%\begin{aligned}
%&X_{n_i} \leq 0 \quad \Longrightarrow \quad \tfrac{X_{n_i} - X_{n_{i-1}}}{\sqrt{\var(X_{n_i})}}
%\leq \tfrac{- X_{n_{i-1}}}{\sqrt{\var(X_{n_i})}} \leq \tfrac{n_{i-1}}{\sqrt{\var(X_{n_i})}} < \tfrac12,\\
%&X_{n_i} \geq 0 \quad \Longrightarrow \quad \tfrac{X_{n_i} - X_{n_{i-1}}}{\sqrt{\var(X_{n_i})}}
%\geq \tfrac{- X_{n_{i-1}}}{\sqrt{\var(X_{n_i})}} \geq -\tfrac{n_{i-1}}{\sqrt{\var(X_{n_i})}} > -\tfrac12.
%\end{aligned}
%\end{equation}
%Hence %it follows from~\eqref{BC} that
%%
%\begin{equation}
$\bb{P} \left( X_{n_i} > 0 \,\, \text{ i.o.}\right) = 1$ and $\bb{P} \left( X_{n_i} < 0 \,\, \text{ i.o.} \right) = 1$,
%\end{equation}
%%
which proves the first claim in Theorem~\ref{thm:RvT}(b).

It remains to show that~\eqref{sufficient_fast} implies~\eqref{sufficient_rec}. In the remainder of the proof, $c,C$ denote constants that may change from line to line, but do not depend on $n$. First note that~\eqref{sinaip} and~\eqref{X_dec} imply 
\begin{equation}\label{varlogpiece}
%\begin{align}
%&c\log^4T_k \leq \var\prt{Y_k} \leq C \log^4 T_k,
%\label{varlogpiece}\\
%&\var (X_{\tau(\ell)}) \geq c \sum_{k = 1}^\ell  \log^4 T_k.
%\label{varlog}
c\log^4T_k \leq \var\prt{Y_k} \leq C \log^4 T_k, \quad
\var (X_{\tau(\ell)}) \geq c \sum_{k = 1}^\ell  \log^4 T_k.
%\end{align}
\end{equation}
For any fixed $\gep>0$, Theorem~\ref{thm:RWREextra2}(III),~\eqref{varlogpiece} % and~\eqref{varlog} 
yield
\begin{equation}\label{eeb}
 \abs{\bb{E}[\mathfrak{X}_{\tau(\ell)}]}
\leq  \sum_{k = 1}^{\ell}\ \frac{\sqrt{\var (Y_k)}}{\sqrt{\var(X_{\tau(\ell)})}}\,
\abs{\bb{E}\crt{\frac{Y_k}{\sqrt{\var(Y_k)}}}}
\leq  C \, \frac{\sum_{k=1}^\ell \log^{\left(\frac43+\gep\right)} T_k}{\sqrt{\sum_{k=1}^\ell \log^4 T_k}}.
\end{equation}
By H\"older's inequality it follows that
\begin{equation}\label{holder}
\sum_{k =1}^\ell \log^{\left(\frac43 + \gep\right)}T_k \leq 
\prt{\sum_{k =1}^\ell \log^4T_k}^{\frac{4/3+\gep}{4}} \ell^{\frac{8/3-\gep}{4}},
\end{equation}
which leads to
\begin{equation}\label{holbound}
\abs{\bb{E}[\mathfrak{X}_{\tau(\ell)}]} \leq C \, 
\frac{\prt{\sum_{k =1}^\ell \log^4T_k}^{\frac{4/3+\gep}{4}} \ell^{\frac{8/3-\gep}{4}}}
{\sqrt{\sum_{k=1}^\ell \log^4 T_k}}\leq C\frac{ \ell^{\frac{8/3-\gep}{4}}}
{\prt{\sum_{k=1}^\ell \log^4 T_k}^{\frac{2/3 -\gep}{4}}}.
\end{equation}
By~\eqref{sufficient_fast} it follows that
%%
%\begin{equation}\label{e:logsum}
$\sum_{k = 1}^\ell\log^{4}T_k \geq c \ell^{4\gamma + 1}$,
%\end{equation}
%%
and so
\begin{equation}\label{e:hol}
\abs{\bb{E}[\mathfrak{X}_{\tau(\ell)}]} 
\leq 
C \frac{\ell^{\frac{8/3-\gep}{4}}}
{\prt{\ell^{4\gamma+ 1}}^{\frac{2/3 -\gep}{4}}}
=
 C  \ell^{\frac12-\gamma\left(\frac23-\gep\right)}.
\end{equation}
when $\gamma>\tfrac{3}{4-6\gep}$, $\abs{\bb{E}[\mathfrak{X}_{\tau(\ell)}]} \to 0$.
% and proves the claim. %Since $\gep>0$ is arbitrary, this proves the claim.

To conclude the proof we take arbitrary $n \in\N$. We have 
\begin{equation}\label{bsplit}
\abs{\bb{E} [\mathfrak{X}_n]} \leq \frac{\var(X_{\tau(\ell(n)-1)})}{\var(X_{n})}\abs{\bb{E} [\mathfrak{X}_{\tau(\ell(n)-1)}]} 
+ \frac{\var(\bar{Y}^n)}{\var(X_{n})}\frac{\abs{\bb{E}[\bar{Y}^n]}}{\var(\bar{Y}^n)}.
\end{equation}
By \eqref{e:hol}, the first term in the right-hand side of \eqref{bsplit} vanishes as $n \to \infty$. As for second term, it is bounded by $\gep + (K/\var(X_n))$. Indeed, by~\eqref{sinaip} and~\eqref{coupling}, for any $\gep>0$ there is a $K>0$ such that
%\begin{equation}\label{epsbbig}
$\bar{T}^n >K$ implies $\abs{\bb{E}[\bar{Y}^n]}/\var(\bar{Y}^n) <\gep$. 
%\end{equation}
%Therefore 
%If $\bar{T}^n\leq K$, since $\bb{E}[\bar{Y}^n]<K$, it follows that the last term in~\eqref{bsplit} is smaller that 
%\begin{equation}\label{gbbound}%global boundary bound
%\frac{\var(\bar{Y}^n)}{\var(X_{n})}\frac{\abs{\bb{E}[\bar{Y}^n]}}{\var(\bar{Y}^n)} 
%\leq 
%this term is smaller than 
%\end{equation}
As $\var(X_n) \to \infty$ and $\gep>0$ is arbitrary, %it follows that
%\begin{equation}\label{bconclude}
$\abs{\bb{E} [\mathfrak{X}_n]} \to 0$.
%\end{equation}
%\end{proof}
\QED

%%%%%%%%%%%%%%%%%%%%%%%%%%%%%%%%%%%%%%%%%%%%%%%%%%%%%%%%%%%%
\subsection{Breaking of transience}
\label{breaktr}

%We examine {\bf(Ex.1)}.

%\begin{proof}
\Proof{Proof of {\bf(Ex.1)}}
We %show how to build
construct the two maps $\tau'$ and $\tau''$ in {\bf(Ex.1)}  %Such an $\alpha$ exists. %as shown by the following construction. 
%Let $\alpha = u\delta_v + (1-u)\delta_u$ and pick $u,v \in (0,1)$ such that $\langle\log\rho\rangle < 0$. Letting $u \downarrow 0$, we get that   $E^\mu_0[Z_1] \downarrow -1$.% and $E^\mu_0[Z_1] = u(2v-1) + (1-u)(2u-1)$.  Pick $u\in (0,\tfrac12)$,  $u\log\frac{1-w}{w} + (1-u)\log\frac{1-u}{u}=0$, and pick $v \in (w,1)$. Then $\langle\log\rho\rangle<0$.  Hence there is a choice of $u,v$ that fits the requirements.     

%\medskip\noindent
\paragraph{The cooling map $\tau'$} 
Take $\alpha$ such that $\langle\rho \rangle>1$ and $\langle \log\rho\rangle <0$. By Proposition \ref{prop:LLN},
%%
%\begin{equation}\label{zerospeed}
$P^\mu_0\prt{\lim_{n\to\infty} \frac{ Z_n}{n} =0} = 1$.
%\end{equation}
%%
In this case we can build a cooling map satisfying \eqref{Cesarocooling} for which
%%
%\begin{equation} \label{ltr}
$\bb{P}(\lim_{n \to \infty} X_n = - \infty) = 1$.
%\end{equation}
%%
The construction goes as follows. Let $v := E^\mu_0[-Z_1]>0$. Using the notation introduced in~\eqref{coupling}, we set $N_0 = 0$ and for each $i\in \N$ we choose $N_i$ such that $N_i\geq 2N_{i-1}  $ and 
\begin{equation}\label{slowspeed}
E^\mu_0[Z_{i N_i}]< vN_i. 
\end{equation}
%The existence of $\prt{N_i}_{i \in \N}$ follows from \eqref{zerospeed}. 
We take the \emph{ $i$-th environment piece} to be composed of one increment of size $i N_i$ followed by $N_i$ increments of size $1$ (see Fig.~\ref{fig:lenv}). By~\eqref{slowspeed}, the increments over the $i$-th piece have negative expectation.%:
%\begin{equation}\label{negdrift}
%\mc{E}\crt{\,  Z^{( 1)}_{i N_i} + \sum_{k = 2}^{N_i + 1}Z^{(k)}_{1} }<0.
%\end{equation}

%%%%%%%%%%%%%%%%%%%%%%%%%%%%%%%%%%%%%%%%%%%%%%%%%
\begin{center}
\begin{figure}[htbp]
\includegraphics[clip,trim=5cm 20cm 6cm 4.5cm, width=.95\textwidth]{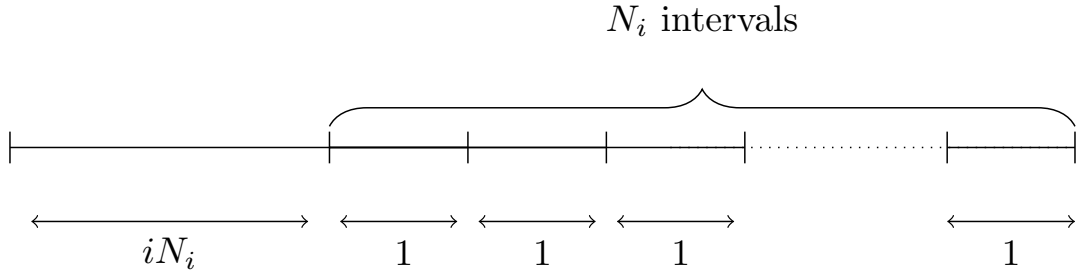}
\caption{The $i$-th environment }\label{fig:lenv}
\end{figure}
\end{center}

%%%%%%%%%%%%%%%%%%%%%%%%%%%%%%%%%%%%%%%%%%%%%%%%%%

The idea to build the cooling map is, for each $i\in\N$, to repeat $I_i$ times the $i$-th environment piece in order to induce left-transience of the random walk. Formally, %define %the cooling map, 
for $i\in \N$, $j \in \N_0$, define
%%%
%\begin{equation}\label{s0left}
$s(0) := 0,$ $s(i,j) := s(i-1) + j (N_i + 1)$, $s(i) := s(i,I_i)$,
%\end{equation}
%%%
let $A_i:= \chv{s(i,j)\colon\, j \leq I_i} $, and define  %, 
the increments of the map $\tau'$, $\chv{T'_k=\tau' (k) - \tau'(k-1)}_{k\in \N}$, by
%%%
\begin{equation}\label{T1leftt}%Tk for cooling 1 left transience 
T'_k :=
\begin{cases}
i N_i,  & \text { if }k-1 \in A_i    \text{ for some } \ell \in \N,\\
1, & \text { else}.
\end{cases}
\end{equation}
%%%
%Before examining the transience we argue 
We note that, irrespective of the choice of $(I_i)_{i \in \N}$, this construction ensures the Cesaro divergence of the increments.

%\paragraph{Cesaro divergence}
%Note first that the sequential average of the first $j$ terms within the $i$-th piece is always at least $i/2$.
%Indeed, for $j \leq  N_{i} + 1$ 
%\begin{equation}\label{cwithin}
%\frac{i N_i + j-1}{j} \geq\frac{i N_i }{N_i+ 1}\geq\frac{i}{2}.
%\end{equation}
%%Because
%%\begin{equation}\label{fractions}
%%\frac{a}{b}<\frac{c}{d} \quad \Longrightarrow \quad \frac{a}{b}<\frac{a + c}{b+d} <\frac{c}{d},
%%\end{equation}
%It follows that, for every $\ell >s(i -1)$,
%\begin{equation}\label{cesaro tail}
%\sum_{k = s(i-1) + 1}^{\ell} \frac{T'_k}{\ell -s(i-1)} \geq \frac{i}{2}.
%\end{equation}
%To conclude the Cesaro divergence, note that \eqref{cesaro tail} implies
%\begin{equation}\label{cesarodiv}
%\begin{aligned}    
%\liminf_{\ell\to \infty}\frac{\sum_{k = 1}^\ell T'_k}{\ell}  
%%&= \liminf_{\ell\to \infty}\bigg(\frac{s(i-1)}{\ell} %\frac{\sum_{k = 1}^{s(i-1)}T'_k}{s(i-1)} 
%%&+ \frac{k - s(i-1)}{k}\sum_{k = s(i-1) + 1}^{\ell} \frac{T'_k}{\ell -s(i-1)} \bigg)\\
%%&=  \liminf_{\ell\to \infty}\prt{\frac{\ell - s(i-1)}{k}\sum_{k = s(i-1) + 1}^{\ell} \frac{T'_k}{\ell -s(i-1)}}
% \geq \frac{i}{2}.
%\end{aligned}
%\end{equation}

\paragraph{Left transience}
Before choosing $I_i$, we note that the displacement over different $i$-th environment pieces are i.i.d\ random variables. For $\ell\in \N$ and $i \in \N_0$, denote such displacements by
\begin{equation}\label{Displacement}
D^{(j)}_i :=  Z^{\left(s(i,j)+1\right)}_{\ell N_i} + \sum_{k = s(i,j)+2}^{s(i,j)+N_i+1}Z^{(k)}_{1}. 
\end{equation}
By the strong law of large numbers, there is a sequence of positive integers $(M_i)_{i\in\N}$ satisfying
\begin{equation}\label{leftdrift}
M_{i+1}\geq M_i,
\qquad \mc{P}\crt{\sup_{m \geq M_i}\sum_{j = 1}^m D^{(j)}_i \geq 0  }< 2^{-i}.
\end{equation}
The sequence $\prt{I_i}_{i \in \N}$ is chosen to satisfy the following condition:
\begin{equation} 
\mc{P}\prt{  \sum_{j= 1}^{I_i} D^{(j)}_i \geq - (M_{i + 1}+1) (i+2) N_{i + 1}} < 2^{-i}, 
\label{fastleft}
\end{equation}
By the Borel-Cantelli Lemma, due to~\eqref{coupling}~\eqref{leftdrift} and~\eqref{fastleft}, we have that eventually
\begin{equation}\label{leftevent}
\begin{aligned}
& X_{\tau'(s(i))} - X_{\tau'(s(i -1))} <  - (M_{i + 1}+1) (i+2) N_{i + 1}, \\[0.2cm]
&\sup_{n \in [\tau'(s(i)), \tau'(s(i+1))]} X_{n} - X_{\tau'(s(i))} < (M_{i + 1}+1) (i+2) N_{i + 1}, \\[0.2cm]
\end{aligned}
\end{equation} 
which implies left-transience (see Fig.~\ref{fig:leftevent}).

%%%%%%%%%%%%%%%%%%%%%%%%%%%%%%%%%%%%%%%%%%
\begin{figure}[htbp] 
\begin{center} 
\includegraphics[clip,trim=5cm 19cm 2cm 4.5cm, width=.95\textwidth]{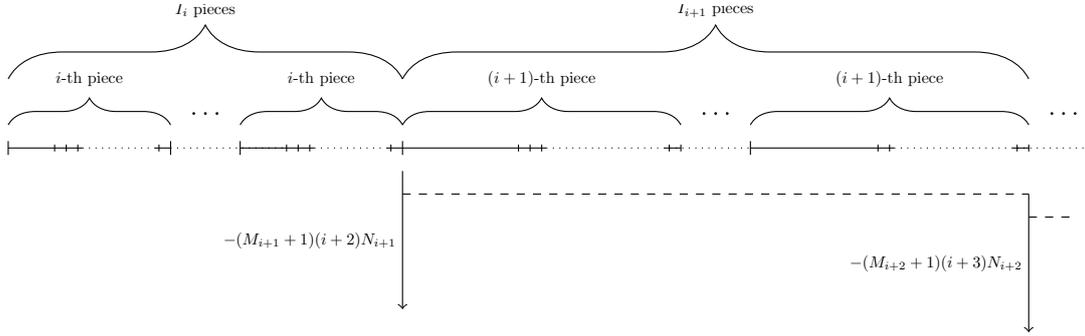}
%\vspace{28mm}
\caption{Picture of the bound  encoded in~\eqref{leftevent}. The downarrows represent the decrease at the end of the last $i$-th environment piece in comparison with the value at the beginning of the first $i$-th environment piece. The dashed line represents the upper bound on the supremum of the random walk.}% along the next $(i+1)$-st environment pieces.}
\label{fig:leftevent}
\end{center}
\end{figure}
%%%%%%%%%%%%%%%%%%%%%%%%%%%%%%%%%%%%%%%%%%%%%

%\medskip\noindent
\paragraph{The map $\tau''$}
In the same setting as {\bf (Ex.1)}, we construct a recurrent RWCRE by modifying the cooling map $\tau'$, inserting large intervals. First note that, since $\langle\log\rho\rangle<0$, we can define, for any $N\in\N$ and $\gep>0$,
\begin{equation}\label{strech}
H(N,\gep) := \inf \left\{m\in\N\colon\,\mc{P}\prt{ \inf_{n> m}Z^{(1)}_n \leq N} < \gep\right\}.
\end{equation}
Let $T' _k := \tau'(k) - \tau'(k-1)$. Inductively, define the increment sequence $\chv{T''_k}_{k \in \N}$ by setting
\begin{equation}\label{T''}
\begin{cases}
T''_k = T'_k, &k \in \N \setminus \chv{s(i)\colon i \in \N},\\
T''_{s(i)} = T'_{s(i)}+ H\left(\sum_{i = 1}^{s(i-1)}T''_i, 2^{-i}\right), &i \in \N,
\end{cases}
\end{equation}
where $s(0):=0$ and
\begin{equation}\label{seqs}
s(i) := \left\{\inf_{k>s(i-1)}\colon\, \mc{P}\left(\sum_{i = s(i-1)+1}^k Z^{(i)}_{T'_i} 
\geq -\sum_{i=1}^{s(i-1)}T''_i\right)< 2^{-i}\right\}.
\end{equation}
With these definitions, set $\tau''(k) := \sum_{i = 1}^k T''_i$ and note that, since $T''_k \geq T'_k$, the increments are Cesaro diverging. We conclude the proof, by noting that~\eqref{coupling},~\eqref{strech},~\eqref{seqs} and the Borel-Cantelli lemma imply
\begin{equation}\label{oscilate}
\begin{aligned}
&\bb{P}\prt{X_{\tau''(s(i)-1)} -X_{\tau''(s(i-1))}> -{\tau''(s(i-1))} \text{ i.o.}}= 0 ,\\
&\bb{P}\prt{X_{\tau''(s(i))}-X_{\tau''(s(i)-1)} < \tau''(s(i)-1)\text{ i.o.}} = 0.
\end{aligned}
\end{equation}
%%
%From ~\eqref{oscilate} %and the nearest-neighbour property of the random walk, it follows that
%\begin{equation}\label{backk_forkio}
%\bb{P}\prt{X_{n}=0 \text{ i.o.}} = 1.
%\end{equation}
%%
%\end{proof}
\QED
%%%%%%%%%%%%%%%%%%%%%%%%%%%%%%%%%%%%%%%%%%%%%%%%%%%%%%%%%%%
\subsection{Breaking of recurrence}

%We examine {\bf(Ex.2)}.

%\begin{proof}
\Proof{Proof of {\bf(Ex.2)}}
We show that there exists a recurrent non-symmetric $\alpha$ and a cooling map $\tau$ for which $(\alpha,\tau)$ is transient. The construction that follows is possible because, by Theorem~\ref{thm:RWREextra2}(1), there is a recurrent non-symmetric $\alpha$ for which at least one of the sets
\begin{equation}
\begin{aligned}
\mathcal{N}_+ := \{n \in \N \colon\, E^\mu_0[Z_n] > 0\}, \quad\mathcal{N}_- := \{n \in \N \colon E^\mu_0[Z_n] < 0\},
\end{aligned}
\end{equation}
is infinite. Assume without loss %of generality 
that $\mathcal{N}_+ =\chv{n_1<n_2<\ldots}$ is infinite.% \infty$. Let $n_1<n_2<\ldots$ denote the elements of $\mathcal{N}_+$.

Successively pick $N_j$ consecutive increments of size $n_j$ for every $j\in\N$, where the sequence $\prt{N_j}_{i\in\N}$ will be chosen below. 
 By the strong law of large numbers, for all $j\in\N$,
\begin{equation}
\lim_{m\to\infty}  \frac{1}{m} \sum_{k=1}^m Z^{(k)}_{n_j} = E^{\mu}_0[Z_{n_j}] >0 \quad \mc{P}\text{-a.s.} ,
\end{equation}
from which it follows that there are $(M_j)_{j\in\N}$ satisfying 
\begin{equation}\label{MBC}% M Borel Cantelli
\mc{P}\left(\inf_{m \geq M_j} \sum_{k=1}^m Z_{n_j}^{(k)} \leq 0 \right) \leq \frac{1}{j^2}.
\end{equation}
Next, pick $N_j$ such that
\begin{equation}
\label{sumtrim}
\mc{P}\left(\frac{1}{N_j}\sum_{k=1}^{N_j} Z_{n_j}^{(k)} \leq \frac12 \,E^\mu_0[Z_{n_j}] \right) \leq \frac{1}{j^2},
\end{equation}
and
\begin{equation} \label{ast}
\frac12 N_j \,E^\mu_0[Z_{n_j}] \geq (M_{j+1}+1) n_{j+1}.
\end{equation}
Define
%%
%\begin{equation}
$s(0) := 0$, $s(j) := s(j-1) + N_j$, for $j \in\N$. %, \quad i\in\N,
%\end{equation}
%%
By~\eqref{coupling}, it follows that
\begin{equation}\label{psumable}
\begin{aligned}
&\bb{P}\left(X_{\tau(s(j))}-X_{\tau(s(j-1))} \leq \frac12 N_j \,E^\mu_0[Z_{n_j}] \right)%\\
% &= \mc{P}\left(\frac{1}{N_j}\sum_{k=1}^{N_j} Z_{n_j}^{(k)} \leq \tfrac12 E^\mu_0[Z_{n_j}] \right)
\leq \frac{1}{i^2}.
 \end{aligned}
\end{equation}
Consequently, by the first Borel-Cantelli lemma, it follows that $\bb{P}$-a.s.\ for $j$ sufficiently large,
\begin{equation}\label{eventualincrement}
X_{\tau(s(j))} >X_{\tau(s(j-1))} + \frac12 N_j \,E^\mu_0[Z_{n_j}].
\end{equation}
We  conclude the proof by noting that \eqref{eventualincrement},~\eqref{ast}~\eqref{coupling} imply that
\begin{equation}
\begin{aligned}
\bb{P}\left(X_n = 0 \, \text{ i.o.} \right)& 
%\bb{P}\left(\inf_{n \in [\tau(s(j)),\tau(s(j+1))]} [X_n - X_{s(j)}] \leq -\tfrac12 N_{j}
%E^\mu_0[Z_{n_j}] \, \text{ i.o.} \right)\\[0.2cm]
%&\leq \bb{P}\left(\inf_{n \in [\tau(s(j)),\tau(s(j+1))]} [X_n - X_{s(j)}]
%\leq -(M_{j+1}+1) n_{j+1} \, \text{ i.o.} \right)\\
&\leq \mc{P}\left(\inf_{m \geq M_{j+1}} \sum_{k=1}^m Z_{n_{j+1}}^{(k)} \leq 0 \,\text{ i.o.}\right) = 0,
\end{aligned}
\end{equation}
%%
%where, the first inequality uses~\eqref{eventualincrement}, the second inequality uses~\eqref{ast}, the third inequality uses~\eqref{coupling} and the fact that $X$ makes steps of size $1$ only,
where the equality follows from~\eqref{MBC}. 
%\end{proof}
\QED

%%%%%%%%%%%% SECTION 3 %%%%%%%%%%%%%%%%%%%%%%%%%%%%%%%%%%
\section{Proofs: Mixed fluctuations} \label{sec:Fluc}

%%%%%%%%%%%%%%%%%%%%%%%%%%%%%%%%%%%%%%%%%%%%%%%%%%%%%%%%%%%%
\subsection{Mixed fluctuations in the Sinai-regime}\label{sec:pTSKlp} %proof Theorem Sinai Kesten limit points

%We prove Theorem~\ref{thm:rlpts}.

%\begin{proof}
\Proof{Proof of Theorem~\ref{thm:rlpts}}
The proof is organised into several steps.

\paragraph{Tightness}
Tightness follows from the constant variance scaling in~\eqref{e:sca}, because for any $K>0$, by Chebyshev's inequality,
\begin{equation}
\bb{P}(\abs{\mathfrak{X}_n-\bb{E}[\mathfrak{X}_n]}>K)\leq  \frac{1}{K^2}.
\end{equation}
We identify the limit points. As noted earlier, the sequence $\prt{\lambda_{\tau,n}}_{n\in\N_0}$ of $\ell_2(\N_0)$-unit vectors in \eqref{e:lfp} admits a subsequence $(n_i)_{i\in \N_0}$ for which there is a vector $\lambda_* \in \ell_2(\N_0)$ with $\|\lambda_*\|_2\leq 1$ such that,
\begin{equation}\label{pointwise}
\lim_{i\to\infty} \lambda_{\tau,n_i}(k) = \lambda_*(k)\quad \forall\, k \in \N_0.
\end{equation}
We proceed by comparing $\mathfrak{X}_n-\bb{E}[\mathfrak{X}_n]$ with $V^{\otimes\lambda_{\tau,n}}$. By~\eqref{sinaip} with $p = 2$,
\begin{equation}\label{0vc}%variance convergence
\sigma^2_{0}(n):=\text{Var}\crt{\frac{Z_n}{\sigma^2_0\log^2n}} \xrightarrow[]{} \sigma^2_V.
\end{equation}

\paragraph{Coupling with error term}
Consider a probability space $\prt{S,\mc{S},\mc{P}}$ that is rich enough to include the sequence of random variables  $(V_k)_{k \in \bb{N}_0}$ defined in Section~\ref{subsec:recfl} and an array of random variables $(R^{(k)}_n)_{k,n\in\N_0}$ satisfying:
\begin{itemize}
\item[(A1)] For any $k,n\in \bb{N_0}$ and $x \in \bb{R}$,
\begin{equation}\label{0eqdist}
P^{\mu}_0\left(\frac{Z_n - E^{\mu}_0\crt{Z_n}}{\sigma_{0}(n)}\leq x\right)
=  \mc{P}\left(\sigma_V^{-1}V_k + R^{(k)}_n \leq x\right).
\end{equation}
\item[(A2)] For all $k,n\in\N_0$, $\mc{E}[R^{(k)}_n]= 0$, where $\mc{E}$ stands for expectation w.r.t. $\mc{P}$. 
\item[(A3)] $(V_k,R^{(k)}_n)_{n,k\in \bb N_0}$ are independent in $k$ under $\mc{P}$.
\item[(A4)] $R^{(k)}_n$ vanishes in $L^2$, i.e.,
\begin{equation}\label{e:l2error}
\lim_{n\to\infty} \sup_{k\in\N_0} \mc{E}\left[\left(R^{(k)}_n\right)^2\right] = 0.
\end{equation}
\end{itemize}

\noindent
By~\eqref{e:c3} and~\eqref{0eqdist}, for any bounded continuous function $f$,
\begin{equation}\label{e:xeqdist}
\bb{E}\left[f\left(\mathfrak{X}_n - \bb{E}\crt{\mathfrak{X}_n}\right)\right]
=\mc{E}\left[f\left(V^{\otimes\lambda_{\tau,n}}+\sum_{k=0}^{\ell(n)-1}\lambda_{\tau,n}(k)R^{(k)}_{T_k}\right)\right],
\end{equation}
i.e., $\mathfrak{X}_n - \bb{E}[\mathfrak{X}_n]$ has the same distribution under $\bb{P}$ as the $\lambda_{\tau,n}$-mixture of Sinai-Kesten random variables defined in~\eqref{e:KSmix}, up to an error term that is negligible because of~\eqref{e:l2error}.

The proof proceeds in two parts. First, we remove the error term. Second, we examine the convergence of the main term. 
% Each part consists of several steps.

\paragraph{$\bullet$ Asymptotics of the error terms} 
As a consequence of (A2) and (A3),% we have
\begin{equation}
\begin{aligned}\label{e:0vanish}
&\lim_{J\to\infty}\limsup_{n\to\infty}\,\, 
\mc{E}\crt{\prt{ \sum_{k=0}^{\ell(n)-1}\lambda_{\tau,n}(k) R^{(k)}_{T_k}\Ind{\chv{T_k>J}}}^2}\\
&\qquad = \lim_{J\to\infty}\limsup_{n\to\infty} 
\sum_{k=0}^{\ell(n)-1}\lambda^2_{\tau,n}(k)\, \mc{E} \crt{\prt{R^{(k)}_{T_k}\Ind{\chv{T_k>J}}}^2}=0,
\end{aligned}
\end{equation}
where the last equality follows from~\eqref{e:l2error}. %and $\sum_{k=0}^{\ell(n)-1}\lambda^2_{\tau,n}(k) = 1$. 
For any fixed $J>0$, under $\bb{P}$, $(Y_k \Ind{\chv{T_k \leq J}})_{k\in\N_0}$ is a collection of bounded independent random variables. Thus, by the CLT for i.i.d.~random variables, for any bounded continuous function $f\colon\,\bb{R} \to \bb{R}$ we have
\begin{equation}\label{e:normalmass}
\begin{aligned}
\lim_{n\to \infty}
\Bigg\vert\bb{E}\Bigg[f\Bigg(&\sum_{k = 0}^{\ell(n)-1} \lambda_{\tau,n}(k)\,
\frac{Y_k - \bb{E}\crt{Y_k}}{\sigma_{0}(T_k)}
\Ind{\chv{T_k \leq J}}\Bigg) \Bigg]\\
&-\mc{E}\Bigg[f\Bigg(\Bigg(\sum_{k = 0}^{\ell(n)-1}\lambda^2_{\tau,n}(k)
\Ind{\chv{T_k \leq J}}\Bigg)^{\frac{1}{2}}\Phi\Bigg)\Bigg]\Bigg\vert = 0
\end{aligned}
\end{equation}
with $\Phi$ a standard normal random variable. In view of~\eqref{e:xeqdist}--\eqref{e:normalmass}, to prove Theorem~\ref{thm:rlpts}, it suffices to show that %if $(\lambda_n)_{n\in\N_0}$ is such that $\lambda_n = \lambda_{n,\tau}^{\downarrow}$
%%
%\begin{equation} \label{e:ptconv}
%\sum_{j\in\N_0} \lambda^2_n(j) = 1,
%\quad \lambda_n(j) \geq \lambda_n(j+1),
%\quad
%\lim_{n\to\infty} \lambda_{n}(j) = \lambda_*(j)
%\quad \forall\, j \in \N_0,
%\end{equation}
%%
%then
%%
\begin{equation}\label{e:kesmix}
V^{\otimes\lambda_{\tau,n}}\overset{(d)}{=}V^{\otimes\lambda^{\downarrow}_{n}}\tod V^{\otimes\lambda_*} + a(\lambda_*) \Phi,
\end{equation}
where the equality is due to Lemma~\ref{SKlemma}, whose proof is given in the sequel.% the ordering imposed in \eqref{e:ptconv} causes no loss. %of generality. 

\paragraph{Convergence of mixtures and removal of the error term}
We explain why \eqref{e:kesmix} suffices. Let $f\colon\,\bb{R}\to \bb{R}$ be such that 
$\|f\|_\infty<\infty$ and $\|f'\|_\infty<\infty$. Abbreviate
\begin{equation}\label{not_ease}
\arraycolsep=1.2pt\def\arraystretch{2}
\begin{array}{rlrl}     
\tilde{\mathfrak{X}}_n
&:= \mathfrak{X}_n - \bb{E}[\mathfrak{X}_n],
& \tilde{\mathcal{Y}}_k
&:=\frac{Y_k - \bb{E}\crt{Y_k}}{\sqrt{\var(Y_k)}}, \\
\overline{\lambda_n}^{\,0,J}(k)
&:= \lambda_{\tau,n}(k) \Ind{\chv{T_k <J}},
&\overline{\lambda_n}^{\,J,\infty}(k)
&:= \lambda_{\tau,n}(k)-\overline{\lambda_{n}}^{\,0,J}(k),\\
\overline{\tilde{\mathfrak{X}}}_n^{\,0,J}
&:= \sum_{k=0}^{\ell(n)-1} \overline{\lambda_{n}}^{\,0,J}(k)\tilde{\mathcal{Y}}_k,
&  \overline{\tilde{\mathfrak{X}}}_n^{\,J,\infty}
&:= \tilde{\mathfrak{X}}_n - \overline{\tilde{\mathfrak{X}}}_n^{\,0,J}, \\  
\overline{R_n}^{\,0,J}
&:=  \sum_{k=0}^{\ell(n)-1}\overline{\lambda_{n}}^{J,\infty}(k) R^{(k)}_{T_k},
&\overline{R_n}^{\,J,\infty} 
&:=  R_n -\overline{R_n}^{\,0,J},
 \end{array}
\end{equation}
and note that from (A1) and (A3) we have
\begin{equation}\label{EmcE}
\begin{aligned}
&\bb{E}\crt{f\left(\,\tilde{\mathfrak{X}}_n\right)}= 
\bb{E}\crt{f\Big(\,\overline{\tilde{\mathfrak{X}}}_n^{\,0,J} + \overline{\tilde{\mathfrak{X}}}_n^{\,J,\infty}\Big)} 
 \\
&=\mc{E}\crt{f\left(V^{\otimes\overline{\lambda_n}^{\,0,J}} + \overline{R_n}^{\,0,J}
+ V^{\otimes \overline{\lambda_n}^{\,J,\infty}} + \overline{R_n}^{\,J,\infty}\right)}.
\end{aligned}
\end{equation}
For fixed $J>0$, $\sup_{k\in\N_0}\overline{\lambda_n}^{\,0,J}(k) \to 0$, because if $T_k<J$, then the numerator in $\lambda_{\tau,n}(k)$ remains bounded while the denominator diverges (recall~\eqref{e:lfp}). Hence, by the Lindeberg-Feller theorem for triangular arrays \cite[Theorem 2.4.5]{Dur96},
\begin{equation}\label{LFVJ}
\lim_{n\to \infty}\abs{\mc{E}\crt{f\prt{V^{\otimes\overline{\lambda_n}^{\,0,J}}}} 
- \mc{E}\crt{f\prt{\norm{\overline{\lambda_n}^{\,0,J}}_{2} \Phi}}} = 0.
\end{equation}
Via (A1) and (A3), \eqref{e:normalmass} translates into
\begin{equation}\label{NormalwJ}
\lim_{n\to \infty}\abs{\mc{E}\crt{f\prt{V^{\otimes\overline{\lambda_n}^{\,0,J}} 
+ \overline{R_n}^{\,0,J}}} - \mc{E}\crt{f\prt{\norm{\overline{\lambda_n}^{\,0,J}}_{2} \Phi}}}
= 0.
\end{equation}
Combining \eqref{LFVJ} and \eqref{NormalwJ}, we get
\begin{equation}\label{limrep}
\begin{aligned}
&\lim_{n\to \infty} \bigg\vert\mc{E}\left[f\left(V^{\otimes\overline{\lambda_n}^{\,0,J}}
+ \overline{R_n}^{\,0,J} + V^{\otimes \overline{\lambda_n}^{\,J,\infty}}
+ \overline{R_n}^{\,J,\infty}\right)\right]\\
&\qquad\qquad\qquad - \mc{E}\crt{f\left(V^{\otimes\overline{\lambda_n}^{\,0,J}}
+ V^{\otimes \overline{\lambda_n}^{\,J,\infty}} + \overline{R_n}^{\,J,\infty}\right)}\bigg\vert = 0.
\end{aligned}
\end{equation}
Hence we can estimate
\begin{equation}\label{f0inlaw}
\begin{aligned}
&\limsup_{n\to \infty} \Big\vert\bb{E}
[f(\tilde{\mathfrak{X}}_n)] - \mc{E}\big[f\big(V^{\otimes \lambda_*} + a(\lambda_*)\Phi\big)\big] \Big\vert\\
&= \limsup_{n\to \infty} \abs{\bb{E}
\left[f\left(\overline{\tilde{\mathfrak{X}}}_n^{\,0,J}
+ \overline{\tilde{\mathfrak{X}}}_n^{\,J,\infty}\right)\right]
- \mc{E}[f(V^{\otimes \lambda_n})]}\\
& = \limsup_{n\to \infty}\Bigg\vert\mc{E}
\left[f\left(V^{\otimes\overline{\lambda_n}^{\,0,J}}
+ V^{\otimes \overline{\lambda_n}^{\,J,\infty}}
+ \overline{R_n}^{J,\infty}\right)\right] \\
&\qquad\qquad\qquad\qquad-\mc{E}\left[f\left(V^{\otimes\overline{\lambda_n}^{\,0,J}} + V^{\otimes \overline{\lambda_n}^{\,J,\infty}}\right)\right] \Bigg\vert\\
& \leq \inf_{\delta,J>0} \limsup_{n\to\infty} C_f\left(\delta + \delta^{-2}
\mc{E}\crt{\prt{\overline{R_n}^{\,J,\infty}}^2}\right),
\end{aligned}
\end{equation}
where $C_f$ is a constant that depends on $\|f\|_\infty, \|f'\|_\infty$. The first equality follows from~\eqref{e:kesmix}, the second from~\eqref{EmcE}--\eqref{limrep}, and the inequality from the following standard  bound, which we state for generic random variables $X$ and $H$: 
\begin{equation}\label{l2bound}
\begin{aligned}
&|\mc{E}\crt{f(X + H) - f(X)}| \\
&\leq  |\mc{E}\crt{(f(X+H) - f(X))\Ind{\chv{\abs{H} \leq \delta}}}|
\\
&\qquad + |\mc{E}\crt{(f(X+H) - f(X))\Ind{\chv{\abs{H}>\delta}}}|\\[0.2cm]
&\leq C_f \delta  + C_f \mc{P}\prt{{\abs{H}>\delta}} \leq C_f \delta + C_f \delta^{-2} \mc{E}[{\abs{H}^2}].
\end{aligned}
\end{equation}
%%
%where we use Chebyshev's inequality.
From~\eqref{e:0vanish},
%\[
$\mc{E}[(\overline{R_n}^{\,J,\infty})^2] \to 0$,
%\]
 and hence \eqref{f0inlaw} yields
\begin{equation}
\label{conv*}
\tilde{\mathfrak{X}}_n=\mathfrak{X}_n - \bb{E}[\mathfrak{X}_n] \tod V^{\otimes \lambda_*} + a(\lambda_*)\Phi,
\end{equation}
which is the claim in \eqref{lpoints0} with convergence in distribution.

\paragraph{$L^p$ convergence}
To prove the convergence in $L^p$, note that for any $r\in \N$,
\begin{equation} 
\label{S_2N}  
\begin{aligned}
&\bb{E}\crt{\tilde{\mathfrak{X}}_n^{2r}}=\mc{E}\crt{\prt{\sum_{k=0}^{\ell(n)} \lambda_{\tau,n}(k) \tilde{\mathcal{Y}}_k}^{2r}}%= \sum_{m=1}^{r}\sum_{\substack{e_1,\ldots,e_m\in\N\setminus\{1\} \\ 
%      e_1+\cdots+e_m=2r}} \binom{2r}{e_1 \cdots e_m}\\[0.2cm]
%  &\quad\times\sum_{\ell(n)\geq k_1>\ldots>k_m} \lambda_{\tau,n}^{e_1}(k_1) \times \cdots \times \lambda_{\tau,n}^{e_m}(k_m)
%\EE\crt{\tilde{\mathcal{Y}}_{k_1}^{e_1} \cdots  \tilde{\mathcal{Y}}_{k_m}^{e_m}}\\
%&\leq \sum_{m=1}^r  \sum_{\substack{e_1,\ldots,e_m\in\N\setminus\{1\}\\ 
%e_1+\cdots+e_m=2r}} (2r)^m \sum_{n\geq k_1>\ldots>k_m} 
%\lambda_{\tau,n}^{2}(k_1) \times \cdots \times \lambda_{\tau,n}^{2}(k_m)\,C_{2r}\\
&\leq C_{2r}<\infty,% \sum_{m=1}^r \binom{2(r-1)}{m-1} (2r)^m,
\end{aligned}
\end{equation}
where we use that $\|\lambda_{\tau,n}\|^2_2=1$ and that for all $k \in \N_0$, $\mc{E}[\tilde{\mathcal{Y}}_k] = 0$ and that,by \eqref{sinaip}, $\sup_k\mc{E}[(\tilde{\mathcal{Y}}_k)^{2r}]<C$. 
%\begin{equation}
%\sup_{n\in\N} E^\mu_0\left[\left(\frac{Z_n}{\log^2 n}\right)^{2r}\right] =: C_{2r} < \infty, \qquad r \in \N.
%\end{equation}
The convergence in distribution in \eqref{conv*}, combined with the uniform bound in \eqref{S_2N}, implies \eqref{lpoints0}.%the convergence in $L^p$.

\paragraph{Limit of Sinai-Kesten mixtures}
In order to prove Theorem~\ref{thm:rlpts}, it remains to show~\eqref{e:kesmix}. We divide this part of the proof into steps.

\paragraph{A triangle inequality}
Let $\lambda_n : = \lambda^{\downarrow}_{\tau,n}$.  Note that~\eqref{pointwise} ensures that
\begin{align}
&\sum_{j\in\N_0} \lambda^2_*(j)= 1-a^2 \text{ for some } a\geq 0,\label{e:c00}\\
&\lim_{K\to\infty} \sum_{j>K} \lambda_*^2(j) = 0, \label{e:c01}\\
&\lim_{K\to\infty} \sum_{j=1}^K \abs{\lambda_{n_i}(j)-\lambda_*(j)} = 0, \quad K \in\N, \label{e:c02}\\
&j \geq K \; \Longrightarrow \; \lambda_{n_i}(j) \leq \frac{1}{\sqrt{K}}, \quad K \in \N, \label{e:c03}
\end{align}
where~\eqref{e:c03} follows from $1 \geq \sum_{j=0}^{K-1}\lambda^2_{n_i}(j) \geq K \lambda^2_n(K)$. % because $\lambda_n$ is normalised and ordered. 
Let $(\Phi_j)_{j\in\N_0}$ be a family of i.i.d.\ standard normal random variables defined on the same probability space $\prt{S,\mc{S},\mc{P}}$. Set $\Phi^{\otimes\lambda} := \sum_{j\in\N_0} \lambda(j) \Phi_j$ for a given vector $\lambda\in\ell_2(\bb{N}_0)$, and note that the following isometry is in force (recall \eqref{e:KSmix}):
\begin{equation}\label{isometry}
\mc{E}\left[\abs{{V}^{\otimes\lambda_n}}^2\right]=  \mc{E}\left[\abs{{\Phi}^{\otimes\lambda_n}}^2\right]=\|{\lambda_n}\|_{\ell_2}.
\end{equation}
To prove~\eqref{e:kesmix}, we will show via a truncation that, for any $f\colon\,\R\to\R$ with bounded derivatives up to order three, $\max\chv{\norm{f}_\infty, \norm{f'}_{\infty}, \norm{f''}_{\infty},\norm{f'''}_{\infty}}<\infty,$
%, 
%%
\begin{equation} \label{e:sufclaim}%vector conditions
\begin{aligned}
 \mc{E}\crt{f({V}^{\otimes\lambda_n}) - f({V}^{\otimes\lambda_*} + a\Phi_0)} \to 0.
\end{aligned}
\end{equation}
Indeed, for $\lambda \in \ell_2(\N_0)$ and $k,K \in \N_0$ with $k<K$, set
\begin{equation} \label{e:trunc1}
\lambda^{k,K}(j) :=
\begin{cases}
0, & \text{ if } 0 \leq j < k,\\
\lambda(j), &\text{ if } k \leq j < K,\\
0, & \text{ if } j \geq K,
\end{cases}
\end{equation}
and $\lambda^{K,\infty}(j) := \lambda(j) - \lambda^{0,K}(j)$. By the triangle inequality, for all $K \in \N$,
\begin{equation}\label{e:tsplit}%truncation split
\begin{aligned}
&\abs{\mc{E}\crt{f({V}^{\otimes\lambda_n}) - f({V}^{\otimes\lambda_*} + a\Phi_0)}}\\
&\quad \leq \abs{\mc{E}\crt{f({V}^{\otimes\lambda_n})
- f\left({V}^{\otimes\lambda^{0,K}_n} + \Phi^{\otimes\lambda_n^{K,\infty}}\right)}}\\[.2cm]
&\quad \quad + \abs{\mc{E}\crt{f\left({V}^{\otimes\lambda_n^{0,K}} + \Phi^{\otimes\lambda_n^{K,\infty}}\right)
- f\left({V}^{\otimes\lambda_*^{0,K}}+ a\Phi_0\right)}}\\[.2cm]
%&\quad+ \abs{\mc{E}\crt{f\left({V}^{\otimes\lambda_n^{0,K}} + a\Phi_0\right)
%- f\left({V}^{\otimes\lambda_*^{0,K}} + a\Phi_0\right)}}\\[.2cm]
&\quad \quad+ \abs{\mc{E}\crt{f\left({V}^{\otimes\lambda_*^{0,K}} + a\Phi_0\right)
- f\left({V}^{\otimes\lambda_*} + a\Phi_0\right)}}.
\end{aligned}
\end{equation}
To conclude the proof, we will argue that the three terms in the right-hand side of~\eqref{e:tsplit} can be made arbitrarily small.%ends to zero as $K\to\infty$ and then $n \to \infty$.

\paragraph{Asymptotic negligibility of the last terms in the triangle inequality}
The last two terms in \eqref{e:tsplit} can be treated via~\eqref{l2bound}, by using~\eqref{e:c00}--\eqref{isometry}.  Indeed, the third term in the right-hand side of~\eqref{e:tsplit} tends to zero as $K \to \infty$ due to~\eqref{e:c01}. 
%The second term, for fixed $K$, tends to zero as $n \to \infty$ due to~\eqref{e:c02}. 
For the second term, note that, by~\eqref{e:c00},~\eqref{e:c02} and $\sum_{i\in\N_0} \lambda_n^2(i) = 1$,
\begin{equation} \label{e:L2ncr}% L2 norm convergence of the remainder
\lim_{K \to \infty} \lim_{n \to \infty}\norm{\lambda_n^{K,\infty}}_{2}
= \lim_{K\to\infty} \lim_{n\to\infty} \sum_{i>K} \lambda^2_n(i) = a^2.
\end{equation}
By~\eqref{l2bound}, %and~\eqref{e:L2ncr} 
we get for any $\delta>0$,
\begin{equation}
\begin{aligned}
&\bigg\vert\mc{E}\crt{f\left({V}^{\otimes\lambda_n^{0,K}}
+ \Phi^{\otimes\lambda_n^{K,\infty}}\right)} -  \mc{E}\crt{f\left({V}^{\otimes\lambda_*^{0,K}} + a\Phi_0\right)}\bigg\vert\\
&\quad=\bigg\vert\mc{E}\crt{f\left({V}^{\otimes\lambda_n^{0,K}}
+\norm{\lambda_n^{K,\infty}}_{2}^{\frac12}\Phi_0\right)}
-\mc{E}\crt{f\left({V}^{\otimes\lambda_*^{0,K}} + a\Phi_0\right)}\bigg\vert\\
&\quad\leq C_f \delta + C_f \delta^{-2} \; \mc{E}\bigg[\abs{\left(a - \norm{\lambda_n^{K,\infty}}_2^2\right)\Phi + V^{\lambda^{0,K}_n} - V^{\lambda^{0,K}_*}}\bigg]
\end{aligned}
\end{equation}
and therefore, by~\eqref{e:c02} and~\eqref{e:L2ncr}, the second term vanishes as one takes $K \to \infty$ and then $n\to \infty$.
%\begin{equation}
%\lim_{K\to \infty}\lim_{n\to \infty}\bigg\vert\mc{E}\crt{f\left({V}^{\otimes\lambda_n^{0,K}}
%+ \Phi^{\otimes\lambda_n^{K,\infty}}\right)} -  \mc{E}\crt{f\left({V}^{\otimes\lambda_n^{0,K}} + a\Phi_0\right)}\bigg\vert = 0.
%\end{equation}  
To show that the first term in the right-hand side of~\eqref{e:tsplit} vanishes as well, we prove a bound that is independent of $n$ by using a classical argument in the spirit of the Lindeberg-Feller theorem (see~\cite[Theorem 2.4.5]{Dur96}).

\paragraph{Interpolation of random variables}
We consider
\begin{equation}
W_{K,n}(M) := V^{\otimes\lambda_n^{0,K\wedge M}} + \Phi^{\otimes\lambda_n^{K,M}} 
+ {V}^{\otimes\lambda_n^{M,\infty}}
\end{equation}
obtained from $V^{\otimes\lambda_n}$ after replacing $\sigma_V^{-1}{V}_j$ by $\Phi_j$ for $K< j \leq M$ in \eqref{e:KSmix}. Note that, by~\eqref{e:trunc1}, $W_{K,n}(M) = {V}^{\otimes\lambda_n}$ for $M \leq K$, and also that, for fixed  $K, n \in \N_0$, $W_{K,n}(M) \todtwo   W_{K,n}(\infty):={V}^{\otimes\lambda^{0,K}_n} + \Phi^{\otimes\lambda_n^{K,\infty}}$% in $L^2$ as $M \to \infty$
. 
With these auxiliary random variables, we see that in order to show that the first term in the right-hand side of~\eqref{e:tsplit} vanishes, we must prove that
\begin{equation}\label{Wconverge}
\limsup_{K\to \infty} \limsup_{n\to \infty} \Big|\mc{E}\big[f\big(W_{K,n}(K)\big) - f\big(W_{K,n}(\infty)\big)\big]\Big| = 0.
\end{equation}
We will show that
\begin{equation}\label{e:sufconv}
\limsup_{K\to \infty}\limsup_{n\to \infty}  \sum_{M>K} 
\Big|\mc{E}\big[f\big(W_{K,n}(M)\big) - f\big(W_{K,n}(M+1)\big)\Big| = 0,
\end{equation}
which in particular implies \eqref{Wconverge}. 

\paragraph{Bound by Taylor expansion}
For the proof of \eqref{e:sufconv} %we argue as follows. 
define, for $M\geq K$,
\begin{equation}\label{Wstar1}
W_{K,n}^*(M) := W_{K,n}(M) - \sigma_V^{-1}\lambda_n(M)V_{M} .
\end{equation}
Note that  $W_{K,n}^*(M)$ is independent of $\Phi_{M}$ and ${V}_{M}$, and that
\begin{equation}\label{Wstar2}
W_{K,n}(M+1)   = W_{K,n}^*(M) + \lambda_n(M) \Phi_{M}
\end{equation}
Consider the Taylor expansion of $f$ up to second order,
\begin{equation} \label{taylorLindeberg}
\begin{aligned}
f(x + h) = f(x) + f'(x)h + \frac12 f''(x)h^2 + C_f (\abs{h}^2\wedge \abs{h}^3). 
\end{aligned}
\end{equation}
Note that, for any $\gep>0$, $\abs{h}^2\wedge \abs{h}^3\leq \abs{h}^2\Ind{\{\abs{h}>\gep\}}+ \abs{h}^3$, and that for $j \in \N_0$, 
\begin{equation}\label{meanvar}
\mc{E}[\Phi_j] = 0, \quad \mc{E}[{V}_j] = 0, \quad \mc{E}\crt{\Phi_j^2} = \mc{E}\crt{\prt{\sigma_V^{-1}V_j}^2} = 1.
\end{equation}
Use \eqref{Wstar1} and \eqref{Wstar2}, respectively, to expand $f(W_{K,n}(M)) -f(W_{K,n}(M+1))$ with the help of~\eqref{taylorLindeberg}, which together with the triangle inequality yield
\begin{equation}\label{end}
\begin{aligned}
&\left|\mc{E}\left[f\left(W_{K,n}(M)\right) - f\left(W_{K,n}(M+1)\right)\right]\right|\\
&\leq C_f \left( \mc{E}\crt{\abs{\lambda_n(M) {V}_{M}}^3 + \abs{\lambda_n(M) \Phi_{M}}^3}\right.\\
&\quad + \mc{E} \left[ \abs{\lambda_n(M) \sigma_V^{-1}{V}_{M}}^2
\Ind{\left\{\abs{\lambda_n(M) \sigma_V^{-1}V_{M}}^2>\gep\right\}}
\right.\\
&\qquad\qquad\qquad\left.\left.+ \abs{\lambda_n(M) \Phi_{M}}^2 \Ind{\left\{\abs{\lambda_n(M) \Phi_{M}}^2>\gep\right\}} \right] \right).
\end{aligned}
\end{equation}
Next, note that H\"older's inequality and Markov's inequality imply that 
\begin{equation} \label{Unifbound}
\begin{aligned}
 \mc{E} &\crt{\abs{\lambda_n(M) \sigma_V^{-1}{V}_{M}}^2
\Ind{\left\{\abs{\lambda_n(M)\sigma_V^{-1} V_{M}}^2>\gep\right\}}}\\
&\leq  \lambda_n(M)^2 \,
\mc{E}\crt{\abs{\sigma_V^{-1}V_M}^4}^{\frac{1}{2}} \mc{P}\prt{{\abs{\lambda(M)\sigma_V^{-1}V_{M}}^2>\gep }}^{\frac{1}{2}}\\
&\leq \lambda_n(M)^2 \mc{E}\crt{\abs{\sigma_V^{-1}V_1}^4}^{\frac{1}{2}} 
\frac{\lambda(M)\mc{E}\crt{\abs{\sigma_V^{-1}V_1}^2}^{\frac{1}{2}}}{\sqrt{\gep}}%\\
\\
&\leq \lambda_n(M)^3\frac{\mc{E}\crt{\abs{\sigma_V^{-1}V_1}^4}^{\frac{3}{4}}}{\sqrt{\gep}}.
\end{aligned}
\end{equation}
Since $\mc{E}[\abs{V_1}^4]<\infty$, \eqref{e:sufconv} follows from~\eqref{end} and \eqref{Unifbound} via an analogous argument as for the terms involving $\Phi_M$, because, for some $C>0$ independent of $K$ and $n$,
\begin{equation} \label{Unifbound2}
\begin{aligned}
&\sum_{M>K} \Bigg(\mc{E} \crt{\abs{\lambda_n(M)\sigma_V^{-1} {V}_{M}}^2
\Ind{\left\{\abs{\lambda_n(M) V_{M}}^2>\gep\right\}}}\\
& \qquad\qquad\qquad\qquad\qquad+ \mc{E}\crt{\abs{\lambda_n(M) \sigma_V^{-1}{V}_{M}}^3}\Bigg)\\
&\qquad \leq C \sum_{M>K} \lambda_n^3(M)
\leq C \sup_{M>K}\lambda_n(M) \sum_{M>K}\lambda_n^2(M) \leq C \frac{1}{\sqrt{K}},
\end{aligned}
\end{equation}
where the last inequality follows from~\eqref{e:c03} and $\sum_{M>K}\lambda_n^2(M)\leq 1$. 
%\end{proof}
%%
\QED
%%%%%%%%%%%%%%%%%%%%%%%%%%%%%%%%%%%%%%%%%%%%%%%%%%%%%%%%%%%%
\subsection{Characterisation of Sinai-Kesten mixtures}\label{SKmix}

%We prove Lemma~\ref{SKlemma}.

%\begin{proof}
\Proof{Proof of Lemma~\ref{SKlemma}}
To prove Lemma~\ref{SKlemma}, it is equivalent to  prove 
\begin{align}
\lambda \sim \lambda' \quad &\Longrightarrow \quad V^{\otimes \lambda}\overset{(d)}{=}V^{\otimes \lambda'},\label{suf}\\
[\lambda] \neq [\lambda'] \quad &\Longrightarrow \quad V^{\otimes \lambda}\overset{(d)}{\neq}V^{\otimes \lambda'}.\label{nec}
\end{align}

\paragraph{Proof of \eqref{suf}}
Let $\lambda$ be a vector with finitely many non-zero entries. In view of the i.i.d.\ property of the random variables $\prt{V_j}_{j\in \N_0}$, we have that
\begin{equation}\label{finiteeq}
\lambda\sim\lambda'  \quad \Longrightarrow \quad V^{\otimes \lambda}\overset{(d)}{=}V^{\otimes \lambda'}.
\end{equation}
For general $\lambda \in \ell_2(\N_0)$, let $\sigma,\sigma'\colon\,\N_0 \in \N_0$ be such that 
%\begin{equation}\label{sigmas}
$\lambda\sigma(i) = \lambda^{\downarrow}(i)$, and $\lambda'\sigma'(i) = \lambda^{\downarrow}(i)$.
%\end{equation}
Define 
\begin{equation}\label{sigmat}
\lambda^{\sigma, 0,k}(j) = \begin{cases}
\lambda(j), &\text{ if } j \in \chv{\sigma(i)\colon i <k},\\
0, & \text{ else.}
\end{cases}
\end{equation}
As in \eqref{finiteeq},
%\begin{equation}\label{feqdist}
$V^{\otimes \lambda^{\sigma, 0,k}}\overset{(d)}{=}V^{\otimes \lambda'^{\sigma', 0,k}}$.
%\end{equation}
By~\eqref{l2bound}, for any $\delta >0$,
\begin{equation}\label{ffull}
\left\vert\mc{E} \crt{f(V^{\lambda^{\sigma, 0,k}})} - \mc{E}\crt{f(V^{\lambda})}\right\vert \leq C_f \delta + C_f \delta^{-2}\norm{\lambda -  \lambda^{\sigma, 0,k}}_2^2.% \mc{E} \crt{V^2}.
\end{equation}
Since $\norm{\lambda^{\sigma, 0,k}}_2 \to \norm{\lambda}_2$, the claim follows.

\paragraph{Proof of \eqref{nec}}
We may assume without loss of generality that
%%
%\begin{equation}\label{hipla}
$\lambda = \lambda^{\downarrow}$, $\lambda'= \lambda'^{\downarrow}$ %\geq \lambda(j+1), \quad \lambda'(j)\geq \lambda'(j+1),%\quad \lambda \neq \lambda',
%\end{equation}
%%
and that there is a $j_0\in \bb{N_0}$ for which
\begin{equation}\label{e:step2eqdist}
\lambda(j)= \lambda'(j) \quad \forall\,0 \leq j < j_0, \qquad \lambda(j_0)>\lambda'(j_0).
\end{equation}
Let $t \mapsto \mc{L}_X(t):=\mc{E}[e^{tX}]$ be the moment generating function of a random variable $X$.
%\[
%\] 
To show that the distributions of $V^{\otimes\lambda}$ and $V^{\otimes\lambda'}$ are different, by \cite[Theorem 30.1]{Bil95} we must show that the moment generating function of $V^{\otimes \lambda}$ is finite in a neighbourhood of the origin and
\begin{equation}\label{Laplaceneq}
\exists\, t \in \bb{R}\colon \mc{L}_{V^{\otimes\lambda}}(t) \neq \mc{L}_{V^{\otimes\lambda'}}(t).
\end{equation}
The proof proceeds in three steps. First, we analyse the Laplace transform of $V^{\otimes\lambda}$ for general $\lambda \in \ell_2(\N_0)$. Second, we prove~\eqref{Laplaceneq} when $j_0= 0$ in~\eqref{e:step2eqdist}. Third, we show~\eqref{Laplaceneq} when $j_0> 0$ by reducing it to the case $j_0=0$.
   
\paragraph{Laplace transform of $V^{\otimes \lambda}$}  
Abbreviate $f(t):=\mc{L}_{\sigma_V^{-1}V_1}(t)$ and note that
\begin{equation}\label{Laplace}
\begin{aligned}
\mc{L}_{V^{\otimes\lambda} }(t)
&:=\mc{E}\crt{e^{t V^{\otimes\lambda} }}
= \prod_{j\in\N_0} \mc{E}\crt{e^{t \lambda(j)\sigma_V^{-1} V_j }}
= \prod_{j\in\N_0} f(\lambda(j) t).
\end{aligned}
\end{equation}
By~\eqref{densityofV},
\begin{equation}\label{Lbehave}
\begin{aligned} 
\abs{t}<\tfrac18 \pi^2\sigma_V &\quad \Longrightarrow \quad \abs{f(t)}<\infty,\\
t \to \tfrac18 \pi^2 \sigma_V &\quad \Longrightarrow \quad f(t) \to \infty.
\end{aligned} 
\end{equation}
Furthermore, by Morera's theorem~\cite[Theorem 5.1]{SteSha10}, $t\mapsto f(t)$ is holomorphic on the open disk
\begin{equation}\label{radius}
B := \left\{t\in \mathbb{C}\colon \abs{t}<\tfrac18 \pi^2\sigma_V\right\}.
\end{equation}
Therefore Taylor expansion of $f$ on $B$ around $0$ gives that
\begin{equation}\label{Maclaurin}
f(t) = 1 + \tfrac12t^2 + t^4 g(t),
\end{equation} 
with $g$ a holomorphic function on $B$. From~\cite[Proposition 3.2]{SteSha10}, the finiteness of the $\ell_2$-norm of $\lambda$, and~\eqref{Maclaurin}, we deduce that $t\mapsto \mc{L}_{V^{\otimes \lambda}}(t)$ is holomorphic on the open disk
\begin{equation}\label{radiuslambda}
B(\lambda) := \left\{t\in \mathbb{C}\colon \abs{t}<\tfrac{\pi^2\sigma_V}{8\lambda(0)}\right\}.
\end{equation}

\paragraph{Case $j_0 = 0$} 
From~\eqref{Laplace} and~\eqref{Lbehave},
%\begin{equation}
$\mc{L}_{V^{\otimes\lambda} }(t) \to \infty$ as  $ t \to \frac{\pi^2\sigma_V}{8\lambda(0)}$,
%\end{equation}
while,  $\lambda(0)>\lambda'(0)$ implies $B(\lambda) \subsetneq B(\lambda')$,  
\begin{equation}
\sup_{t \in B(\lambda)} \abs{\mc{L}_{V^{\otimes\lambda'}}(t)} < \infty.
\end{equation}
%%
%Therefore there exists a $t<\frac{\pi^2}{8\lambda(0)}$ such that $\abs{\mc{L}_{V^{\otimes\lambda} }(t) } \neq  \abs{\mc{L}_{V^{\otimes\lambda'} }(t)}$, 
from which~\eqref{Laplaceneq} follows.

\paragraph{Case $j_0 \in \N$}  
Recall the notation in \eqref{e:trunc1}. By \eqref{e:step2eqdist}, we have $\lambda^{0,j_0} =\lambda'^{0,j_0}$. Suppose that 
\begin{equation}
\label{suppose}
V^{\otimes\lambda} \overset{(d)}{=}V^{\otimes\lambda'}.
\end{equation} 
Since $V^{\otimes\lambda} = V^{\otimes\lambda^{0,j_0}} + V^{\otimes\lambda^{j_0,\infty}}$ \text{and} $V^{\otimes\lambda'} = V^{\otimes\lambda^{0,j_0}} + V^{\otimes\lambda'^{j_0,\infty}}$, taking the Laplace transform of both random variables and using the independence, we get that
\begin{equation}
V^{\otimes\lambda^{j_0,\infty}} \overset{(d)}{=}V^{\otimes\lambda'^{j_0,\infty}},
\end{equation}
which is a contradiction.% impossible because of Lemma~\ref{SKlemma} and~\eqref{e:step2eqdist}.
%\end{proof}
\QED
%%%%%%%%%%%%%%%%%%%%%%%%%%%%%%
\subsection{Identification of the limit points}

%We prove Corollary~\ref{cor:rRecThree}(a)--(b).

%\begin{proof}
%$\mbox{}$
\Proof{Proof of Corollary~\ref{cor:rRecThree}{\rm (a)--(b)}}	
%\medskip\noindent
(a): To prove the necessity of the condition on $\lambda_{\tau,\tau(k)}(k)$, suppose that
%%
%\begin{equation}
$\limsup_{k\to\infty}\lambda_{\tau,\tau(k)}(k) =c>0$,
%\end{equation}
%%
and take a subsequence $\prt{k_i}_{i\in \N}$ such that
%%
%\begin{equation}
$\lambda_{\tau,\tau(k_i)}(k_i) \to c$ and 
%\qquad 
$\lambda^{0\downarrow}_{\tau,\tau(k_i)}(j) \to \lambda_*(j) %\quad \forall \,j \in \N_0
$
%\end{equation}
%%
for any $j \in \N_0$ and for some $ \lambda_*\in \ell_2(\N_0)$. Next, take a subsequence $\prt{n_i}_{i\in \N}$ for which
\begin{equation}\label{bbig}
\lim_{i\to\infty} \lambda_{\tau,n_i}(0) > \frac{c}{2} \quad \text{and} \quad \lim_{i\to\infty} \lambda^{0\downarrow}_{\tau,n_i}(j) = \lambda'_*(j)
\quad \forall \,j \in \N_0
\end{equation}
for some $ \lambda'_*\in \ell_2(\N_0)$. By Theorem \ref{thm:rlpts},
%%
%\begin{equation}
$\mathfrak{X}_{\tau(k_i)} - \bb{E}[\mathfrak{X}_{\tau(k_i)}] \tod V^{\otimes \lambda_*}$, and %\qquad
$\mathfrak{X}_{n_i} - \bb{E}[\mathfrak{X}_{n_i}] \tod V^{\otimes \lambda'_*}$.
%\end{equation}
%%
To conclude the proof it suffices to show that  $V^{\otimes \lambda_*}$ and $V^{\otimes \lambda_*}$ have different distributions. But this follows from Lemma~\ref{SKlemma}, for which we argue next that $[\lambda'_*]\neq[\lambda_*]$.
	
Note that, for any $\gep >0$ and for $n_i$ large enough, 
%\begin{equation}\label{supbound}
$\sup_j \lambda^2_{\tau,n_i}(j) < c^2+ \gep$.
%\end{equation}
Furthermore, by~\eqref{bbig}, for $n_i$ large enough,
\begin{equation}\label{varratio}	
\frac{\var(X_{\tau(\ell(n_i)-1)})}{\var(X_{n_i})} \leq 1 - \frac{c^2}{4}.
\end{equation}
Therefore, % combining~\eqref{bbig} and~\eqref{supbound} we get that for $n_i$ large enough and
for  $\gep < \frac{c^4}{4-c^2}$,
\begin{equation}
\begin{aligned}
\sup_{j \in \N} \lambda'^2_{\tau,n_i}(j) &= \sup_{j \in \N} \frac{\var(X_{\tau(\ell(n_i)-1)})}{\var(X_{n_i})}\,
\lambda^2_{\tau,\tau(\ell(n_i)-1)}(j)\\
 & \leq \left(1 - \frac{c^2}{4}\right)(c^2 + \gep )< c^2,
\end{aligned}
\end{equation}
and therefore $\sup_{i\in \N_0}  \lambda'_*(i)< \sup_{i\in \N_0} \lambda_*(i)$. For the reverse implication, by Theorem~\ref{thm:rlpts}, it suffices to show that $\lambda_{\tau,\tau(k)}(k) \to 0$ implies that for all $i \in \N_0$,
%%
%\begin{equation}\label{convgoal}
$\lim_{n\to \infty} \lambda^{0\downarrow}_n(i)=  0$. %\, i \in \bb{N}_0$.
%\end{equation}
%%
To prove this, we will show that
\begin{equation}\label{e:incrementcontrol}
\lim_{n\to\infty}\sup_{i\in \N_0}\lambda_{\tau,n}(i) =0.
\end{equation}
%%
%To see~\eqref{e:incrementcontrol}, 
Note that for  $i >0$,
\begin{equation}\label{ei}
\lim_{n\to\infty}\lambda_{\tau,n}(i) = 0,\qquad\sup_{n\in\N} \lambda_{\tau,n}(i) \leq \lambda_{\tau,\tau(i)}(i).
\end{equation}
By~\eqref{ei}, it follows that
\begin{equation}
\begin{aligned}
\lim_{n\to \infty}\sup_{i\in \N}\lambda_{\tau,n}(i) &\leq \lim_{J\to \infty}\lim_{n\to \infty} \sup_{i>J} \lambda_{\tau,n}(i)\\
&\leq\lim_{J\to \infty} \sup_{k>J} \lambda_{\tau,\tau(k)}(k) = \limsup_{k\to \infty}\lambda_{\tau,\tau(k)}(k) =0.
\end{aligned}
\end{equation}
As to $i = 0$, define $s_n= \sigma^2_{0}(\bar{T}^n)/\sigma^2_{0}(T_{\ell(n)})$ and note that, by~\eqref{0vc}, there is a constant $C\in (1,\infty)$ such that $s_{n} < C$. Therefore, since $\bar{T}^n \leq T_{\ell(n)}$ and $x \mapsto \frac{x}{x+y}$ is increasing on $\R_+$ for $y>0$,
%replacing $\bar{T}^n$ by 
we get
\begin{equation}\label{bscontrol}
\begin{aligned}
\lambda^2_{\tau,n}(0)  &= \frac{\sigma^2_{0}(\bar{T}^n)\log^2 \bar{T}^n  }
{ \var(X_{\tau(k)}) + \sigma^2_{0}(\bar{T}^n)\log^2 \bar{T}^n} 
%\\
%&\leq \frac{C\sigma^2_{0}(T_{\ell(n)})\log^2 T_{\ell(n)}  }{ \var(X_{\tau(k)}) + C\sigma^2_{0}(T_{\ell(n)})\log^2 T_{\ell(n)}}
%\\
&\leq C \lambda_{\tau,\tau(\ell(n))}(\ell(n))\xrightarrow[]{}0,
\end{aligned}
\end{equation}
from which \eqref{e:incrementcontrol} follows.

\noindent (b) As in the proof of (a), we examine the sequence of vectors $(\lambda^{0\downarrow}_{\tau,\tau(k)})_{k \in \N}$ and prove that it converges to $\lambda_{q}$ with
%%
%\begin{equation}
$q = \lim_{k\to\infty} \lambda_{\tau,\tau(k)}(k)$.
%\end{equation}
%%
Abbreviate $q_k := \lambda_{\tau,\tau(k)}(k) $ and note that
%%
%\begin{equation}\label{vsum}
$(1 - q_k)^2\var (Y_k) = q_k^2 \sum_{i=1}^{k-1} \var (Y_i)$.
%\end{equation}
%%
Adding $(1-q_k^2) \sum_{i=1}^{k-1} \var (Y_i)$ on both sides%~\eqref{vsum}
, we get
\begin{equation}\label{indq}
\left(1 - q_k^2\right)\sum_{i=1}^k \var (Y_i) =  \sum_{i=1}^{k-1} \var (Y_i).
\end{equation}
Since, $\var (Y_{k-j}) = q_{k-j}^2 \sum_{i=1}^{k-j} \var (Y_i)$, recursively applying~\eqref{indq}, yields to
%%
%\begin{equation} \label{double}
%\begin{aligned}
$\var (Y_{k-j}) %&= q_{k-j}^2 \sum_{i=1}^{k-j} \var (Y_i) = q_{k-j}^2 (1 - q_{k-j+1}^2) \sum_{i=1}^{k-j+1} \var (Y_i)\\
%&= \ldots 
= q_{k-j}^2 \prod_{i=1}^{j} (1-q_{k-j+i}^2) \sum_{i=1}^k \var (Y_i)$,
%\end{aligned}
%\end{equation}
%%
which implies that
\begin{equation}
\lambda^2_{\tau,\tau(k)}(k-j) = \frac{\var (Y_{k-j})}{\var (X_{\tau(k)})} = q_{k-j}^2 \prod_{i=1}^{j} \left(1-q_{k-j+i}^2\right).
\end{equation}
For any $k \in \N$, $\lambda_{\tau,\tau(k)}(0) = 0$. As for $j \in \N$, since 
%\[
$q_k \to q >0$,
%\]
%%
\begin{equation}\label{e:bpure}
\lim_{k\to \infty} \left(\lambda^{0\downarrow}_{\tau,\tau(k)}(j)\right)^2=\lim_{k\to\infty}\lambda^2_{\tau,\tau(k)}(k-j+1)
= q^2 \left(1 - q^2\right)^{j-1},
\end{equation}
and since $\norm{\lambda_q}_{2} = 1$, the first part of (b) follows from a direct application of Theorem~\ref{thm:rlpts}. As to the second part of (b), if $\lambda_{\tau,n_i}(0) \to w$, then
\begin{equation}\label{e:boundweight}
1 - w^2= \lim_{k\to\infty}  \frac{\var(X_{\tau(\ell(n_k)-1)})}{\var (X_{n_k})}.
\end{equation}
From~\eqref{e:bpure} and~\eqref{e:boundweight}, for any $i \in \N$,
\begin{equation}\label{e:bi}
\begin{aligned}
\lim_{i\to\infty} \left(\lambda^{0\downarrow}_{\tau,n_k}(j)\right)^2 &= \lim_{i \to \infty} \lambda^2_{n_i,\tau}(\ell(n_i)-j)\\
&=\lim_{i\to\infty}\frac{\var\prt{X_{\tau(\ell(n_i)-1)}}}{\var (X_{n_i})} \, \lambda^2_{\tau,\tau(\ell(n_i)-1)}(\ell(n_i)-j)\\
&=\left(1-w^2\right)q^2\left(1 - q^2\right)^{i-1}.
\end{aligned}
\end{equation}
Then, By Theorem~\ref{thm:rlpts}, %%
%\begin{equation}
$V^{\otimes{\lambda}_{n_i,\tau}} \tod w\,\sigma_V^{-1} V_0
+ (1 - w^2)^{\frac12} V^{\otimes \lambda_q}$.
%\end{equation}
%%
%\end{proof}
\QED
%%%%%%%%%%%%%%%%%%%%%%%%%%%%%%%%%%%%%%%%%%%%%%%%%%%%%%%
\subsection{Divergence of cooling maps and crossovers}\label{ss:threeexamples}

In this section we examine the different classes of cooling maps presented
in Section \ref{4examples} and identify the corresponding limit laws.

\begin{proof}
We treat the different examples one by one.
	
\paragraph{{\bf (Ex.3)} Polynomial cooling}
When $k^{-\beta} T_k  \to B$, we have
\begin{equation}\label{npoly}
\lim_{n\to\infty} \frac{n}{B\ell(n)^{\beta+1}} = 1.
\end{equation}
Furthermore, by~\eqref{0vc} we get that
\begin{equation}\label{polypiece}
\lim_{k\to\infty}\frac{\var (Y_k)}{\sigma_V^2\sigma_0^4 \log^4(B k^\beta)} = 1.
\end{equation}
Since $\tfrac{\sum_{k=1}^\ell\log^4(B k^\beta)}{\beta^4 \ell \log^4 \ell} \to 1$, it follows that
\begin{equation}\label{polyvark}
\begin{aligned}
\lim_{\ell\to\infty}\frac{\var (X_{\tau(\ell)})}{\sum_{k = 1}^{\ell}\sigma_V^2\sigma_0^4 \log^4(B k^\beta)}
=\lim_{\ell\to \infty} \frac{\var (X_{\tau(\ell)})}{\sigma_V^2\sigma_0^4 \beta^4\ell\log^4\ell}=1.
\end{aligned}
\end{equation}
%%
%From~\eqref{polypiece} and \eqref{polyvark} 
It follows that
%%
%\begin{equation}\label{polygauss}
$\lambda^2_{\tau,\tau(\ell)}(\ell) \to 0$,
%and  
% \qquad 
$\tfrac{\var (X_n)}{\var(X_{\tau(\ell(n)-1)})} \to 1$,
%\end{equation}
%%
and by~\eqref{polyvark}
%Combining this with~\eqref{polyvark} we obtain that 
%%
\begin{equation}\label{polyvarn}
\begin{aligned}
%\lim_{n\to\infty}\frac{\var (X_n)}{\sum_{i = 1}^{\ell(n)-1}\sigma_V^2\sigma_0^4 \log^4(B i^\beta)}
%=
\frac{\var (X_{n})}{\sigma_V^2\sigma_0^4 (\frac{\beta}{\beta+1})^4
\prt{\frac{n}{B}}^{\frac{1}{\beta + 1}} \log^4n}\to 1.
\end{aligned}
\end{equation}
Finally, note that~\eqref{polyvarn} implies
\begin{equation}\label{inlawpoly}
\frac{X_n-\bb{E}[X_n]}{\sigma_0^2 n^{\frac{1}{2(\beta + 1)}} \log^2 n}
= \alpha_n \big(\mathfrak{X}_n-\bb{E}[\mathfrak{X}_n]\big), \end{equation}
 with $\alpha_n \to (\tfrac{\beta}{\beta+1})^2 B^{-\frac{1}{2(\beta + 1)}}$.
By %From~\eqref{polygauss},~\eqref{inlawpoly} and 
Corollary~\eqref{cor:rRecThree}(a), it follows that
\begin{equation}
\frac{X_n-\bb{E}[X_n]}{\sigma_0^2 n^{\frac{1}{\beta + 1}}
\log^2n} \todp \prt{\frac{\beta}{\beta+1}}^2 B^{-\frac{1}{2\beta+1}}\Phi.
\end{equation}

\paragraph{{\bf (Ex.4)} Exponential cooling}
When
%The condition 
$k^{-1}\log T_k \to c \in (0,\infty)$, 
%implies that
%%
\begin{equation}\label{eg1pre}
\ell^{-1}\log \tau(\ell) \to c.
\end{equation}
Since $\tfrac{\sum_{k = 1}^\ell k^4}{\ell^5}   \to \tfrac15$, via~\eqref{0vc} %and~\eqref{eg1pre} 
it follows that
\begin{equation}\label{asymp1}
\begin{aligned}
\lim_{\ell\to\infty}\frac{\var (X_{\tau(\ell)})}{\sum_{k = 1}^{\ell}\sigma_V^2\sigma_0^4k^4}
=\lim_{\ell\to \infty} \frac{\var (X_{\tau(\ell)})}{\sigma_V^2\sigma_0^45^{-1}c^{-5}\log^5 \tau(\ell)}=1,
\end{aligned}
\end{equation}
that
%From~\eqref{asymp1} 
%%
%\begin{equation}\label{eg1pos}
$ \lambda^2_{\tau,\tau(\ell)}(\ell) \to 0$, 
%\qquad 
and that $\tfrac{\var (X_{\tau(\ell(n)-1)})}{\var (X_n)} \to 1$.
%\end{equation}
%%
From~\eqref{sufficient_fast} we obtain that 
%%
%\begin{equation}
%\label{vanishmean1}
$\bb{E}[\mathfrak{X}_n] \to 0$.
%\end{equation}
%%
Finally, note that %these estimates%~\eqref{asymp1} and~\eqref{vanishmean1} 
%imply
%%
\begin{equation}\label{inlaw1}
\frac{X_n}{\frac{1}{\sqrt{5c^5}}\sigma_V\sigma_0^2\log^{\frac{5}{2}}n}
= \alpha_n \big(\mathfrak{X}_n-\bb{E}[\mathfrak{X}_n]\big)
+ \beta_n 
\end{equation}
with $\alpha_n \to 1$,  and $\beta_n \to 0$.
By%~\eqref{eg1pos},
~\eqref{inlaw1} and Corollary~\eqref{cor:rRecThree}(a), %it follows that
\begin{equation}
\frac{X_n}{\frac{1}{\sqrt{5c^5}}\sigma_V\sigma_0^2\log^{\frac{5}{2}}n} \todp \Phi.
\end{equation}

\paragraph{{\bf (Ex.5)} Double exponential cooling}
When
%The condition 
$k^{-1}\log\log T_k \to c \in (0,\infty)$, %implies that
\begin{equation}\label{eg2pre}
\frac{\tau(\ell)}{T_\ell} \to 1, \qquad
\quad \frac{\sum_{k = 1}^\ell\log^4T_k }{\sum_{k = 1}^\ell e^{4ck}} \to 1.
\end{equation}
From~\eqref{0vc} it follows that
%%
%\begin{equation}\label{asymp2}
$\tfrac{  \var (X_{\tau(\ell)})}{\sigma_V^2\sigma_0^4e^{4c\ell}(1 - e^{-4c})^{-1}} \to 1$
%\end{equation}
%%
and therefore
\begin{equation}
\label{eg2pos}
\lambda^2_{\tau,\tau(\ell)}(\ell) \to \tfrac{e^{4c}-1}{e^{4c}} = q_c^2.
\end{equation}
Note that
\begin{equation}
\label{vardouble}
\lim_{\ell\to\infty} \frac{\var (X_{\tau(\ell)})}{\sigma_V^2\sigma_0^4\log^4\tau(\ell)}
= \lim_{\ell\to\infty} \frac{ \sum_{k = 1}^\ell \log^4 T_k}{\log^4 T_\ell}\frac{\log^4 T_\ell}{\log^4\tau(\ell)} = q_c^{-2}.
\end{equation}
Combining \eqref{eg2pos} and \eqref{vardouble} with Corollary~\ref{cor:rRecThree}(b), we conclude that
\begin{equation}\label{inlaw2}
\begin{aligned}
\frac{X_{\tau(\ell)}}{\sigma_0^2\log^2\tau(\ell)} 
&= \frac{\sqrt{\var (X_{\tau(\ell)})}}{\sigma_0^2\log^2\tau(\ell)}\frac{X_{\tau(\ell)}}{\sqrt{\var (X_{\tau(\ell)})}}\\
&= \sigma_V\frac{\sqrt{\var (X_{\tau(\ell)})}}{\sigma_V\sigma_0^2\log^2\tau(\ell)} 
\mathfrak{X}_{\tau(\ell)} \todp  \sigma_V q_c^{-1}V^{\otimes \lambda_{q_c}}.
\end{aligned}
\end{equation}

\paragraph{{\bf (Ex.6)} Faster than double exponential cooling}
In this case
\begin{equation}\label{eg3pre}
\frac{\tau(\ell)}{T_\ell} \to 1,\qquad
\frac{\sum_{k = 1}^\ell\log^4T_k }{\log^4T_\ell} \to 1,
\end{equation}
from which, by~\eqref{0vc}, it follows that
\begin{equation}\label{asymp3}
\lim_{\ell \to \infty} \lambda_{\tau,\tau(\ell)}(\ell) = 1, \qquad \lim_{\ell\to \infty}
\frac{\var (X_{\tau(\ell)})}{\sigma_0^2\sigma_V^4\log^4 T_\ell}=1,
\end{equation}
and therefore
%%
%\begin{equation}
%\label{eg3pos}
$\lambda^2_{\tau,\tau(\ell)}(\ell) \to 1$.
%\end{equation}
%%
%As before, we combine
By~\eqref{asymp3} %with 
and Corollary~\ref{cor:rRecThree}(b),% to obtain
\begin{equation}\label{inlaw3}
\frac{X_{\tau(\ell)}}{\sigma_0^2\log^2\tau(\ell)} \todd{\ell}  V.
\end{equation}

\paragraph{Subsequences}
In {\bf (Ex.5)} and {\bf (Ex.6)} we need to examine the effect of the boundary. Let $\prt{n_i}_{i \in \N}$ be a subsequence for which~\eqref{e:beg} holds. Then
\begin{equation}
\label{eq:boundcases}
 \frac{ \log \tau(\ell(n_i)-1)}{\log n_i} \to
\begin{cases}
1,    &\text{ if } b\leq 1, \\
b^{-1}, & \text{ if } b>1.
\end{cases}
\end{equation}
Decompose $X_{n_i}=X_{\tau(\ell(n_i)-1)} + \bar{Y}^{n_i}$. By conveniently rewriting the scaling factors, we obtain
\begin{equation}
\label{e:vbb}
\begin{aligned}
&\frac{X_{n_i}}{\sigma_V\sigma_0^2\log^2 n_i}    
 =\frac{ \log^2 \tau(\ell(n_i)-1)}{\log^2 n_i}\times\\
&\bigg(\frac{X_{\tau(\ell(n_i)-1)}}
{\sigma_V\sigma_0^2 \log^2 \tau(\ell(n_i)-1)}
+\frac{\log^2 \bar{T}^{n_i}}{ \log^2 \tau(\ell(n_i)-1)}
\frac{\bar{Y}^{n_i}}{\sigma_V\sigma_0^2 \log^2  \bar{T}^{n_i}} \bigg).
\end{aligned}
\end{equation}
Using~\eqref{inlaw2} and~\eqref{inlaw3} in combination with~\eqref{eq:boundcases} and~\eqref{e:vbb}, we conclude that
\begin{equation}
\label{e:convbb}
\frac{X_{n_i}}{\sigma_V\sigma_0^2\log^2 n_i}\todp
\begin{cases}
q_c^{-1}V^{\otimes \lambda_{ q_c}} + b^2 \sigma_V^{-1}V_0, &\text{if } b\leq 1,\\[0.4cm]
b^{-2}q_c^{-1}V^{\otimes \lambda_{q_c}} + \sigma_V^{-1} V_0, &\text{if } b>1.
\end{cases}
\end{equation}
\end{proof}

%%%%%%%%%%%%%%%%% SECTION 4 %%%%%%%%%%%%%%%%%%%%%%%%%%%

\section{Proofs: Gaussian fluctuations}
\label{s:Gauss}

\subsection{Convergence in the Gaussian regime}
\label{s:Gauss2}

%We prove Theorem~\ref{thm:TranGauss}.
\vspace{-2mm}
%\begin{proof}
\Proof{Proof of Theorem~\ref{thm:TranGauss}}
By Theorem~\ref{thm:RWREextra1}, 
\begin{equation}\label{svc}%s case variance convergence 
\sigma^2_{s}(n):=\text{Var}\crt{\frac{Z_n}{\sqrt{n}}} \xrightarrow[]{} \sigma^2_s.
\end{equation}
Consider a probability space $\prt{S,\mc{S},\mc{P}}$ that is rich enough to include a sequence of i.i.d.\ standard normal random variables $\prt{\Phi_k}_{k \in \bb{N}_0}$ and a collection of random variables $(\mc{R}^{(k)}_n)_{k,n\in\N_0}$ satisfying (recall (A1)--(A4) in Section~\ref{sec:pTSKlp}):
\begin{itemize}
\item[(B1)] For any $k,n\in \bb{N_0}$ and $x \in \bb{R}$,
\begin{equation}\label{seqdist}
\begin{aligned}
&P^{\mu}_0\prt{\frac{Z_n - E^{\mu}_0\crt{Z_n}}{\sigma_{s}(n)}\leq x}
= \mc{P}\left(\Phi_k + \mc{R}^{(k)}_n \leq x\right) .
\end{aligned}
\end{equation}
\item[(B2)] For all $k,n\in\N_0$, $\mc{E}[\mc{R}^{(k)}_n] = 0$.
\item[(B3)] $(\Phi_k,\mc{R}^{(k)}_n)_{n,k\in \bb N_0}$ are independent in $k$ under $\mc{P}$.
\item[(B4)] $\mc{R}^{(k)}_n$ vanishes in $L^2$, i.e.,
\begin{equation}\label{e:l2serror}
\lim_{n\to\infty} \sup_{k\in\N_0} \mc{E}\crt{\prt{\mc{R}^{(k)}_n}^2} = 0.
\end{equation}
\end{itemize}
	
\noindent
Note that, by %the definition of $\lambda_{\tau,n}$ in~
\eqref{e:lfp} %and the independence of $\prt{\Phi_k}_{k\in \N_0}$,
%%
%\begin{equation}\label{Gaumix}
$\sum_{k = 0}^{\ell(n) -1} \lambda_{\tau,n}(k)\Phi_k \overset{(d)}{=}\Phi$.
%\end{equation}
%%
and by%By~\eqref{e:c3} and
~\eqref{seqdist},
\begin{equation}\label{e:seqdist0}
\mathfrak{X}_n - \bb{E}\crt{\mathfrak{X}_n} 
\overset{(d)}{=} \sum_{k=0}^{\ell(n)-1}\lambda_{\tau,n}(k) \left(\Phi_k + \mc{R}^{(k)}_{T_k}\right) 
\overset{(d)}{=}  \Phi +  \sum_{k=0}^{\ell(n)-1}\lambda_{\tau,n}(k)  \mc{R}^{(k)}_{T_k} .
\end{equation}
i.e., $\mathfrak{X}_n - \bb{E}[\mathfrak{X}_n]$ has the same distribution as a standard normal distribution, up to an error term
that is negligible because of~\eqref{e:l2serror}. By (B2) and (B3), we have
\begin{equation}
\mc{E}\crt{\prt{
\sum_{k=0}^{\ell(n)-1}\lambda_{\tau,n}(k)\, R^{(k)}_{T_k}}^2}
= \sum_{k=0}^{\ell(n)-1} \lambda^2_{\tau,n}(k)\,
\mc{E} \big[\big(R^{(k)}_{T_k}\big)^2\big].
\end{equation}
Since $\sum_{k=0}^{\ell(n)-1}\lambda^2_{\tau,n}(k) = 1$, it follows from~\eqref{e:l2serror} that
\begin{equation}\label{e:svanish}
\lim_{J\to\infty}\limsup_{n\to\infty}
\mc{E}\crt{\prt{ \sum_{k=0}^{\ell(n)-1} \lambda_{\tau,n}(k)\,
R_k\Ind{\chv{T_k>J}}}^2} = 0.
\end{equation}
On the other hand, for any fixed $J>0$, under $\bb{P}$, $(Y_k \Ind{\chv{T_k \leq J}})_{k\in\N}$ is a collection of bounded independent random variables. Thus, by the CLT for i.i.d.\ random variables, we get
\begin{equation}\label{e:snormalmass}
\begin{aligned}
\lim_{n\to \infty}
\Bigg\vert\bb{E}\Bigg[f\Bigg(&\sum_{k = 0}^{\ell(n)-1} \lambda_{\tau,n}(k)\,
\frac{Y_k - \bb{E}\crt{Y_k}}{\sigma_{s,T_k}}
\Ind{\chv{T_k \leq J}}\Bigg) \Bigg]\\
&-\mc{E}\Bigg[f\Bigg(\Bigg(\sum_{k = 0}^{\ell(n)-1}\lambda^2_{\tau,n}(k)
\Ind{\chv{T_k \leq J}}\Bigg)^{\frac{1}{2}}\Phi\Bigg)\Bigg]\Bigg\vert = 0
\end{aligned}
\end{equation}
%\begin{equation}\label{e:snormalmass} 
%\begin{aligned}
%&\lim_{n\to \infty}  \left\vert\bb{E}\left[f\left(\sum_{k = 0}^{\ell(n)-1} \lambda_{\tau,n}(k)\,
%\frac{Y_k - \bb{E}\crt{Y_k}}{\sigma_{0,T_k}}\,\Ind{\chv{T_k \leq J}}\right)\right] \right. \\
%&\qquad\qquad\qquad\qquad\qquad \left.
%- \bb{E}\left[f\left(\left(\sum_{k = 0}^{\ell(n)-1}\lambda^2_{\tau,n}(k)\right)^{1/2}\Phi\right)\right]\right\vert = 0.
%\end{aligned}
%\end{equation}
%%
Theorem~\ref{thm:TranGauss} follows from%~\eqref{Gaumix},
~\eqref{e:svanish} and~\eqref{e:snormalmass} by applying the same arguments that led to~\eqref{f0inlaw}, 
%Indeed, as in~\eqref{not_ease}, we consider the truncated random variables %define 
%%
%\arraycolsep=1.2pt\def\arraystretch{2} 
%\begin{array}{rlrl}
%:= \sum_{k=0}^{\ell(n)-1} \overline{\lambda_{n}}^{\,0,J} \Phi_k$, 
%:= \sum_{k=0}^{\ell(n)-1} \overline{\lambda_{n}}^{\,J,\infty} \Phi_k,
%:= \sum_{k=0}^{\ell(n)-1}\overline{\lambda_{n}}^{\,J,\infty}(k)\,\mc{R}^{(k)}_{T_k}\,
%:= \overline{\mc{R}_n}^{\,0,J}-\overline{\mc{R}_n}^{\,0,J}$. 
%\end{array} 
%%\begin{equation}\label{snotation}
%$\overline{\Phi}_n^{\,0,J}$,%\qquad
%$\overline{\Phi}_n^{\,J,\infty}$,%\qquad 
%$\overline{\mc{R}_n}^{\,0,J}$, %\qquad
%%\Ind{\chv{T_k>J}}, \quad
%$\overline{\mc{R}_n}^{\,0,J}$.
%%\end{equation}
%%
%With this notation, for any bounded function $f\colon\,\bb{R}\to \bb{R}$ with bounded first derivative, via arguments analogous to those employed in~\eqref{f0inlaw},
 it follows that
\begin{equation}\label{fsinlawshort}
\mathfrak{X}_n - \bb{E}\crt{\mathfrak{X}_n} \tod \Phi.
\end{equation}
To prove the convergence in $L^2$ of  $(\tilde{\mathfrak{X}_n^2})_{n\in \N}$, as in~\eqref{not_ease} we consider the truncated random variables  %define 
%%
%\arraycolsep=1.2pt\def\arraystretch{2} 
%\begin{array}{rlrl}
%:= \sum_{k=0}^{\ell(n)-1} \overline{\lambda_{n}}^{\,0,J} \Phi_k$, 
%:= \sum_{k=0}^{\ell(n)-1} \overline{\lambda_{n}}^{\,J,\infty} \Phi_k,
%:= \sum_{k=0}^{\ell(n)-1}\overline{\lambda_{n}}^{\,J,\infty}(k)\,\mc{R}^{(k)}_{T_k}\,
%:= \overline{\mc{R}_n}^{\,0,J}-\overline{\mc{R}_n}^{\,0,J}$. 
%\end{array} 
%\begin{equation}\label{snotation}
$\overline{\Phi}_n^{\,0,J}$,%\qquad
$\overline{\Phi}_n^{\,J,\infty}$,%\qquad 
$\overline{\mc{R}_n}^{\,0,J}$, %\qquad
%\Ind{\chv{T_k>J}}, \quad
$\overline{\mc{R}_n}^{\,0,J}$. Now,% note that
%\end{equation}
%%
%%
\begin{equation} 
\label{sunifint}  
\begin{aligned} 
&\lim_{M\to \infty} \sup_{n\in\bb{N}_0}\bb{E}\crt{\tilde{\mathfrak{X}}_n^{2}\Ind{\chv{\tilde{\mathfrak{X}}_n^2>M}}}\\
&\quad = \lim_{M\to \infty}\sup_{n\in\bb{N}_0}\mc{E}\crt{\prt{\prt{\overline{\Phi}_n^{\,0,\infty} 
+ \overline{\mc{R}_n}^{\,0,\infty}}^2\Ind{\chv{\tilde{\mathfrak{X}}_n^2>M}}}}\\
&\quad \leq \inf_J\sup_{n\in \N_0}\mc{E}\crt{\prt{ \overline{\mc{R}_n}^{\,J,\infty}}^2} = 0,
\end{aligned} 
\end{equation}
where we used the uniform integrability in $L^2$ of $(\overline{\Phi}_n^{\,0,J})_{n\in\N_0}$ and $(\overline{\mc{R}}^{\,0,J}_n)_{n\in\N_0}$ to obtain the the inequality.
%\end{proof}
\QED

%%%%%%%%
\subsection{Limit points and stability of the variance}

%We prove Corollary~\ref{Gaussregular}.

%\begin{proof}
\Proof{Proof of Corollary~\ref{Gaussregular}}
To  prove~\eqref{gaussgeneral}, %simply 
note that if $\sigma_{s,\tau}(n_i) \to \sigma$, then
\begin{equation}
\frac{X_n - \bb{E}[X_{n_i}]}{\sigma \sqrt{n_i}}  = \frac{\sigma_{s,\tau}(n_i)}{\sigma} \prt{\mathfrak{X}_{n_i} - \bb{E}\crt{\mathfrak{X}_{n_i} }} \tod \Phi.
\end{equation}
To prove~\eqref{gausscooling}, use~\eqref{svc}. Indeed, if $T_k \to \infty$, then 
%\[
$\sigma^2_{s}(T_k) \to \sigma^2_s$
%\] 
and
\begin{equation}\label{sncool}
\begin{aligned}
\lim_{n\to\infty} \frac{\sigma^2_{s,\tau}(n)}{\sigma_s^2}
= \lim_{n\to\infty} \frac{\var\prt{X_n}}{n\sigma^2_s }
= \lim_{n\to\infty} \sum_{k = 0}^{\ell(n)} \frac{T_{k}}{n}
\frac{\sigma^2_{s}(T_k)}{\sigma_s^2} = 1,
\end{aligned}
\end{equation}
where the last equality follows from the Toeplitz Lemma~\cite[Thm.1.2.3]{R68}. 
%\end{proof}
\QED

%%%%%%%%%%%%%%%%%%%%%%%%%%%%%%%%%%%%%%%%%%%%%%%%%%%%%%%%%%%%
\subsection{Stable centering and counterexample}
\label{scentering}

%We next examine {\bf (Ex.7)} and {\bf (Ex.8)}.

\begin{proof}
We first turn to {\bf (Ex.7)}. To prove~\eqref{sclt}, note that  the $L^2$-convergence in~\eqref{Lpsnormal} implies that 
\begin{equation}\label{staticcentering}
\bb{E}\crt{\frac{Z_n - nv_\mu}{\sqrt{n}}} \to 0.
\end{equation}
Let $C_n := \sum_{k = 0}^{\ell(n)-1} \lambda_{\tau,n}(k)$ and $C := \sup_{n\in\N} C_n$. Condition~\eqref{scltgrowth} corresponds to $C<\infty$. In this case, by Markov's inequality,~Theorem~\ref{thm:RWREextra2}(II) and the Toeplitz lemma~\cite[Thm.1.2.3]{R68}, we have
\begin{equation}\label{scenterfastcool}
\begin{aligned}
\abs{  \bb{E} \crt{ \frac{X_{n} - nv_\mu}{\sigma_s\sqrt{n}}  }} &\leq 
\sum_{k=0}^{\ell(n)-1} \frac{\lambda_{\tau,n}(k)}{C}\abs{ \bb{E}\crt{\frac{Y_k- T_kv_\mu }{\sigma_s\sqrt{T_k}}} }  
\xrightarrow[n\to\infty]{} 0.
\end{aligned}
\end{equation}
From~\eqref{sncool} and~\eqref{scenterfastcool} it follows that  
\begin{equation}
\frac{X_n - n v_\mu}{\sigma_s \sqrt{n}} = \alpha_n\prt{\mathfrak{X}_n
- \bb{E}\crt{\mathfrak{X}_n}} + \beta_n
\end{equation}
with $\alpha_n \to 1$, $\beta_n \to 0$,
and~\eqref{sclt} follows from~\eqref{e:GFcis}.
	
We next turn to {\bf (Ex.8)}. To show~\eqref{infinity}, consider the sets
\begin{equation}
\begin{aligned}
\mathcal{N}_+ := \{n \in \N \colon\, E^\mu_0[Z_n] > 0\}, \quad
\mathcal{N}_- := \{n \in \N \colon E^\mu_0[Z_n] < 0\}.
\end{aligned}
\end{equation}
By Theorem~\eqref{thm:RWREextra2}(2), there exists an $s$-transient $\alpha$ for which at least one of these sets is infinite. Assume without loss that $\mathcal{N}_+ =\chv{n_1<n_2<\ldots}$ is infinite. 
%Let $n_1<n_2<\ldots$ denote the elements of $\mathcal{N}_+$. 
Define the cooling map by successively picking $N_\ell$ consecutive increments of size $n_\ell$ for every $\ell \in \N$, where the values of $\prt{N_\ell}_{\ell\in \N}$ are chosen such that
\begin{equation}\label{Nellcenter}
\frac{N_\ell}{\sqrt{\sum_{m=1}^\ell N_m n_m}} \prt{E^\mu_0\crt{Z_{n_\ell}}- v_{\mu}} > \ell.
\end{equation}
Let
%%
%\begin{equation}
$s(0) := 0$, and for  $\ell\in\N$, define $s(\ell) := s(\ell-1) + N_\ell n_\ell$.
%\end{equation}
Therefore, %by noting that
\begin{equation}
\bb{E}\crt{X_{n(k)}} - \sqrt{s(k)}\,v_\mu 
=  \sum_{\ell=1}^k \frac{N_\ell}{\sqrt{\sum_{m=1}^\ell N_m n_m}} \prt{E^{\mu}_0\crt{Z_{n_\ell}}- v_{\mu}}>k,
\end{equation}
which proves~\eqref{infinity}.
\end{proof}

%%%%%%% APPENDICES %%%%%%%%%%%%%%

\appendix

%%%%%%% APPENDIX A %%%%%%%%%%%%%%%%%%%%%%%%%%%%

\section{$L^p$-convergence in the Gaussian regime} \label{appA} 

We prove Theorem~\ref{thm:RWREextra1}.

\paragraph{Preparation}
Recall that $\rho_j = \frac{1-\omega(j)}{\omega(j)}$. Following Zeitouni~\cite[Section 2.2]{ZZ04}, we have
\begin{equation}\label{e:Delta}
\Delta(j,\omega) := -1 + v_\mu \,\Sigma(\theta^j \omega), \quad j \in \Z,
\end{equation}
where
\begin{equation}\label{e:S}
\begin{aligned}
&\Sigma(\omega) := \sum_{i = -\infty}^0\ \frac{1}{\omega_i} \prod_{j = i+1}^0 \rho_j
\end{aligned}
\end{equation}
and $\theta$ denotes the spatial shift operator acting on $(0,1)^\Z$ (i.e., $(\theta \omega)(j) = \omega(j + 1)$, $j \in \Z$). Now  define, for $n \in \N$,
\begin{equation}\label{e:CLTd} % CLT decomposition
M_n := Z_n - v_\mu n +S_n + R_n, 
\end{equation}
where ($S_0=0$)
\begin{equation}\label{e:SRn}
\begin{aligned}
S_n  &:= \sum_{j = 0}^{nv_\mu} \Delta(j,\omega),%\\[0.3cm]
 \qquad R_n :=
\begin{cases}
\sum_{j = Z_n}^{nv_\mu} \Delta(j,\omega), \quad &\text{ if } Z_n < nv_\mu,\\[0.2cm]
0, \quad &\text{ if } Z_n = nv_\mu,\\[0.2cm]
\sum_{j = nv_\mu + 1}^{Z_n -1} \Delta(j,\omega), \quad &\text{ if } Z_n > nv_\mu.
\end{cases}
\end{aligned}
\end{equation}
Note that, in this decomposition, $S_n$ depends only on $\omega$. Therefore we will distinguish between the different measures and write $E_\mu$ for expectation with respect to $\mu$. Next, by~\cite[Theorem 1.16 (i)]{Sol75}, $v^{-1}_\mu = E_\mu\crt{\Sigma(\omega)}$, and consequently 
\begin{equation}\label{Deltamean0}
E_\mu\crt{\Delta(x,\omega)}=0.
\end{equation}
Therefore, for $s\in(2,\infty)$,~\eqref{e:CLTd} is a decomposition of $\prt{Z_n- v_\mu n}_{n\in\N_0}$ into a martingale $\prt{M_n}_{n\in \N_0}$ with respect to the natural filtration of the random walk $\mc{F}_n =\sigma(Z_i\colon 0 \leq i \leq n) $ and the probability measure $P^\omega_0$ for any $\omega\in (0,1)^\Z$; a mean-zero stationary sequence $\prt{S_n}_{n\in\N_0}$ with respect to the shift operator $\theta$ and the measure $\mu$; and a remainder term $\prt{R_n}_{n\in\N_0}$. Furthermore, the assumptions of~\cite[Theorem 2.2.1]{ZZ04} are satisfied and, under the annealed measure $P^{\mu}_0$,
\begin{equation}\label{e:conv}
\begin{aligned}
&  n^{-\frac{1}{2}} R_n \tod 0, \quad n^{-\frac{1}{2}} M_n \tod \sigma_{1,\mu}\Phi_1,\quad
&n^{-\frac{1}{2}} S_n \tod \sigma_{2,\mu} \Phi_2.
\end{aligned}
\end{equation}
where $\sigma_{1,\mu}$, $\sigma_{2,\mu}$ will be introduced below and $\Phi_1$, $\Phi_2$ are standard normal random variables. To prove $L^p$-convergence, it suffices to show that, for any $p\in(2,s)$
\begin{equation}\label{e:R20}
\sup_{n\in\N} E^{\mu}_0\crt{\abs{n^{-\frac{1}{2}}R_n}^p} <\infty,\\
\end{equation}
\begin{equation} \label{martingalep}
\sup_{n\in\N}E^{\mu}_0\crt{\abs{ n^{-\frac{1}{2}}M_n}^p} <\infty,\\
\end{equation}
%%
%and
%%
\begin{equation} \label{stationaryp}
\sup_{n\in\N}E_{\mu}\crt{\abs{ n^{-\frac{1}{2}}S_n}^p} <\infty.
\end{equation}
These conditions ensure uniform integrability in $L^p$ for $p<s$ and, combined with~\eqref{e:conv}, yield the desired result.
The proof of \eqref{e:R20} is given in Section~\ref{apprem}, and the proofs of~\eqref{martingalep},~\eqref{stationaryp} are given in Section~\ref{ss:L2}.  

%%%%%%%%%%%%%%%%%%%%%%%%%%%%%%%%%%%%%%%%%%%%%%%%%%%%%%%%%%%%
\subsection{Remainder term}
\label{apprem}

For $p\in (2,s)$, note that
\begin{equation} \label{e:R2}
\sup_{n\in\N}E^{\mu}_0\crt{\abs{n^{-\frac{1}{2}} R_n}^p}
= \sup_{n\in\N}\int_0^\infty p\delta^{p-1}\,P^{\mu}_0 \prt{n^{-\frac{1}{2}}\abs{R_n}>\delta}\,d\delta.
\end{equation}
As $P^{\mu}_0 \prt{\abs{Z_n - n v_\mu}>2n} = 0$, by~\eqref{e:SRn}, we have
\begin{equation} \label{e:PRd}
\begin{aligned}
&P^{\mu}_0 \prt{n^{-\frac{1}{2}}\abs{R_n}>\delta}\\[0.2cm]
%&\leq  P^{\mu}_0 \prt{\abs{Z_n - n v_\mu}>2n}\\
&\leq  \mu\prt{\max_{j-, j+ \in (v_\mu n -2n ,v_\mu n +  2n)} 
\abs{\sum_{i=j-}^{j+} \frac{\Delta(i,\omega)}{\sqrt{n}}} \geq \delta}\\[0.2cm]
&=  \mu\prt{\max_{j-, j+ \in (-2n ,2n)} \abs{\sum_{i=j-}^{j+} \frac{\Delta(i,\omega)}{\sqrt{n}}} \geq \delta}\\
&\leq 2\mu \prt{\max_{ j \in (0, 2n)} \abs{\sum_{i=j}^{0} \frac{\Delta(i,\omega)}{\sqrt{n}}} \geq \frac{\delta}{2}},
\end{aligned}
\end{equation}
%%
%where in the first inequality we dominated the value of $R_n$ by the maximum of the sum in the interval $(v_\mu n -2n ,v_\mu n +  2n)$, and the second equality follows from $P^{\mu}_0 \prt{\abs{Z_n - n v_\mu}>2n} = 0$ (because $\abs{Z_n}\leq n$ and $\abs{v_\mu}\leq 1$). 
where in the first inequality, since the random variable does not depend on the random walk and is a function of the environment $\omega$ only, we replace $P^\mu_0$ by $\mu$; the equality   follows from the stationarity of $\Delta(i,\omega)$ and to obtain the last inequality we estimate the invcrement from $j-$ to $j+$ in terms of the distance to the origin and use symmetry.
%%
%\begin{equation} \label{e:osmi} %one side max inequality
%\begin{aligned}
%&\mu \prt{\omega \,\colon\,\max_{j-, j+ \in (-2n, 2n)} \abs{\sum_{i=j-}^{j+} \frac{\Delta(i,\omega)}{\sqrt{n}}} \geq \delta}\\
%%&\leq 2\mu \prt{\omega \,\colon\,\max_{ j+ \in [0, 2n)} \abs{\sum_{i=0}^{j+} \frac{\Delta(i,\omega)}{\sqrt{n}}} \geq \delta/2}
%\\
%%&\qquad + 2\mu \prt{\omega \,\colon\,\max_{ j- \in (-2n,0]} \abs{\sum_{i=j-}^{0} \frac{\Delta(i,\omega)}{\sqrt{n}}} \geq \delta/2}\\
%& \leq 4\mu \prt{\omega \,\colon\,\max_{ j- \in (0, -2n)} \abs{\sum_{i=j-}^{0} \frac{\Delta(i,\omega)}{\sqrt{n}}} \geq \delta/2}.
%\end{aligned}
%\end{equation}
%%
By Markov's inequality,
\begin{equation} \label{e:MRd}
\begin{aligned}
&\mu \prt{\max_{ j \in (0, 2n)} \abs{\sum_{i=0}^{j} \frac{\Delta(i,\omega)}{\sqrt{n}}} \geq \delta} \leq \frac{1}{\delta^p} E_\mu \crt{\max_{ j \in (0, 2n)} \abs{\sum_{i=0}^{j} \frac{\Delta(i,\omega)}{\sqrt{n}}}^p}.
\end{aligned}
\end{equation}
We estimate this expectation with the help of~\cite[Proposition 7]{MerFloPel06},
\begin{equation}\label{e:Lpmax}
\begin{aligned}
E_{\mu} \prt{\max_{ j+ \in (0, n)} \abs{\sum_{i=0}^{j+} \frac{\Delta(i,\omega)}{\sqrt{n}}}^p}
\leq C_p \prt{\frac{\sum_{i = 1}^{n} b_{i,n,p}}{n}}^{\frac{p}{2}},
\end{aligned}
\end{equation}
where
\begin{equation}\label{e:bin}
\begin{aligned}
b_{i,n,p} := \max_{i\leq \ell \leq n} \norm{\Delta(i, \omega)
\sum_{k = i}^\ell \mu\crt{\Delta(k, \omega) \vert \mc{G}_i}}_{\frac{p}{2}},
\end{aligned}
\end{equation}
$\mc{G}_i:=\sigma(\omega(j)\colon j \leq i)$, and 
$\norm{f}_p= \int_0^1 \abs{f(\omega)}^p \, d\mu(\omega)$.
 Below we show that 
\begin{equation}\label{stationaryub}
\sup_{i,n}   b_{i,n,p} =: K < \infty.
\end{equation}

\medskip\noindent
To conclude the proof of  \eqref{e:R20} with the help of \eqref{stationaryub}, note that \eqref{e:Lpmax} is uniformly bounded in $n \in \N$ and therefore by combining it with~\eqref{e:PRd}--\eqref{e:MRd} we can bound the right-hand side of~\eqref{e:R2} by
\begin{equation}\label{e:R2dec}
\int_0^1 p\delta^{p-1}\, P^{\mu}_0 \prt{n^{-\frac{1}{2}}\abs{R_n}>\delta} \,d\delta
+ C\int_1^\infty\delta^{p-1} \frac{1}{\delta^{p'}}\, d\delta
\end{equation}
for some $C>0$ and $p'\in (p,s)$. Since for $p'>p$ the second integral above is finite,
%%
%\begin{equation}\label{e:Ai}
%\int_1^\infty \delta^{p-1} \frac{1}{\delta^{p'}}\,d\delta <\infty,
%\end{equation}
%%
this conclude the proof of~\eqref{e:R20}. It remains to verify~\eqref{stationaryub}.

%%%%%%%%%%%%%%%%%%%%%%%%%%%%%%%%%%%%%%%%%%%%%%%%%%

\bigskip\noindent
{\bf Bound on $b_{i,n,p}$.}
To prove~\eqref{stationaryub}, %note that, as $v^{-1}_\mu = E_\mu\crt{\Sigma(\omega)}$, 
the expression~\eqref{e:S} allows us to bound the conditional expectation in~\eqref{e:bin} by
\begin{equation} \label{e:ubc}
\begin{aligned}
&E_{\mu}\crt{\Delta(i + k, \omega) \vert \mc{G}_i} = -1 + \frac{1}{E_{\mu}\crt{\Sigma(\omega)}}\\
&\qquad\times\bigg(\langle\omega(0)^{-1}\rangle\prt{1 + \langle\rho\rangle + \cdots + \langle\rho\rangle^{k-1}}
+ \langle\rho\rangle^k \Sigma(\theta^i \omega)\bigg)\\
&\leq \frac{1}{E_{\mu}\crt{\Sigma(\omega)}}\langle\rho\rangle^k \prt{\frac{-\langle\omega(0)^{-1}\rangle}
{1 - \langle\rho\rangle}  + \Sigma(\theta^i \omega)},
\end{aligned}
\end{equation}
where the inequality follows from observing that
\begin{equation}
E_{\mu}\crt{\Sigma(\omega)} = \langle\omega^{-1}\rangle
\prt{1 + \langle\rho\rangle + \langle\rho\rangle^2 + \cdots}.
\end{equation}
The right-hand side of~\eqref{e:bin} is bounded by
\begin{equation} \label{e:sbc}
\begin{aligned}
&\sum_{k\in\N_0} \norm{\Delta(i,\omega) E_\mu\crt{\Delta(i + k,\omega) \vert \mc{G}_i}}_{\tfrac p 2}\\
&\quad \leq \sum_{k\in\N_0} v_\mu\langle\rho\rangle^k 
\norm{\prt{-1 + v_\mu \Sigma(\theta^i \omega)} \prt{\frac{-\langle\omega^{-1}\rangle}
{1 - \langle\rho\rangle}  + \Sigma(\theta^i \omega)}}_{\tfrac p 2}\\
&\quad \leq \sum_{k\in\N_0}^\infty v_\mu\langle\rho\rangle^k 
C\prt{1 + \mu\crt{\prt{\Sigma(\theta^i \omega)}^{p}}}^{\frac{2}{p}}\\
&\quad =  \frac{1}{1 - \langle\rho\rangle}v_\mu C\prt{1 + \mu\crt{\prt{\Sigma(\omega)}^{p}}}^{\frac{2}{p}} <\infty,
\end{aligned}
\end{equation}
where in the second inequality we used that $ab \leq a^2 + b^2$ to separate the constants from the random variable $\Sigma(\theta^i \omega)$, and $C>0$ is a constant that does not depend on $k$ or $i$. The last equality follows from $\Sigma(\theta^i \omega)\overset{(d)}{=}\Sigma( \omega)$, and the final bound follows from $E_\mu\crt{\prt{\Sigma(\omega)}^{p}}<\infty$ for $p<s$, as can be seen by applying Minkowsky's inequality on the $L^p$ norm of~\eqref{e:S}.

%%%%%%%%%%%%%%%%%%%%%%%%%%%%%%%%%%%%%%%%%%%%%%%%%%
\subsection{$L^p$-convergence to the normal} \label{ss:L2}

To prove~\eqref{martingalep} and~\eqref{stationaryp} we will bound the expectations  with a bound on the difference of the distribution functions of the each random variables and the standard  normal distribution.

Let $\prt{W_n}_{n\in\N}$ be a sequence of random variables,  $P$ the underlying probability measure and $E $ its corresponding expectation. Assume that this sequence converges in distribution to the standard normal. To prove that this convergence is also in $L^p$- we need to show that, for $p <s$,
\begin{equation} \label{lp}
\sup_{n\in\N} E \left[ \abs{W_n}^p \right] < \infty.
\end{equation}
We have
\begin{equation} \label{p-moment}
\begin{aligned}
E \left[ \abs{W_n}^p \right]
&= \int_0^\infty dx\,p x^{p-1} P \left( \abs{W_n} > x \right)\\
&= \int_0^\infty dx\,p x^{p-1} \big[ P \left( \abs{W_n} > x \right)
- P \left( \abs{\Phi} > x \right) \big]
\\
&\qquad\qquad\qquad\qquad \qquad+ \int_0^\infty dx\, p x^{p-1} P \left( \abs{\Phi} > x \right).
\end{aligned}
\end{equation}
Since
%%
%\begin{equation}
$\int_0^\infty dx\,p x^{p-1} P \left( \abs{\Phi} > x \right) < \infty$,
%\end{equation}
%%
if
\begin{equation}
\abs{P(\abs{W_n} \geq x) - P \left( \abs{\Phi} > x \right)} \leq C \, a_n \, f(x),
\end{equation}
where $a_n$ and $f(x)$ satisfy
\begin{equation}\label{e:ubpd}%uniform bound probability difference
\sup_{n\in\N} a_n < \infty, \qquad \int_0^\infty dx\,x^{p-1} \, f(x) < \infty,
\end{equation}
then~\eqref{lp} follows.

%%%%%%%%%%%%%%%%%%%%%%%%%%%%%%
\subsubsection{Martingale part}

We will use a result in~\cite{HaeJoo88} to prove \eqref{martingalep}. Define $M_0=0$, and the square-integrable martingale difference sequence $(D_n)_{n\in\N}$ by
%%
%\begin{equation}
$D_k := M_k - M_{k-1}$.
%\end{equation}
%%
As shown in~\cite[p.211]{ZZ04}, the quadratic variation of $(M_n)_{n\in\N}$ under $P^\omega_0$ is given by $A^\omega_n := \sum_{k = 1}^n E^\omega_0[D_k^2\vert \mc{F}_{k-1}]$, where
\begin{equation}\label{qvmartingale}
\begin{aligned}
&E^\omega_0[D_k^2\vert \mc{F}_{k-1}]\\
&\quad = v^2_\mu\bigg[\bar{\omega}_0(k) \big(\Sigma(\bar{\omega}(k))-1\big)^2 + (1 -\bar{\omega}_0(k)) (\Sigma\big(\theta^{-1}\bar{\omega}(k))+1\big)^2 \bigg],
\end{aligned}
\end{equation}
As shown in~\cite[Corollary 2.1.25]{ZZ04}, the sequence $(\bar{\omega}(k): = \theta^{Z_k} \omega)_{k\in \N}$ is stationary and ergodic under $Q\otimes P^\omega_0$,  where 
%\begin{equation}\label{ergm}
$Q(d\omega) := \Lambda(\omega) P(d\omega)$,
%\end{equation}
and 
%\begin{equation}\label{radonderivative}
$\Lambda(\omega) := \frac{1}{\omega_0} + \frac{1}{\omega_0}\rho_1 +  \frac{1}{\omega_0}\rho_1\rho_2 + \cdots 
=\frac{1}{\omega_0} \prt{ \sum_{i = 0}^\infty\prod_{j = 0}^{i-1}\rho_j}$.
%\end{equation}
 Therefore, letting $E^Q$ denote the expectation with respect to $Q$, we see that the following limit exists $Q$-almost surely:
\begin{equation}\label{Qexp}
\begin{aligned}
\sigma^2_{\mu,1} &:= \lim_{n\to\infty} \frac{1}{n} \sum_{k=1}^n E_0^\omega\crt{D_k^2\vert \mc{F}_{k-1}}\\
&= E^Q\bigg[v^2_\mu\bigg[\omega_0 \big(\Sigma(\omega)-1\big)^2 %\\\
%&\qquad+ (1 -\omega_0) 
+(\Sigma\big(\theta^{-1}\omega)+1\big)^2 \bigg]\bigg].
 \end{aligned}
\end{equation}
%%
%for all $ k \in \N_0$.
Fix $\delta > 0$ such that $2 + 2 \delta < s$, let $D_{k,n} := (\sigma_{\mu,1} \sqrt{n}) ^{-1}D_k$, and consider the following two quantities:
\begin{align}
&A_{n,\delta} := \sum_{k=1}^n E^{\mu}_0 \left[ \left\vert D_{k,n} \right\vert^{2 + 2\delta} \right],\\[0.2cm]
&B_{n,\delta} := E^{\mu}_0 \left[ \abs{1-\sum_{k=1}^n E^{\mu}_0 \left[ D_{k,n}^2 \vert \mc{F}_{k-1} \right] }^{1+\delta} \right].
\end{align}
Since
\begin{equation}
E^{\mu}_0 \crt{\abs{D_{k,n}}^{2 + 2 \delta}} = \frac{1}{\sigma_{\mu,1}^{2 + 2 \delta} n^{1 + \delta}} \,
E^{\mu}_0 \left[ \abs{M_k - M_{k-1}}^{2 + 2 \delta} \right],
\end{equation}
we can bound $A_{n,\delta}$ by:
\begin{equation}
\begin{aligned}
%A_{n,\delta} &= \frac{1}{\sigma_{\mu,1}^{2 + 2 \delta} n^{1+\delta}} \sum_{k=1}^n E^{\mu}_0 
%\left[ \abs{M_k - M_{k-1}}^{2 + 2 \delta} \right] \\
%&\leq 
 \frac{\sup_{k\in\N} E^{\mu}_0 
\left[ \abs{M_k - M_{k-1}}^{2 + 2 \delta} \right] \,}{\sigma_{\mu,1}^{2 + 2 \delta}n^{\delta}}.
\end{aligned}
\end{equation}
Since $2 + 2 \delta < s$, we have $\sup_{k\in\N} E^{\mu}_0 [ \abs{M_k - M_{k-1}}^{2 + 2 \delta} ] < \infty$, and therefore $A_{n,\delta} \to 0$.

To estimate $B_{n,\delta}$, we first note that 
\begin{equation}\label{condeq}
\begin{aligned}
E_0^\mu\crt{D_k^2\vert \mc{F}_{k-1}} =  \int_0^1 E_0^\omega\crt{D_k^2\vert \mc{F}_{k-1}} \, d\mu(\omega).
\end{aligned}
\end{equation}
Now note that since $\Lambda(\omega) \geq 1$, for all positive $f$,
%\begin{equation}\label{monot}
$E^Q\crt{f} \geq   E^\mu_0\crt{f}$.
%\end{equation}
Next, we apply the von Neumann $L^p$-ergodic theorem in~\cite[Corollary 1.14.1]{Wal00} to the ergodic sequence $(E^\omega_0[D_k^2\vert \mc{F}_{k-1}])_{k\in\N}$ in $ L^{1 + \delta}(Q\otimes P^\omega_0)$, to conclude that
\begin{equation}
\lim_{n\to \infty}E^Q\crt{\abs{1-\sum_{k=0}^n E^\omega_0 \left[ D_{k,n}^2 \vert \mc{F}_{k-1} \right]}^{1 +\delta}} = 0
\end{equation}
and that 
%\begin{equation}
$\lim_{n \to \infty} B_{n,\delta} = \lim_{n\to \infty}E^\mu_0\crt{\abs{1-\sum_{k=0}^n 
E^\omega_0 \left[ D_{k,n}^2 \vert \mc{F}_{k-1} \right]}^{1 +\delta}} = 0$.
%\end{equation}
By~\cite[Theorem 1]{HaeJoo88}, whenever $a_{n,\delta}:= A_{n,\delta} + B_{n,\delta}<1$, then for any $\delta > 0$ there exists a finite constant $C_\delta$ such that
\begin{equation} \label{estimate_martingaleCLT}
\abs{P^{\mu}_0\left( \sum_{k=1}^n D_k \leq x \right) - P \left( \abs{\Phi} > x \right)}
\leq C_\delta \, a_{n,\delta}^{\tfrac{1}{3+2\delta}}\,\left(1 + |x|^{2+2\delta}\right)^{-1}
\end{equation}
for all $x \in \R$. Since $a_{n,\delta} \to 0$, the terms in~\eqref{estimate_martingaleCLT} satisfy~\eqref{e:ubpd}. If we replace $W_n$ in~\eqref{lp} by $\tfrac{M_n}{\sigma_{\mu,1} \sqrt{n}}$, then we obtain~\eqref{martingalep}.

%%%%%%%%%%%%%%%%%%%%%%%%%%%%%%
\subsubsection{Stationary part}

We show~\eqref{stationaryp} with the help of~\cite[Theorem 2.4]{Jir16}. Indeed, if $\{\Delta(j,\omega)\}_{j \in \N}$ satisfies~\cite[Assumption 2.1]{Jir16}, then for some constants $C_p > 0$ and $b_{n,p} > 0$
\begin{equation} \label{e:es} %estimate_stationary
\abs{ E^\mu_0\left( \sum_{k=1}^{n v_\mu} \Delta(j,\omega)
\leq \sigma_{\mu,2} \sqrt{n} \, x \right) - \phi(x)}
\leq C_p \, b_{n,p} \, (1 + |x|^p)^{-1},
\end{equation}
for any $x \in \R$, where
\begin{equation} \label{varstationary}
\begin{aligned}
\sigma_{\mu,2}^2 &:= \lim_{n\to\infty} \frac{1}{n} E^\mu_0\crt{\prt{\sum_{k=0}^n \Delta(k,\omega)}^2}.
\end{aligned}
\end{equation}
To verify the conditions in~\cite{Jir16}, we need to introduce some notation.

Let $\omega'(0)\in(0,1)$ be an independent random variable selected according to $\alpha$, and define
\begin{equation}
\omega'(k) :=
\begin{cases}
\omega(k), & \text{ if } k \neq 0,\\
\omega'(0), & \text{ if } k = 0.
\end{cases}
\end{equation}
Recall~\eqref{e:Delta} and~\eqref{e:SRn}. Since the sequence $(\omega_{x})_{x\in \Z}$ is stationary with respect to $\theta$ under $\mu$, we have
\begin{equation} \label{B-0}
\text{$\Delta(j,\omega)$ is stationary with respect to $\theta$ under $\mu$}.
\end{equation}
In what follows, we verify the remaining conditions~\cite[Assumption 2.1]{Jir16} and fix $p \in (2,s)$. First note that
\begin{equation} \label{B-1}
\norm{\Delta(k,\omega)}_p \leq 1 + v_\mu \norm{\Sigma(\omega)}_p < \infty, 
\qquad E^\mu_0 [\Delta(j,\omega)] = 0.
\end{equation}
Next note that, since $\norm{\Delta(k,\omega) - \Delta(k,\omega')}_p \leq C_p \norm{\rho}^k_p$ with $\norm{\rho}_p< 1$ (because $p<s$), we obtain that
\begin{equation} \label{B-2}
\sum_{k=1}^\infty k ^2 \, \norm{\Delta(k,\omega) - \Delta(k,\omega')}_p < \infty.
\end{equation}
To verify the last condition note that, since $\Delta(j,\theta^{-k}\omega) =\Delta(j-k,\omega)$, by expanding~\eqref{varstationary} and using the stationarity of $\Delta(k,\omega)$, we get
\begin{equation} \label{B-5}
\begin{aligned}
\sigma_{\mu,2}^2 =   E^\mu_0 \left[ \Delta(0,\omega)^2 \right] 
+ 2 \sum_{k\in\N} E^\mu_0 \left[ \Delta(0,\omega) \, \Delta(k,\omega) \right].
\end{aligned}
\end{equation}
Since 
\begin{equation}
\Sigma(\theta^k \omega) = \frac{1}{\omega_k} + \frac{1}{\omega_{k-1}} \rho_k
+ \cdots + \frac{1}{\omega_1} \rho_k \times \cdots \times \rho_2 
+ \rho_k \times \ldots \times \rho_1 \Sigma(\omega),
\end{equation}
by~\eqref{Deltamean0} it follows that
\begin{equation}
E^\mu_0[\Sigma(\omega) \Sigma(\theta^k \omega)] = v_\mu^{-2} (1 - \langle\rho\rangle^k) + \langle\rho\rangle^k  E^\mu_0 \left[ \Sigma(\omega)^2 \right].
\end{equation}
Since, for non-degenerate $\alpha$, $ E^\mu_0[\Sigma(\omega)^2] > E^\mu_0[\Sigma(\omega)]^2 = v_\mu^{-2}$, we obtain that
\begin{equation}\label{B-3}
\sigma^2_{\mu,2}\sum_{k\in\N}  E^\mu_0\left[ \Delta(0,\omega) \, \Delta(k,\omega) \right] > 0,
\end{equation}
which implies that $\sigma_{\mu,2}^2 > 0$. Conditions~\eqref{B-0},~\eqref{B-1},~\eqref{B-2} and~\eqref{B-3} allow us to apply the result in~\cite{Jir16} and obtain~\eqref{e:es}. By substituting~\eqref{e:es} into the right-hand side of~\eqref{p-moment} we obtain~\eqref{stationaryp} and thereby conclude the proof of Theorem~\ref{thm:RWREextra1}.

%%%%%%%%%%%%%%%%%%%%%%%%%%%%%%%%%%%%%%
%%%%%        Appendix B     %%%%%%%%%%
%%%%%%%%%%%%%%%%%%%%%%%%%%%%%%%%%%%%%%
\section{Oscillations of mean displacement} \label{appB}

\subsection{Asymmetry in the Sinai regime}
\label{apprec}
We prove Theorem~\ref{thm:RWREextra2}(I).

\begin{proof} 
To show that 
\begin{equation}\label{goalrec} 
\Big\{ \alpha \colon \langle\log\rho\rangle = 0,\, E^\mu_0[Z_n] \neq 0 \,\, \text{i.o.} \Big\} \neq \emptyset,
\end{equation}
define, for $x \in (0,1)$,
%%
%\begin{equation}\label{alphax}
$\alpha_x := x \delta_x + (1-x) \delta_{\eta(x)}$,
%\end{equation}
%%
where $\eta(x)\in(0,1)$ is defined by the relation $\langle\log \rho\rangle=0$, %i.e.,
%%
%\begin{equation}\label{etax}
%x\log\left(\frac{1-x}{x}\right) + (1-x) \log\left(\frac{1-\eta(x)}{\eta(x)}\right) = 0,
%\end{equation}
%%
which makes $\alpha_x$ recurrent. Let $\mu_x=\alpha_x^\Z$ (recall \eqref{alpha}), and consider the sets
\begin{equation}
A_n := \chv{x \in (0,1) \colon E^{\mu_x}_0\crt{Z_n}= 0}, \qquad n\in\N.
\end{equation}
By the implicit function theorem, $x \mapsto \eta(x)$ is analytic. Therefore $A_n$ is finite (otherwise $x \mapsto \bb{E}^{\mu_x}_0\crt{Z_n}$ would be constant equal to $0$, which is not the case because $\lim_{x \uparrow 1} E^{\mu_x}_0\crt{Z_n} = n$ and $\lim_{x \downarrow 1} E^{\mu_x}_0\crt{Z_n} = -n$). Consequently, $A := \cup_{n\in\N} A_n$ is countable and hence $A^c:=(0,1) \setminus A \neq \emptyset$. Now \eqref{goalrec} follows because, for any $x \in A^c$,
%%
%\begin{equation} 
$E^{\mu_x}_0[Z_n] \neq 0 \quad \forall\, n \in \N$.
%\end{equation}
%%
\end{proof}

%%%
\subsection{Asymmetry in the Gaussian regime}
In this section, we prove Theorem~\ref{thm:RWREextra2}(II).

\begin{proof}
Fix $s \in (2,\infty)$. To show that
\begin{equation}
\label{goaltrans}
\Big\{ \alpha \colon \langle\log\rho\rangle < 0,\, \langle\rho^s\rangle =1,
\,E^\mu_0[Z_n] \neq v_\mu n \,\, \text{i.o.} \Big\} \neq \emptyset,
\end{equation}
we proceed as above. Define $\alpha_x$ as in~\ref{apprec}, but define $\eta(x)\in(0,1)$ to satisfy
\begin{equation}\label{etax0}
x\prt{\frac{1-x}{x}}^s + (1-x) \prt{\frac{1-\eta(x)}{\eta(x)}}^s = 1,
\end{equation}
which implies that $\eta_x$ satisfies $\langle \rho \rangle <0$. Let $\mu_x= \alpha_x^{\Z}$, and consider the sets
\begin{equation}\label{sBn}
B_n := \chv{x\in (0,1) \colon E^{\mu_x}_0 \crt{Z_n} = v_\mu n}, \qquad n \in \N.
\end{equation}
By the implicit function theorem, $x \mapsto \eta(x)$ is analytic. Consequently, $B := \cup_{n\in\N} B_n$ is countable and hence $B^c := (0,1) \setminus B \neq \emptyset$. Now \eqref{goaltrans} follows because, for any $x \in B^c$, 
%%
%\begin{equation} %\label{eSigma}
$E^{\mu_x}_0\crt{Z_n} \neq v_\mu n \quad \forall\, n \in \N$.
%\end{equation}
%%
\end{proof}

%%%%%%% APPENDIX C %%%%%%%%%%%%%%%%%%%%%%%%%%%%%%%%%%%%%%%%%%%
\section{Bound on recurrent fluctuations} \label{appC}

We prove Theorem~\ref{thm:RWREextra2}(III). The line of proof was suggested by Zhan Shi.

\begin{proof}
$\mbox{}$
Throughout this section, $C$ is a constant that does not depend on $n$ and may vary from line to line.

\paragraph{Scaled potential process}
Define 
%%
%\begin{equation}
$U^{\omega,n}(t) := \frac{1}{\sigma_0\log n}\, U^\omega({\lfloor t \log^2 n \rfloor})$,
%\end{equation}
%%
where
\begin{equation}
U^{\omega}(k) = \left\{
\begin{array}{ll}
\sum_{i=1}^k \log \rho_i, \qquad &k \in \N,\\[0.2cm]
0, \qquad &k = 0,\\[0.2cm]
- \sum_{i=k+1}^{0} \log \rho_i, \qquad &k \in - \N.
\end{array}
\right.
\end{equation}
From \eqref{alpha} and \eqref{uellcond} it follows that 
%\begin{equation}
$t \mapsto U^{\omega,n}(t)$
%\end{equation} 
converges weakly to a Brownian motion. Let $\overline{b}^n$ be the position of the bottom of the smallest valley $(\overline{a}^n, \, \overline{b}^n, \, \overline{c}^n)$ of the process $(U^{\omega,n}(t))_{t \in \R}$, which contains the origin and has depth larger than 1 (for a formall definition of the smallest valley see~\cite[Sec. 2.5]{ZZ04}).  Similarly, for any $\delta > 0$, let $(\overline{a}^n_\delta, \, \overline{b}^n_\delta, \, \overline{c}^n_\delta)$ be the smallest valley containing the origin with depth larger than $1 + \delta$. We start with the decomposition
\begin{equation}
\frac{Z_n}{\log^2 n} = \Big( \frac{Z_n}{\log^2 n} - \overline{b}^n \Big) + \overline{b}^n =: \bar{B}_n + \overline{b}^n.
\end{equation}
To control the left-hand side above, it suffices to show that for any $\gep>0$ there is a $C\in (0,\infty)$ such that
\begin{eqnarray}
\label{mbottom}
E_\mu\big[\overline{b}^n\big] &\leq& \frac{C}{\log^{\frac{2}{3} - \gep}n},\\
\label{mdbottom}
E^\mu_0\crt{\bar{B}_n} &\leq& \frac{C}{\log^{\frac{2}{3}- \gep}n}.
\end{eqnarray}

%%%

\paragraph{$\bullet$ Decay of $E_\mu\big[\overline{b}^n\big]$}
The proof of~\eqref{mbottom} is done via a Skorohod embedding. It is organised in three parts. In the first part we define the Skorohod embedding. In the second part, using the Skorohod embedding we compare the bottom of the valley $\overline{b}^n$ of the scaled potential process with the bottom of the valley $\hat{b}^n$ embedded potential process. In this part we use Kolmogorov's inequality combined with estimates on the random times that define the embedding. The third part consists of comparing the bottom of the embedded valley with the bottom of the underlying Brownian motion that we used for the embedding. This part relies on the control of the oscilations between the random times in the embedding together with the relation between conditioned Brownian motion and the Bessel bridge.

\paragraph{Skorohod embedding} 
Let $\prt{B_t}_{t\in \R}$ be a two sided Brownian motion with $B_0 := 0$ defined on the~probability~space $(\hat{\Omega}, \hat{\mc{F}},\chv{\hat{\mc{F}_t}}_{t\in \R},\hat{P})$, endowed with the double sided filtration generated by $\prt{B_t}_{t \in \R}$ starting from $0$, i.e., $\hat{\mc{F}}_t = \sigma(B_{\frac{st}{\abs{t}}},0\leq s \leq \abs{t})$. By the Skorokhod embedding \cite[Thm 7.6.3, p. 404]{Dur96} for each $n$, there is a sequence of stopping times, $(\hat{T}_{n,k})_{k \in \Z}$ with $\hat{T}_{n,0}=0$ and satisfying 
\begin{equation}\label{eqdistemb}
U^{\omega,n}\prt{\frac{k}{\log^2 n}}\overset{(d)}{=}B_{\hat{T}_{n,k}}.
\end{equation}
Let $t_{n,k}:= k \log^{-2}n$ denote the jump times of the scaled potential process. From now on  
\begin{equation}
\prt{\hat{U}^{\omega,n}(t))}_{t \in \R}
\end{equation}
refers to the \emph{embedded potential process} determined by $\prt{B_t}_{t \in \R}$ with jump times $\hat{T}_{n,k}$. We denote by $(\hat{a}^n, \hat{b}^n, \hat{c}^n)$ the smallest valley of the process $(\hat{U}^{\omega,n}(t)))_{t \in \R}$ that contains the origin and has depth larger than $1$. We write $\hat{E}$ to denote expectation w.r.t.\ the embedded random variables $\prt{\log \rho_i}_{i \in Z}$ that regulate the jumps of the scaled and embedded potential processes. 

Let $(\hat{a}, \hat{b}, \hat{c})$ be the smallest valley of depth $1$ containing the origin of the Brownian motion $\prt{B_t}_{t\geq 0}$. Note that the distribution $\hat{b}$ is given by~\eqref{densityofV} and by symmetry: 
\begin{equation}\label{Vzeromean}
\hat{E}[\hat{b}]=0.
\end{equation}
Note first that $\bar{b}^n \leq \overline{c}^n - \overline{a}^n$. As shown in~\cite[Appendix C]{AdH17}, the random variable $\bar{J}^n:=\overline{c}^n - \overline{a}^n$ satisfies 
%\begin{equation} \label{zhanshi}
$\sup_p E\big[\abs{\bar{J}^n}^p\big]<\infty$.
%\end{equation} 
Note next that 
\begin{equation}\label{Jnbound}
\max\chv{|\bar{b}^n|,|\hat{b}^n|, |\hat{b}|} \leq \bar{J}^n.
\end{equation}

The general idea to prove \eqref{mbottom} is to find sets $A_n$ for which
\begin{equation}\label{ideabounds}
\hat{E}[\bar{b}^n \Ind{A_n}] \leq \frac{C}{\log^{\frac{2}{3}-\gep}n},
\qquad \hat{P}[A^c_n] \leq \frac{C}{\log^{\frac{2}{3}-\frac{\gep}{2}}n }.
\end{equation} 
To obtain \eqref{mbottom}, we use H\"older's inequality to bound $\hat{E}[\bar{b}^n \Ind{A^c_n}]$ by
\begin{equation}\label{holderconc}
%\hat{E}[\bar{b}^n \Ind{A^c_n}] \leq 
 \hat{E}[\bar{J}^n\Ind{A^c_n}] 
\leq \hat{E}\big[\abs{\bar{J}^n}^p\big]^{\frac{1}{p}} \prt{\frac{1}{\log^{\frac{2}{3}-\frac{\gep}{2}}n}}^{\frac{p-1}{p}} \leq \frac{C}{\log^{\frac{2}{3} - \gep}n},
\end{equation}
where the last inequality follows by taking $p$ sufficiently large. More specifically, to prove \eqref{mbottom} will show that there are sets $A_n$ and $E_n$ for which
\begin{align}
&\hat{E}\crt{\abs{\bar{b}^n - \hat{b}^n} \Ind{A_n}} \leq \frac{C}{\log^{\frac{2}{3}-\gep}n}, 
&& \hat{P}(A_n^c)\leq \frac{C}{\log^{\frac{2}{3}-\frac{\gep}{2}}n},
\label{1claimdecay}\\
&\hat{E}\crt{\abs{\hat{b}^n - \hat{b}}\Ind{(A_n \cap E_n^c)}} \leq \frac{C}{\log^{\frac{2}{3}-\gep}n}, 
&& \hat{P}(A_n^c \cup E_n)\leq \frac{C}{\log^{\frac{2}{3}-\frac{\gep}{2}}n}.
\label{2claimdecay}
\end{align}
Reasoning as in~\eqref{ideabounds}--\eqref{holderconc}, using \eqref{Vzeromean}~\eqref{Jnbound}, \eqref{1claimdecay} and \eqref{2claimdecay}, we obtain
\begin{equation}\label{mainbound}
\begin{aligned}
\hat{E}\big[\overline{b}^n\big] %&= \hat{E}\big[\overline{b}^n - \hat{b}^n+ \hat{b}^n - \hat{b}\big]
%\\
&\leq \hat{E}\left[|\overline{b}^n - \hat{b}^n|\right]+ \hat{E}\left[|\hat{b}^n - \hat{b}|\right]
\leq  \frac{C}{\log^{\frac{2}{3}-\gep} n}.
\end{aligned}
\end{equation}

In the next two paragraphs we will show~\eqref{1claimdecay} by comparing the deterministic times $t_{n,k}$ with the random times $\hat{T}_{n,k}$ with the help of moment estimates. After that we will show \eqref{2claimdecay} by comparing the location of the embedded minimum $\hat{b}^n$ with the location of the true minimum $\hat{b}$ with the help of estimates on Bessel bridges.

\paragraph{Comparing $t_{n,k}$ with $\hat{T}_{n,k}$}
Let $a^ n := \bar{a}^ n\log^ 2 n $, $b^ n := \bar{b}^ n \log^ 2 n $ and $c^ n := \bar{c}^ n \log^ 2 n $ and let  $J(n) := {c}^n - {a}^n$.  The times $(\hat{\tau}_{n,k}:=\hat{T}_{n,k}-\hat{T}_{n,k-1})_{k \in \Z}$ defined by the Skorokhod embedding theorem stated in~\cite[Thm 7.6.3]{Dur96} are i.i.d.\ and satisfy
\begin{equation}\label{stopmoments}
\begin{aligned}
&\hat{E}\crt{\hat{\tau}_{n,k}} =  \hat{E}\crt{\prt{\frac{\log \rho_0}{\sigma_0\log n}}^2} = \frac{1}{\log^2 n},\\
&\hat{E}\crt{\hat{\tau}_{n,k}^2} \leq C \hat{E} \crt{\prt{\frac{\log \rho_0}{\sigma_0\log n}}^4}
< \frac{C}{\log^4n}.
\end{aligned}
\end{equation}
Furthermore, since $B^{2k} - p_k(t)$ is a martingale for some polynomial $p_k(t)$ of degree $k$, the optional stopping theorem and~\eqref{uellcond} give
\begin{equation}\label{2kmom}
\hat{E}\crt{\prt{\hat{\tau}_{n,k}}^{k}} 
\leq C \hat{E}\crt{B_{\hat{\tau}_{n,k}}^{2k}} 
= C \hat{E}\crt{\prt{\frac{\log \rho_0}{\sigma_0\log n}}^{2k}} 
\leq  \frac{C}{\log^{2k}n}.
\end{equation}
Therefore, by Markov's inequality, for any $k \in \N$,
\begin{equation}\label{stopbound}
\hat{P}\prt{\hat{\tau}_{n,k}>2\frac{\sigma_\mu}{\log^{2-\gep} n} } 
\leq \frac{\prt{\log  n}^{k( 2-  \gep)}}{\sigma_\mu^2} \frac{C}{\log^{2k}n} 
\leq \frac{C}{\log^{k \gep} n}.
\end{equation}
For
%\begin{equation}\label{choosek}
$k \gep-2>2 + 2\gep$,
%\end{equation}
and any fixed $J_0 >0$, a union bound gives that
\begin{equation}\label{stopunion}
\begin{aligned}
&\hat{P}\prt{\exists \, k \leq J(n) \colon \hat{\tau}_{n,k}>2\frac{\sigma_\mu}{\log^{2-\gep}n },\,\, \frac{J(n)}{\log^2n} \leq J_0 }
\leq \frac{C}{\log^{2 + 2 \gep} n}.
\end{aligned}
\end{equation}
Abbreviate $\bar{J}^n := J(n) \log^{-2}n$ and define the set
\begin{equation}\label{goodset}
A_{n} := \left\{\omega \colon \sup_{k\leq J(n)}\hat{\tau}_{n,k}<2\frac{\sigma_\mu}{\log^{2-\gep} n},
\bar{J}(n) \leq  (\log\log^4 n) \right\}.
\end{equation}
We have
\begin{equation}\label{Jngrowth}
\hat{P}\prt{\bar{J}_n > \log\log^4 n} \leq c_1\hat{P}\prt{\sup_{t \in [0,\log\log^4 n]} \abs{B_t}<1}\leq \frac{C}{\log^4 n},
\end{equation}
where $c_1$ stands for a constant that takes into account the double-sided necessary estimates to the right and to the left of the origin. Furthermore, the constant $c_1$ also absorbs the uniform approximation error of the discrete walk, with respect to the Brownian motion. From~\eqref{stopunion} and~\eqref{Jngrowth} it follows that
\begin{equation}\label{meshdecay}
 \hat{P}\prt{A^c_{n}} \leq \frac{C}{\log^{2 + 2\gep}n}.
\end{equation}
Therefore, on $A_n$, using that $\prt{\hat{\tau}_{n,k} - \log^{-2} n}_{k \in \Z}$ is a sequence of i.i.d.\ mean zero random variables, by Kolmogorov's inequality  and~\eqref{stopmoments} it follows that, for any $\gep >0$,
\begin{equation}\label{supdecay}
\begin{aligned}
&\hat{P} \prt{\sup_{a^n \leq j \leq c^n} t_{n,j} - \hat{T}_{n,j}> \frac{1}{\log n},\,\, A_{n}} \\
&\quad \leq \hat{P}\crt{\sup_{j \leq \log \log^4 n}\sum_{k = 0}^j\hat{\tau}_{n,k} - \log^{-2} n  > \frac{1}{\log n}}\\
&\quad \leq (\log^2n)\hat{E}\crt{\prt{\sum_{k = 0}^{\log (\log^4n)}{\hat{\tau}_{n,k} - \log^{-2} n}}^2}\\
&\quad \leq (\log\log^4n) \log^2 n\, \frac{C}{\log^4 n} \leq \frac{C}{ \log^{2 - \frac{\gep}{2}} n}.
\end{aligned}
\end{equation}
Let 
\begin{equation}
A_{n,\leq } := \left\{\omega \colon \sup_{a^n \leq j \leq c^n} t_{n,j} - \hat{T}_{n,j} \leq  \frac{1}{\log n}\right\}.
\end{equation}
Since $\hat{b}_n  = \hat{T}_{n,b^n}$, by~\eqref{meshdecay} and~\eqref{supdecay}, and arguing as in~\eqref{ideabounds}--\eqref{holderconc}, we get that
\begin{equation}\label{disxembed}
\begin{aligned}
&\hat{E}\crt{\abs{\overline{b}^n - \hat{b}^n }} \leq \hat{E}\crt{\abs{\overline{b}^n - \hat{b}^n }\Ind{A_n}} + \hat{E}\crt{\abs{\overline{b}^n - \hat{b}^n }\Ind{A^c_n}}\\
&\quad \leq \frac{1}{\log n} + \hat{E}\crt{\abs{\bar{J}^n }\Ind{A_{n,\leq}^c}\Ind{A_n}} 
+ \hat{E}\crt{\abs{\bar{J}^n}\Ind{A^c_n}}
 \leq \frac{C}{\log^{\frac{2}{3}- \gep}n}.
\end{aligned}
\end{equation}

\paragraph{Comparing $\hat{b}^n$ with $\hat{b}$}

To prove \eqref{mbottom} it suffices to show that 
\begin{equation}\label{bottomdist}
\hat{E}\left[|\hat{b}_n - \hat{b}| \Ind{A_{n}}\right]  \leq\frac{C}{\log^{\frac{2}{3}-\gep} n}.
\end{equation}
To prove~\eqref{bottomdist} we first note that, conditioned on $\hat{b}$ being the bottom of the valley $(\hat{a},\hat{b},\hat{c})$ of depth 1, the trajectory of the Brownian motion $B_t$ behaves as a two-sided Bessel bridge of dimension $3$ (see~\cite{Pit75}). Taking the point $(\hat{b}, B_{\hat{b}})$ to be the origin, we see that the two bridges we are considering can be described by
\begin{align}
dX_t &=  \left(\frac{1}{X_t} +\frac{1}{1-X_t} + \frac{X_t-1}{t-(\hat{c}-\hat{b})}\right)\,dt 
+ dB_t \quad \text{ for } t \leq \hat{c}-\hat{b},\label{rightbessel} \\ 
dX_t &=  \prt{\frac{1}{X_t} +\frac{1}{1-X_t} + \frac{X_t-1}{t-(\hat{b}-\hat{a})}}\,dt 
+ dB_t\quad \text{ for } t \leq \hat{b}-\hat{a}.
\label{leftbessel}
\end{align}
By the symmetry of Brownian motion, it suffices to analyse \eqref{rightbessel}, which, in integral form, for $t \leq (\hat{c}-\hat{b})$ reads as
\begin{equation}\label{besselint}
Y_t =  \int_0^t\prt{\frac{1}{X_s} +\frac{1}{X_s-1} + \frac{1-X_s}{s-(\hat{c}-\hat{b})}} ds +  B_t. 
\end{equation}
By the invariance  of Brownian motion, 
%\begin{equation}\label{eqdistbm}
$B_1 \overset{(d)}{=}-B_1 \overset{(d)}{=}\frac{B_t}{\sqrt{t}}$.
%\end{equation}
Then, by~\eqref{besselint}, for any $\delta>0$ and $\eta \in (0,\tfrac14)$ we get
\begin{equation}\label{escapeb}
\begin{aligned}
\hat{P}(\sup_{t \in [0,\delta]}Y_t< \eta)  
&\leq  \hat{P} \left(\frac{1}{2}\eta^{-1} \delta+ B_{\delta}<\eta\right)
\\
&=\hat{P} \left(\frac{1}{2}\eta^{-1}\delta^{\frac{1}{2}} - \eta\delta^{-\frac{1}{2}}< B_1\right).
\end{aligned}
\end{equation}
Next, by taking $\delta_n = \log^{-\alpha} n$, $\eta_n = \log^{-\beta}n<\tfrac14$ with $\beta = \tfrac13 - \tfrac14\gep$ and $\alpha = \tfrac23 -\gep$, with $\gep>0$ sufficiently small, we get
\begin{equation}\label{eventually}
\begin{aligned}
\hat{P}\left(\exists\,  t \leq \delta_n \text{ with } Y_t > \eta_n\right)&
=1 - \hat{P}(\sup_{t \in [0,\delta_n]}Y_t< \eta_n)  \\
&\geq 1 - \hat{P}\left(\frac{1}{2}\log^{\frac{\gep}{2}}n - \log^{-\frac{\gep}{2}}n < B_1\right)\\
&\geq 1 - \frac{1}{\exp(c\log^{\gep}n)},
\end{aligned}
\end{equation}
for some $c>0$. Now let 
%\begin{equation}\label{Besmall}
$B(\delta_n,\eta_n) := \chv{\omega \colon\, \exists\,  t \leq \delta_n \text{ with } Y_t > \eta_n}$
%\end{equation}
and note that
\begin{equation}\label{Beleaves}
\hat{P}\prt{B(\delta_n,\eta_n)} \leq \frac{1}{\exp(c\log^{\gep}n)}.
\end{equation}
By the construction of the Skorohod embedding and by \eqref{uellcond}, 
\begin{equation}\label{accuracy}
\sup_{k\in\N}\sup_{t\in [T_{n,k-1},T_{n,k}]}\abs{B_t - B_{\hat{T}_{n,k-1}}}
\leq \log{\frac{1-\mathfrak{c}}{\mathfrak{c}} } \frac{1}{\log n}.
\end{equation}
So, on the event $B(\delta_n,\eta_n)$, $|\hat{b}^n - \hat{b}| > \frac{1}{\log^{\frac{2}{3} -\gep}n} $ implies
\begin{equation}\label{godown}
\inf_{t \leq \hat{c}-\hat{b}}Y_t < \frac{1}{\log^{1 - \gep}n} \text{ when } Y_0  = \frac{1}{\log^{\frac{1}{3}-\frac{\gep}{4}}n}.
\end{equation}
Let
\begin{equation}\label{eventdown}
E_n := \left\{\omega \colon \inf_{t \leq \hat{c}-\hat{b}}Y_t < \frac{1}{\log^{1 - \gep}n} 
\text{ given } Y_0  = \frac{1}{\log^{\frac{1}{3} - \frac{\gep}{4}}n}\right\}.
\end{equation} 
From the hitting times for Bessel processes~\cite[Problem 3.3.23, p.162]{KarShr98} it follows that 
\begin{equation}\label{besselhit}
\hat{P}\prt{E_n} \leq \frac{C}{\log^{\frac{2}{3} - \frac{3\gep}{4}}n}.
\end{equation}
Noting that $\bar{J}_n \leq \log \log^4 n$ on $A_n$ and using \eqref{Jnbound} and \eqref{godown}, we get
\begin{equation}\label{botsplit}
\begin{aligned}
&\hat{E}\crt{|\hat{b}_n - \hat{b}|\,\Ind{A_{n}}} \\
&\quad \leq \frac{1}{\log^{\frac{2}{3}-\gep} n } + \hat{E}\crt{|\hat{b}_n - \hat{b}|\, 
\Ind{A_n} \Ind{\left\{|\hat{b}^n - \hat{b}| > \log^{-(\frac{2}{3}-\gep)} n\right\}}}\\
&\quad \leq \frac{1}{\log^{\frac{2}{3}-\gep} n } + \log\log^4 n \prt{\hat{E}\crt{1_{E_n}} 
+ \hat{E}\crt{\Ind{\left(B(\delta_n, \eta_n)\right)^c}}  }\\
&\quad \leq \frac{C}{\log^{\frac{2}{3} - \gep}n},
\end{aligned}
\end{equation}
where the last inequality uses~\eqref{Beleaves} and \eqref{besselhit}. Using~\eqref{botsplit},~\eqref{disxembed}, and~\eqref{mainbound}, we conclude the proof of \eqref{mbottom}.
 	
%%%

\paragraph{$\bullet$ Decay of $E_0^\mu[\bar{B}_n]$}
	
It remains to show~\eqref{mdbottom}. We follow~\cite[pp. 249--251]{ZZ04}, with appropriate modifications. We define the set of ``good environments", with $\delta$ and $J$ both depending on $n$, as
\begin{equation}
A_n^{J,\delta} := \left\{ \omega \colon
\begin{array}{lll}
\overline{b}^n = \overline{b}^n_\delta,\\
\text{any refinement} (a, b, c) \text{ of }(\overline{a}^n_\delta, \overline{b}^n, \overline{c}^n_\delta) \\
\qquad\qquad\text{with }b \neq \overline{b}^n\text{ has depth }< 1-\delta,\\
|\overline{a}^n_\delta| + |\overline{c}^n_\delta| \leq J, \\
\inf_{t-\bar{b}^n>\delta} B_t - B_{\hat{b}}> \delta^{\frac{3}{2}},
\end{array}
\right\}
\end{equation}
with $\delta$ and $J$ chosen as
\begin{equation}\label{choice}
\delta = \delta(n) := \frac{1}{\log^r n}, \qquad J  = J(n):= \log \log^4 n,
\end{equation}
with $r\in (0,1)$ a parameter to be fixed later. From now on, we simply write $P$ and $E$ for the annealed measure and corresponding expectation, as well as for the measure of the underlying Brownian motion that was used for the embedding in the previous paragraph and its corresponding expectation.

Recall that $\bar{B}_n = \frac{Z_n - b^n}{\log^2 n}$. Let 
\begin{equation} \label{hitac}
G_n := \{\omega \colon (X_i), 0 \leq i \leq n \text{ hits the boundary of } [a^n_\delta, c^n_\delta] \}.
\end{equation}
With these definitions, we split $E[|\bar{B}_n|]$ as 
\begin{equation}
\begin{aligned}
&E[|\bar{B}_n|] = E \left[ |\bar{B}_n| \Ind{(A_n^{J,\delta})^c}  \right] + E \left[ |\bar{B}_n| \Ind{A_n^{J,\delta}} \right]\\
&\quad = {\rm I}_n+  E \left[ |\bar{B}_n| \Ind{A_n^{J,\delta}} \Ind{\{\overline{b}^n < 0\}} \right]  +E\left[ |\bar{B}_n| \Ind{A_n^{J,\delta}} \Ind{\{\overline{b}^n > 0\}} \right]\\
&\quad = {\rm I}_n + {\rm II}_n +E\left[ |\bar{B}_n| \Ind{A_n^{J,\delta}} \Ind{\{\overline{b}^n > 0\}} \Ind{G_n^c} \right] \\
&\qquad\qquad\qquad\qquad+ E\left[ |\bar{B}_n| \Ind{A_n^{J,\delta}} \Ind{\{\overline{b}^n > 0\}} \Ind{G_n} \right]\\
&\quad = {\rm I}_n + {\rm II}_n + {\rm III}_n + {\rm IV}_n,
\end{aligned}
\end{equation}
where
\begin{equation}
\begin{array}{lr}
{\rm I_n}  := E \left[ |\bar{B}_n| \Ind{(A_n^{J,\delta})^c}  \right],&
{\rm III}_n := E\left[ |\bar{B}_n| \Ind{A_n^{J,\delta}} \Ind{\{\overline{b}^n > 0\}} \Ind{G_n^c} \right], \\[0.2cm]
{\rm II_n} := E \left[ |\bar{B}_n| \Ind{A_n^{J,\delta}} \Ind{\{\overline{b}^n < 0\}} \right],
&{\rm IV}_n  := E\left[ |\bar{B}_n| \Ind{A_n^{J,\delta}} \Ind{\{\overline{b}^n > 0\}}\Ind{G_n} \right].
\end{array}
\end{equation}
To prove \eqref{mdbottom}, it suffices to show that there is a constant $C>0$ for which
\begin{equation}\label{Isuffice}
\begin{aligned}
\max\chv{{\rm I}_n, {\rm II}_n ,{\rm III}_n ,{\rm IV}_n} \leq \frac{C}{\log^{\frac{2}{3} -\gep}n}.
\end{aligned}
\end{equation}	
In what follows we will show that this bound holds for each of the above terms.

%%%

\paragraph{Estimate of ${\rm I}_n$}
The estimate of ${\rm IV}_n$ follows directly from the definition of $A_n^{J,\delta}$.  By H\"older's inequality, for $p, q > 1$ with $\tfrac{1}{p} + \tfrac{1}{q} = 1$,
\begin{equation}\label{ivhol}
{\rm IV}_n \leq E\left[ |\bar{B}_n|^p \right]^{\tfrac{1}{p}} P((A_n^{J,\delta})^c)^{\tfrac{1}{q}}.
\end{equation}
Since $\sup_{n\in\N} E\left[ |\bar{B}_n|^p \right] < \infty$ (see~\cite{AdH17}), it suffices to estimate 
%\begin{equation}
$P((A_n^{J,\delta})^c)$.
%\end{equation}
The definition of $A_n^{J,\delta}$ consists of four conditions. Therefore we estimate
\begin{equation}\label{vconditions}
\begin{aligned}
P((A_n^{J,\delta})^c)
&\leq P(\bar{b}^n \neq \bar{b}^n_\delta)  \\
&\quad +P( \exists \text{  refinement }(a, b, c) \text{ of }(\overline{a}^n_\delta, \overline{b}^n, \overline{c}^n_\delta) \\
&\qquad\qquad\qquad\text{ with }  b \neq \overline{b}^n \text{ and depth }> 1-\delta)\\
&\quad + P\prt{|\overline{a}^n_\delta| + |\overline{c}^n_\delta| > J}\\
&\quad +  P ( \inf_{t-\bar{b}^n>\delta} B_t - B_{\hat{b}}< \delta^{\frac{3}{2}}).
\end{aligned}
\end{equation}
Note that
\begin{equation}\label{valleybessel}
\begin{aligned}
\chv{\bar{b}^n \neq \bar{b}^n_\delta}&\subset \Big\{\exists\text{  refinement }(a, b, c) 
\text{ of }(\overline{a}^n_\delta, \overline{b}^n, \overline{c}^n_\delta) \\
&\quad\qquad\qquad\text{ with } 
b \neq \overline{b}^n \text{ and depth }> 1-\delta\Big\}.
\end{aligned}
\end{equation}
Furthermore, the probability of having a valley of depth larger than $1-\delta$ is bounded from above by the probability for a Bessel bridge starting from $1$ to reach a value smaller than $\delta$, which in turn is bounded from above by the probability for the infimum of a Bessel process of dimension $3$ starting from $1$ to be smaller than $\delta$. By the estimate for hitting times of Bessel process, it follows that
\begin{equation} \label{valleybound}
\begin{aligned}
&P\left(\exists\text{ refinement }(a, b, c) \text{ of } (\overline{a}^n_\delta, \overline{b}^n, 
\overline{c}^n_\delta) \right.\\
&\qquad\qquad\left.\text{ with } b \neq \overline{b}^n \text{ and depth }> 1-\delta\right) \\
&\quad \leq 2P(\text{ Bessel process of dimension } 3 \\
&\qquad \qquad  \text{ started from 1 reaches a value smaller than } \delta) \\ 
&\quad \leq 2  \delta \leq \frac{2}{\log^r n},
\end{aligned}
\end{equation}
where the factor $2$  takes into account the double-sided necessary estimates (to the right and to the left of $\bar{b}^n$). Combining \eqref{valleybessel} with \eqref{valleybound}, we get  
\begin{equation}\label{midway}
P((A_n^{J,\delta})^c)
\leq \frac{C}{\log^r n}
+ P\prt{|\overline{a}^n_\delta| + |\overline{c}^n_\delta| > J}
+  P (\inf_{t-\bar{b}^n>\delta} B_t - B_{\hat{b}}< \delta^{\frac{3}{2}}).
\end{equation}
To estimate the remaining terms in  \eqref{vconditions}, we first note that, by \eqref{Jngrowth},
\begin{equation} \label{valleylenght}
\begin{aligned}
&P(\abs{\bar{a}^n_\delta} + \abs{\bar{c}^n_\delta}>J ) \leq \frac{C}{\log^4 n}.
\end{aligned}
\end{equation}
The last term in \eqref{vconditions} can be bounded via the same reasoning used in \eqref{Beleaves} and~\eqref{besselhit}, and so we get that
\begin{equation}\label{bessel}
P\left(\inf_{t-\bar{b}^n>\delta} B_t - B_{\hat{b}}> \delta^{\frac{3}{2}}\right) \leq \frac{C}{\log^{r}n}.
\end{equation}
Therefore, with $r =\tfrac{2}{3}$ in \eqref{choice} it follows from~\eqref{midway},~\eqref{valleylenght} and~\eqref{bessel} that there is a choice of $p,q$ in \eqref{ivhol} such that
%\begin{equation}\label{AdJ}
${\rm I}_n  \leq \frac{C}{\log^{\frac{2}{3}-\gep} n}$.
%\end{equation}

%%%

\paragraph{Estimate of ${\rm II}_n$}
The estimate of ${\rm II}_n$  is analogous to  ${\rm III}_n + {\rm IV}_n$.
	
%%%

\paragraph{Estimate of ${\rm III}_n$}
Before proving the this estimate, we recall the expression for the hitting times as stated in~\cite[p.196 (2.1.4)]{ZZ04}): for $a < x < b$,
\begin{equation} \label{hitting}
\begin{aligned}
P^\omega_x(H_a < H_b) = \frac{\sum_{i=x}^{b-1} \exp{U^\omega(i)}}
{\sum_{i=a}^{b-1} \exp{U^\omega(i)}}, \\[0.2cm]
P^\omega_x (H_b < H_a) = \frac{\sum_{i=a}^{x-1} \exp{U^\omega(i)}}{\sum_{i=a}^{b-1} \exp{U^\omega(i)}},
\end{aligned}
\end{equation}
where, for any $y \in \Z$,
%%
%\begin{equation}\label{Hdef}
$H_y := \inf\{ i \in \N_0 \colon  Z_i = y \}$.
%\end{equation}
%%
On the event $E^c_n \cap A^{J,\delta}_{n} \cap \{\bar{b}^n >0\}$ the random walk $\prt{Z_t}_{t \in \N_0}$ is equivalent to the reflecting random walk  at $a^n$  denoted by $(\tilde{Z}_t)_{t \in \N_0}$. More formally $\tilde{Z}_t$ is the random walk in the environment $\overline{\omega}_z := \omega_z$ for $z > a_\delta^n$, $\overline{\omega}^{\, +}_{a_\delta^n} = 1$ and $\overline{\omega}^{\, +}_{a_\delta^n-1} = 0$. Therefore, for $\omega \in A^{J,\delta}_n$,
\begin{equation} \label{pf_ineg1}
\begin{aligned}
&E^\omega_{0} \left( \left| \frac{Z_t}{\log^2 n} - \overline{b}^n \right|\,,\, G^c_n \right)\\
&\quad \leq \log \log^4 n \,P^\omega_0(H_{b^n}>n)
+ E_{0}^\omega \left( \left| \frac{\tilde{Z}_t}{\log^2 n} - \overline{b}^n \right|\Ind{\chv{T_{b^n} <n}}\right) \\
&\quad \leq \log\log^4 n \,P^\omega_0(T_{b^n}>n)%\quad 
+ \max_{t\in [0, \, n] \cap \Z} E_{b^n}^\omega \left( \left| \frac{\tilde{Z}_t}{\log^2 n} - \overline{b}^n \right|\right).
\end{aligned}
\end{equation}
The arguments that lead to \cite[Eqs.\ (2.4.4)--2.5.5, pp. 249-250]{ZZ04}, imply that 
\begin{equation}\label{missbottom}
P^\omega_0(H_{b^n}>n) \leq \frac{C}{\exp(2^{-1}\delta_n\log n) } \leq \frac{C}{\exp\prt{2^{-1}\log^{1 - r }n}}.
\end{equation}
Therefore, for any $r<1$ there is a $C>0$ for which
\begin{equation}\label{missdecay}
(\log \log^4 n)  \,P^\omega_0(T_{b^n}>n)\leq \frac{C}{\log^{\frac{2}{3} - \gep}n}.
\end{equation}
To estimate the second term in the right-hand side of \eqref{pf_ineg1}, we follow~\cite[pp.250--251]{ZZ04}. Define
\begin{equation}\label{inv}
\begin{aligned}
f(z) &:= \frac{\prod_{a_\delta^n +1 \leq i < z} \omega_i}{\prod_{a_\delta^n +1 \leq i < z} (1 -\omega_{i+1})} 
= \frac{(1-\omega_{a_\delta^n +1})}{\omega_z} \, n^{-[U^{\omega,n}(z)-U^{\omega,n}(a_\delta^n)]}, \\[0.2cm]
\overline{f}(z) &:= \frac{f(z)}{f(b^n)}.
\end{aligned}
\end{equation}
For $g\colon\,\Z \to \R$, let 
%\begin{equation}
$\nu_g= \sum_{z \in \Z} \delta_{z} g(z)$,
%\end{equation}
where $\delta_{z}$ is the Dirac measure concentrated at $z$. The one-step transition operator of this process $\mc{A}$ acts on a measures on $\Z$ as follows: 
\begin{equation}\label{action}
\prt{\nu\mc{A}}(z) := \overline{\omega}_{z-1} \, \nu(z-1) + (1 - \overline{\omega}_{z+1}) \, \nu(z+1).
\end{equation}
Note that, by~\eqref{inv} and \eqref{action}, $\nu_{\overline{f}}\mc{A} = \nu_{\overline{f}}$. In words, $\nu_{\overline{f}}$ is an invariant measure for the reflecting random walk $(\overline{Z}_t)_{t \geq 0}$. Since $\overline{f}(z) \geq \Ind{b^n}(z)$ for all $z$, and $g\mathcal{A}  \geq 0$ for all $g \geq 0$,  we obtain that
\begin{equation}\label{invbound}
\begin{aligned}
P^\omega_{b^n} (\overline{Z}_t = z) &= \nu_{\Ind{b^n}} \mathcal{A}^t(z) \, 
\leq \nu_{\overline{f}} \mathcal{A}^t(z) = f(z)\\
&=\frac{\omega^+_{b^n}}{\omega^+_z} \, n^{-[U^{\omega,n}(z) - V(b^n)]}
\leq \frac{1}{\mathfrak{c}} \, n^{-[U^{\omega,n}(z) - U^{\omega,n}(b^n)]},
\end{aligned}
\end{equation}
the last inequality being a consequence of the uniform ellipticity assumption ($\omega_0^+ \geq \mathfrak{c}$).
Note now that, for $\omega \in A^{J,\delta}_n$,
\begin{equation}
\abs{z-b^n}>\delta_n \Longrightarrow U^{\omega,n}(z) -U^{\omega,n}(\bar{b}_n)\geq \delta^{\frac{3}{2}}_n = \frac{1}{\log^{\frac{3r}{2}}n}.
\end{equation} 
Hence, uniformly in all $t \geq \Z_+$, for any $r<\tfrac{2}{3}$ there is a  $C>0$ for which
\begin{equation} \label{II}
\begin{aligned}
& E^\omega_{b^n} \left( \abs{ \frac{\overline{Z}_t}{\log^2 n} - \overline{b}^n } \right) 
= \sum_{z\in[a_\delta^n,c_\delta^n]\cap\Z} P^\omega_{b^n} (\overline{Z}_t = z) \, 
\abs{ \frac{z}{\log^2 n} - \overline{b}^n }\\
&\quad \leq \frac{1}{\mathfrak{c} \, \log^2 n} \sum_{z\in[a_\delta^n,c_\delta^n]\cap\Z} \, 
|z - b^n| \, n^{-[U^{\omega,n}(z) - U^{\omega,n}(b^n)]}\\
&\quad \leq  2\delta_n + \frac{(J \log^2n)^2}{\mathfrak{c}\log^2n} \sup_{z - b^n > \delta \log^2 n}e^{-\log n[V(z)-v(b^n)]} \\
&\quad \leq \frac{2}{\log^rn} + C\log^2 n \prt{\log \log^4 n} e^{- \log^{(1-\frac{3r}{2})}n} \leq \frac{C}{\log^r n},
\end{aligned}
\end{equation}
which yields the bound for ${\rm III}_n$.
	
%%%

\paragraph{Estimate of ${\rm IV}_n$}
 
By H\"older's inequality, for $p, q > 1$ with $\tfrac{1}{p} + \tfrac{1}{q} = 1$,
\begin{equation}\label{HolderIn}
\begin{aligned}
{\rm I}_n &\leq E \left[ |\bar{B}_n| \Ind{A_n^{J,\delta}}\Ind{\chv{\bar{b}^n >0}}\Ind{G_n} \right] \\
&\leq E\left[ |\bar{B}_n|^p \right]^{\tfrac{1}{p}} \bb{P} (\chv{\bar{b}^n >0} \cap A_n^{J,\delta}\cap G_n)^{\frac{1}{q}}.
\end{aligned}
\end{equation}
As $\sup_n E\left[ |\bar{B}_n|^p \right] < \infty$, we estimate $\bb{P}(\chv{\bar{b}^n >0} \cap A_n^{J,\delta}\cap G_n)$. Define
\begin{equation}
\begin{aligned}
H_{b,n}&:= \inf \{ i \geq 0 \colon \, Z_i = b^n \},\\
H_{a,b,n} &:= \inf \{ i \geq 0 \colon \, Z_i = b^n \text{ or } Z_i = a^n_\delta \}.
\end{aligned}
\end{equation}
Then, by~\eqref{hitting}, we have (this is the same inequality as in~\cite[Eq.\ (2.5.4)]{ZZ04})
\begin{equation}\label{hita}
P^\omega_0 (Z_{H_{a,b,n}} = a^n_\delta)
\leq \frac{J \log^2 n}{n^\delta}.
\end{equation}
Again, denote by $(\overline{Z}_t)_{t\geq 0}$ the random walk in random environment with a reflecting barrier at $a^n_\delta$, and let $\overline{H}_{a,b,n}$ be the analogue of $H_{a,b,n}$ for $(\overline{Z}_t)_{t \geq 0}$, and $\overline{H}_{a,n}$ be the hitting time of $b^n$ by the reflecting walk. Then, by \eqref{uellcond},
\begin{equation} \label{Hexp}
\begin{aligned}
E^\omega_0 \left[ H_{a,b,n} \right] &%= E^\omega_0 \left[ \overline{H}_{a,b,n} \right] 
\leq E^\omega_0 \left[ \overline{H}_{b,n} \right] = \sum_{i=1}^{b^n} 
\sum_{j=0}^{i-1-a^n_\delta} \frac{1}{\omega_{i-j-1}} \prod_{k=1}^j \rho(i-k)\\
&\leq \frac{1}{\mathfrak{c}} \sum_{i=1}^{b^n} \sum_{j=0}^{i-1-a^n_\delta} 
\exp^{(U^{\omega,n}(i)-U^{\omega,n} (i-j)) \log n}\\
& \leq \frac{ \prt{2J \log^2 n}^2}{\mathfrak{c}} \, \exp^{(1 - \delta) \log n},
\end{aligned}
\end{equation}
see~\cite[p.250]{ZZ04}. Consequently, for $\omega \in A_n^{J,\delta}$ satisfying $\overline{b}^n > 0$, by \eqref{hita} and \eqref{Hexp} and Markov' s inequality, we obtain
\begin{equation} \label{2.5.5}
\begin{aligned}
 & P^\omega (H_{b,n} \geq n) \leq P^\omega(H_{a,b,n}< n,Z_{H_{a,b,n}}=a^n_\delta) 
 + P^\omega(H_{a,b,n}\geq n)  \\
 & \frac{J \log^2 n}{n^\delta} + \frac{1}{n} \frac{ 2 (J \log n)^2}{\mathfrak{c}} \, 
 e^{(1-\delta)(\log n)} = \frac{J \log^2 n + \frac{2 (J \log n)^2}{\mathfrak{c}}}{n^\delta}.
\end{aligned}
\end{equation}
This is the analogue of~\cite[Eq.\ (2.5.5)]{ZZ04} and says that, with overwhelming probability, the random walk hits $b^n$ before time $n$. Let us now argue that, with overwhelming probability, after hitting $b^n$, the random walk will come back to $b^n$ before hitting either $a_\delta^n$ or $c_\delta^n$. By~\eqref{hitting}, for all  $\omega\in A_n^{J,\delta}$, 
\begin{equation}
\begin{aligned}
P^\omega_{b^n -1} \left(\text{$(Z_i)_{i\in\N_0}$ hits $a^n_\delta$ before $b^n$}\right)
&\leq n^{- (1 + \frac{\delta}{2})}, \\
P^\omega_{b^n+1} \left(\text{$(Z_i)_{i\in\N_0}$ hits $c^n_\delta$ before $b^n$}\right) 
&\leq n^{- (1 + \frac{\delta}{2})}.   
\end{aligned}
\end{equation}
Compare with~\cite[Eq.(2.5.6)]{ZZ04}.]  As such, for $\omega \in A_n^{J,\delta}$  the $P^\omega$-probability of the event that ``after hitting $b^n$, the random walk exits $[a_\delta^n, \, c_\delta^n]$ within the next $n$ steps'' is bounded by
\begin{equation}
1 - \left( 1 - n^{- (1 + \frac{\delta}{2})}\right)^n \leq \frac{C}{n^{\frac{\delta}{2}}},
\end{equation}
%%
%for some $C>0$. 
Combined with~\eqref{2.5.5}, this gives
\begin{equation} \label{En}
P^\omega(A_n^{J,\delta} \cap G_n) \leq \frac{J \log^2 n + \frac{2 (J \log n)^2}{\mathfrak{c}}}{n^\delta} 
+ \frac{2}{n^{\frac{\delta}{2}}} \leq \frac{C}{n^{\frac{\delta}{2}}}.
\end{equation}
Applying H{\"o}lder's inequality and the $L^p$ bound $\sup_{p,n\in\N} E[\abs{\bar{B}_n}^p]<\infty$, we find that
${\rm IV}_n \leq \frac{C}{n^{\tfrac{\delta}{3}}} \leq \frac{C}{\log^{4}n}$,
which proves the desired estimate in~\eqref{mdbottom}.
\end{proof}

%%%%%%%%%%%%%%%%%%%%%%%%%%%%%%%%%%%%%%%%%%%%%%%%%%

%\bibliographystyle{plain}
%\bibliography{RWCRE3}

\end{document}